\documentclass[oneside,11pt]{article}
\usepackage[utf8]{inputenc}
\usepackage[T1]{fontenc}
\usepackage{amsmath,amsfonts,amssymb,amsthm,amscd,tipa,array,longtable}
\usepackage[english]{babel}
\usepackage{fancyhdr}
\usepackage{indentfirst}
\usepackage{color}
\usepackage{graphicx}
\usepackage[a4paper,top=4.5cm,bottom=3.5cm]{geometry}
\usepackage[all]{xy}
\usepackage{hyperref}
\hypersetup{hidelinks}
\renewcommand{\epsilon}{\varepsilon}

\renewcommand{\P}{\mathbb{P}}
\newcommand{\C}{\mathbb{C}}
\newcommand{\hc}{\textbf H^2_\C}
\newcommand{\R}{\mathbb{R}}

\DeclareMathOperator{\im}{Im}
\DeclareMathOperator{\re}{Re}
\DeclareMathOperator{\id}{Id}
\DeclareMathOperator{\is}{Is}

\DeclareMathOperator{\area}{Area}

\theoremstyle{plain}
\newtheorem{theo}{Theorem}[section]
\newtheorem{lemma}[theo]{Lemma}
\newtheorem{prop}[theo]{Proposition}

\theoremstyle{definition}
\newtheorem{defin}[theo]{Definition}

\newtheorem{rk}[theo]{Remark}

\title{Deligne-Mostow lattices with three fold symmetry and cone metrics on the sphere}
\author{I. Pasquinelli}

\begin{document}

\maketitle

\begin{abstract}
Deligne and Mostow, in \cite{mostow2}, \cite{mostow3} and \cite{delignemostow}, constructed a class of lattices in $PU(2,1)$ using monodromy of hypergeometric functions. 
Thurston in \cite{thurston} reinterpreted them in terms of cone metrics on the sphere. 
In this spirit we construct a fundamental domain for the lattices with three fold symmetry in the list of Deligne and Mostow.
This is a generalisation of the works in \cite{livne} and \cite{boadiparker} and gives a different interpretation of the fundamental domain constructed in \cite{type2}. 
\end{abstract}


\section{Introduction}
One of the main goals in complex hyperbolic geometry is the study of lattices in $PU(n,1)$. 

In complex dimension two, Deligne and Mostow, in several works, including \cite{delignemostow} (as explained in the survey article \cite{survey}), gave several constructions of lattices arising as monodromy groups of hypergeometric functions, which were defined using 5 parameters satisfying some properties, called a ball quintuple.
This also leads to a sufficient condition on the ball quintuple for the monodrmy group to be a lattice, called $\Sigma$INT.
Later, Thurston (see \cite{thurston}) showed that they can equivalently be seen as modular groups of flat cone metrics on the sphere. 
Following this approach, one can consider quintuples of cone angles at singularities (strictly related to the ball quintuples) and obtain an explicit, sufficient condition on them for the modular group to be a lattice.
This condition is called Thurston's orbifold condition and is equivalent to Mostow's $\Sigma$INT condition.
Mostow then also found more ball quintuples giving discrete groups but not satisfying $\Sigma$INT.
Sauter, in \cite{sauter}, studied these groups and showed that they are all commensurable to some groups in the original list.

Among the lattices in the original list from Deligne and Mostow work, we consider the ones with three fold symmetry. 
This means that three of the five singularities have the same angle. 
In \cite{survey}, Parker gives a table summarising all the three fold symmetry lattices, included the ones studied by Sauter, that will not be treated in this work. 
The table with Deligne and Mostow's lattices is explained in Section \ref{list} and in this work we will consider all the 39 values of parameters contained in the table.

For some of these a fundamental domain has already been constructed.
In particular, Deraux, Falbel and Paupert in \cite{type2} gave a construction for some Mostow groups.
Later, Parker in \cite{livne} constructed a fundamental polyhedron for the Livné lattices using a slightly different method. 
Later on, Boadi and Parker in \cite{boadiparker} used the same method for obtaining a fundamental domain for some Mostow groups of the first type. 
We will use the latter method.
In \cite{survey} Parker summarises the known constructions and shows the relation between the method used for Livné lattices and lattices of the first type and the one used in \cite{type2}. 

In this paper we will give a general construction which covers all the remaining cases, but also contains the previous ones. 

The next section will briefly define the complex hyperbolic space and give basic properties about it, its isometries and its subspaces. 
Each lattice we will be working on is identified by some parameters which we explain in Section \ref{list} and we will say it is of a certain type according to their values. 
The three fold symmetry lattices will be described in Section \ref{list}.
In particular, we will explain how they are determined by two parameters, $p$ and $k$, as explained in the survey \cite{survey}, that determine which class the lattice is in. 
Working in full generality on the parameters, we will see that as the parameters vary within a certain range, the combinatorics of the fundamental polyhedron is fully determined. 

Starting with a cone metric on a sphere with five cone singularities with prescribed angles (arising from Mostow's ball quintuples) and area one, we show that the moduli space of such configurations is a complex hyperbolic space, following Thurston's approach.

The key remark lies in the fact that if we cut the sphere through the five singular points and then we open it up, we get an octagonal shaped figure $\Pi$.
Such $\Pi$, and hence the space of cone metrics on the sphere with prescribed angles, can be parametrised by points in $\C^3$ and its area is a Hermitian form $H$ of signature (1,2) on $\C^3$. 
Since we consider metrics of area one, hence configurations up to rescaling and we want the area to be positive, we get a complex hyperbolic 2-space as the moduli space. 

Then we define some moves on the cone structures, which correspond to isometries in $\hc$. 
The first two are obtained by swapping two of the three singularities with the same angle. 
The third one is a generalisation of Thurston's butterfly moves (see \cite{thurston}). 
The isometries given by the moves are generators for the group $\Gamma$, which is a lattice in $PU(H)$ and for which we construct a fundamental domain. 

Following Thurston's idea, we consider what happens when one or more cone singularities collapse, becoming a single point. 
These will be the vertices of the polyhedron and of its images under the isometries defined by the moves. 

Each side of the polyhedron (i.e. maximal dimension facet) is contained in a bisector.
Bisectors are among the best understood subspaces of the complex hyperbolic plane and have some useful properties. 
By intersecting the sides and calculating the dimension of these intersections we then find also 2-dimensional and 1-dimensional facets of the polyhedron. 
They are called the ridges and the edges. 

Finally, we use Poincaré's polyhedron theorem to prove that the polyhedron we constructed is actually a fundamental domain for $\Gamma$. 
For the polyhedron to verify Poincaré's theorem it needs to satisfy a few conditions. 
In particular, some combinations of the three moves, that are the generators of $\Gamma$, have to pair the sides sending one in the other, in a way that satisfies some special properties, according to the theorem. 
Because of this they are called side pairing maps.
Moreover, we have some conditions on the ridges, the most difficult of which has been to prove that the polyhedron and its images under the side pairing maps tessellate a neighbourhood of the interior of each ridge. 
A slightly different method will be used when the $k$ parameter is not an integer.

The power of Poincaré's polyhedron theorem lies not only in the fact that it proves that the polyhedron is actually a fundamental domain for the group, but also because it gives a presentation for the group.
The conditions on sides and ridges consist, in fact, also of some relations on the maps, called respectively reflection relations and cycle relations. 
Using the side pairing maps as generators and such relations, we get a full presentation for the group, which makes the picture more complete. 

The previous cases mentioned, in which such a method has been already applied, are implicitly contained in the construction we worked out and our approach unifies them.
As we said, according to the range values of $p$ and $k$, we have different cone angles and hence different configurations. 

In particular, we showed that all previous cases can be obtained from our polyhedron by collapsing some vertices, mainly three by three. 
Equivalently, we can find our polyhedron from the previous ones by "cutting" some of the vertices, so as to obtain three vertices and a new ridge instead. 
In particular, the polyhedron found in \cite{livne} and \cite{boadiparker} is exactly the one we construct here. 
In the last part we explain the relation between our construction and the one in \cite{type2}.

I would like to thank my supervisor, John  Parker, for his constant support and the many insightful discussions during the preparation of this work.
This research was supported by a Doctoral EPSRC Grant, awarded by Durham University.  

\section{Complex hyperbolic space} \label{H2C}

In this section we will define the complex hyperbolic space, its main properties and some information about its isometries. 
All the information presented here can be found in more depth in the book from Goldman \cite{goldman}.

\subsection{Definition}

The complex hyperbolic space arises naturally as a complex analogue to the real hyperbolic space $\textbf{H}^n_\R$.
The real hyperbolic plane is, in fact, an example of complex hyperbolic space of dimension 1. 
Generalising this construction to a complex vector space we get complex hyperbolic space. 

Let us take a complex vector space $\C^{n,1}$ of dimension $n+1$, equipped with a Hermitian form of signature $(n,1)$.
We consider the Hermitian form in matrix form, given by an Hermitian matrix $H$ (i.e. $H=H^*$), which is non singular, with $n$ positive eigenvalues and one negative. 
Here $A^*$ is always be defined by $A^*=\overline{A^T}$ and the same notation will be used for vectors.

Such matrix gives a product law on $\C^{n,1}$ that we denote
\[
\langle \textbf z, \textbf w \rangle = \textbf w ^* H \textbf z.
\]
For $\textbf z \in \C^{n,1}$, its norm under the product just defined, $\langle \textbf z, \textbf z \rangle = \textbf z^* H \textbf z$, is real, but it can be positive, negative or zero. 
We hence decompose the space $\C^{n,1} \setminus \{0\}$ in subspaces made of vectors where $\langle \textbf z, \textbf z \rangle$ is positive, zero or negative, namely $V_+, V_0, V_-$ respectively. 

We now projectivise $\C^{n,1} \setminus \{0\}$ by identifying all non-zero complex multiples of a given vector. 
In other words, we are considering the projection $\P$ of $\C^{n,1} \setminus \{0\}$ onto $\C\P^n$. 
The projection $\P$ preserves the subspaces $V_+, V_0$ and $V_-$, because for $\lambda \in \C \setminus \{0\}$, we have 
\[
\langle \lambda \textbf z, \lambda \textbf z \rangle= (\lambda \textbf z )^* H (\lambda \textbf z )= \lvert \lambda \rvert ^2 \textbf z ^* H \textbf z = \lvert \lambda \rvert ^2 \langle \textbf z, \textbf z \rangle
\]
and hence $\langle \lambda \textbf z, \lambda \textbf z \rangle$ and $\langle \textbf z, \textbf z \rangle$ must have same sign. 
In other words $\textbf z$ and $\lambda \textbf z$ must be in the same subspace. 

We are now ready to define the complex hyperbolic space as $\textbf H_\C^n=\P V_-$, i.e. the space of vectors of negative norm, up to multiplication by complex numbers. 
Its boundary is $\partial \textbf H_\C^n=\P V_0$.

On such space we consider the Bergman metric, given by the formula 
\[
ds^2=\frac{-4}{\langle \textbf z, \textbf z \rangle^2} \det
\begin{pmatrix}
\langle \textbf z, \textbf z \rangle & \langle d\textbf z, \textbf z \rangle \\
\langle \textbf z, d \textbf z \rangle & \langle d\textbf z, d\textbf z \rangle
\end{pmatrix}.
\]
Consequently, for two points $\textbf z$ and $\textbf w$, their distance $\varrho (\textbf z, \textbf w)$ is given by
\begin{align}\label{dist}
\cosh ^2 \left( \frac{\varrho(\textbf{z}, \textbf w)}{2} \right) =\frac{\langle \textbf{z},\textbf w \rangle \langle \textbf w, \textbf z \rangle}{\langle \textbf{z},\textbf{z} \rangle \langle \textbf w, \textbf w \rangle}.
\end{align}

\subsection{The group of isometries and its subgroups}

The group of holomorphic isometries of $\textbf H_\C^n$ is generated by the projectivisation of the group of matrices that are unitary with respect to $H$. 
More precisely, let $U(H)$ be the group of square matrices of dimension $n+1$ such that $A^* H A=H$.
We say that such matrices are unitary with respect to $H$. 
Naturally, we will have $SU(H)$ the subgroup of such matrices with determinant equal 1. 

To get the holomorphic isometries of $\textbf H_\C^n$, we need to projectivise such a group as we did for the space itself, whence the holomorphic isometry group of $\textbf H_\C^n$ is
\[
PU(H)=U(H) / \{e^{i \theta}I \colon \theta \in [0,2\pi)\}.
\]
This group and complex conjugation generate the full isometry group of $\textbf H_\C^n$. 
Sometimes, to stress the dimension of the complex hyperbolic space it acts on, we will denote this group as $PU(n,1)$.

The goal of this work is to give an explicit construction of a fundamental domain for some lattices in $PU(H)$ for the 2-dimensional complex hyperbolic space.  
We make the convention that a fundamental domain is always an open region.
Lattices are a particular kind of subgroup and we will give this definition to conclude this section. 
Let $G=PU(H)$. 
A discrete subgroup $\Gamma$ is a lattice when the quotient $\Gamma \backslash \textbf H_\C^n$ has finite volume with respect to the Bergman metric. 

\subsection{Bisectors}\label{bisectors}

One of the most important classes of submanifolds in complex hyperbolic geometry is that of bisectors. 
In this section we will give a brief description and expose the main properties we will need. 
These subspaces have been widely studied and more details can be found in \cite{goldman}.

Bisectors are defined as the locus of points in the complex hyperbolic space which are equidistant from two given points, say $\textbf{z}_i$ and $\textbf{z}_j$.
By the formula in \eqref{dist}, it gives 
\[
\frac{\langle \textbf{z},\textbf{z}_j \rangle \langle \textbf{z}_j, \textbf{z} \rangle}{\langle \textbf{z},\textbf{z} \rangle \langle \textbf{z}_j, \textbf{z}_j \rangle}=\cosh ^2 \left( \frac{\varrho(\textbf{z}, \textbf z_j)}{2} \right) = \cosh^2 \left( \frac{\varrho(\textbf{z}, \textbf z_i)}{2} \right) =\frac{\langle \textbf{z},\textbf{z}_i \rangle \langle \textbf{z}_i, \textbf{z} \rangle}{\langle \textbf{z},\textbf{z} \rangle \langle \textbf{z}_i, \textbf{z}_i \rangle},
\]
and if $\textbf z_i$ and $\textbf z_j$ have the same norm, the definition becomes:
\[
B=B(\textbf z_i,\textbf z_j)=\{ \textbf z \in \hc \colon \lvert \langle \textbf z, \textbf z_i \rangle \rvert =
\lvert \langle \textbf z, \textbf z_j \rangle \rvert \}.
\]

The complex line $L$ spanned by $\textbf z_i$ and $\textbf z_j$ is called \emph{complex spine} of the bisector. 
Inside $L$ there is a geodesic $\gamma$ which is the intersection between the complex spine and the bisector and it is called the \emph{spine} of the bisector. 

In the complex hyperbolic space there are no totally geodesic real hypersurfaces, so also the bisectors are obviously not totally geodesic. 
They can be foliated though by totally geodesic subspaces in two different ways: with slices or with meridians. 

To define the slices first take the map $\Pi_L$, which is the orthogonal projection of the whole space on the complex spine $L$. 
Then $B$ is the preimage by $\Pi_L$ of $\gamma$.
We hence define a \emph{slice} to be a complex line that is a fibre of the map $\Pi_L$, i.e. the preimage of a point of $\gamma$.

The other foliation is by meridians. A \emph{meridian} is a totally geodesic Lagrangian plane containing the spine $\gamma$. 
The bisector is the union of all its meridian.
A meridian is also the set of points fixed by a antiholomorphic involution which swaps $\textbf z_i$ and $\textbf z_j$.

Other important subspaces related to bisectors are Giraud discs. 
Take three points $\textbf{z}_i,\textbf{z}_j$ and $\textbf{z}_k$, not all contained in a complex line. 
Consider then $B(\textbf z_i,\textbf z_j, \textbf z_k)$, the set of points equidistant from these three points. 
Giraud's theorem tells us that such set is contained in exactly three bisectors $B(\textbf z_i,\textbf z_j), B(\textbf z_i,\textbf z_k)$ and $B(\textbf z_j,\textbf z_k)$.
Moreover, $B(\textbf z_i,\textbf z_j, \textbf z_k)$ is a smooth non totally geodesic disc, called a \emph{Giraud disc}.

\section{Mostow lattices with a 3-fold symmetry}\label{list}

The main goal of this work is to give a fundamental domain for all Deligne-Mostow lattices with three fold symmetry. 
In this section we will briefly describe, following \cite{survey}, how to parametrise these lattices. 

The initial work of Deligne and Mostow makes such lattices arise as monodromy groups of hypergeometric functions. 
Later, Thurston reinterpreted them in terms of modular group of cone metrics on the sphere.
Following this approach, we will show that the moduli space of cone metrics on the sphere with prescribed cone singularities have a complex hyperbolic structure and see the lattices as subgroups of automorphisms of the sphere. 

An important concept, appearing first in the work of Deligne and Mostow, is the one of \emph{ball $N$-tuple}. 
A ball $N$-tuple is a set of $N$ real numbers $\mu=(\mu_1, \dots, \mu_N)$ verifying the conditions
\begin{align}\label{ballntuple}
\sum_{i=1}^{N} \mu_j &=2, & 0<\mu_i&<1, & \text{ for } i=1, \dots, N.
\end{align}

Now, a cone singularity on a surface is a point around which the total angle is not $2\pi$. 
In general, it can be any value, but in this work we will consider it to be in $(0,2\pi)$. 

A flat cone metric on the sphere is a metric modelled on $\R^2$ except for a finite number of points that are cone singularities. 
Around these points the surface can be described by  taking the part of $\R^2$ defined by $\{z=r e ^{i \theta} \in \C \colon 0 \leq \theta \leq \theta_0 \}$ and identifying the edges of the sector through the map $r \sim r e^{i \theta_0}$.
We will say that such a point is a cone singularity of angle $\theta_0$ and we will call its curvature the value $\alpha = 2\pi -\theta_0$.
Outside the singularities the curvature is 0.

If we have $N$ cone singularities, the curvatures $\alpha_1, \dots, \alpha_N$  must satisfy 
\begin{align}\label{curvatures}
\sum_{i=1}^{N} \alpha_j &=4 \pi, & 0<\alpha_i&<2\pi, & \text{ for } i=1, \dots, N.
\end{align}

Comparing \eqref{ballntuple} and \eqref{curvatures} it is obvious that there is a correspondence between the two: for a ball $N$-tuple $(\mu_1, \dots, \mu_N)$ we can construct a cone metric on the sphere with curvatures $\alpha_i=2\pi \mu_i$ and vice versa. 

As we will see for the case $N=5$, $N$ cone singularities on the sphere will give a subgroup in $PU(N-3,1)$, which is a lattice for singularities with certain prescribed curvatures (equivalently, for certain ball 5-tuples). 
For $N=5$, we say that a lattice has three fold symmetry when at least three of the values of the corresponding ball $N$-tuple (and hence the cone angles) are equal. 

Table \ref{pkangle} summarizes all Deligne-Mostow lattices with three fold symmetry. 
The lattices are divided according to the sign of the four parameters in the first four columns, $p,k,l$ and $d$. 
These values are very important, as we will see that they are the order of some special elements of the lattices. 
In particular, the first two can uniquely determine the ball quintuple and hence the curvature and the cone angles of the singularities on the sphere, from which we can obtain a lattice in the way we will see in the following sections. 

The elements of the ball quintuple, listed in the last columns, are related with the parameters $p$ and $k$ in the following way: 

\begin{align}\label{pkangle}
\mu_1&=\frac{1}{2}+\frac{1}{p}-\frac{1}{k}, 
& \mu_2=\mu_3=\mu_4&=\frac{1}{2} -\frac{1}{p}, 
& \mu_5=\frac{2}{p}+\frac{1}{k}.
\end{align}
The other parameters are defined from the two first ones in the following way:
\begin{align}\label{ldef}
\frac{1}{l}&=\frac{1}{2}-\frac{1}{p}-\frac{1}{k}, 
& \frac{1}{d}&=\frac{1}{2} -\frac{3}{p},
& t&=-\frac{1}{2}+\frac{1}{p}+\frac{2}{k}.
\end{align}

\newpage
\begin{center}
\begin{longtable}{|c|c|c|c|c|c|c|c|}
\hline
$p$ & $k$ & $l$ & $d$ & $t$ & $\mu_1$ & $\mu_{2,3,4}$ & $\mu_5$ \\
\hline
3 & 4 & -12 & -2 & 1/3 & 7/12 & 1/6 & 11/12 \\
3 & 5 & -30 & -2 & 7/30 & 19/30 & 1/6 & 13/15 \\
3 & 6 & $\infty$ & -2 & 1/6 & 2/3 & 1/6 & 5/6 \\
4 & 3 & -12 & -4 & 5/12 & 5/12 & 1/4 & 5/6 \\
4 & 4 & $\infty$ & -4 & 1/4 & 1/2 & 1/4 & 3/4 \\
5 & 2 & -5 & -10 & 7/10 & 1/5 & 3/10 & 9/10 \\
5 & 5/2 & -10 & -10 & 1/2 & 3/10 & 3/10 & 4/5 \\
5 & 3 & -30 & -10 & 11/30 & 11/30 & 3/10 & 11/15 \\
6 & 2 & -6 & $\infty$ & 2/3 & 1/6 & 1/3 & 5/6 \\
6 & 3 & $\infty$ & $\infty$ & 1/3 & 1/3 & 1/3 & 2/3 \\
\hline
3 & 7 & 42 & -2 & 5/42 & 29/42 & 1/6 & 17/21 \\
3 & 8 & 24 & -2 & 1/12 & 17/24 & 1/6 & 19/24 \\
3 & 9 & 18 & -2 & 1/18 & 13/18 & 1/6 & 7/9 \\
3 & 10 & 15 & -2 & 1/30 & 11/15 & 1/6 & 23/30 \\
3 & 12 & 12 & -2 & 0 & 3/4 & 1/6 & 3/4 \\
4 & 5 & 20 & -4 & 3/20 & 11/20 & 1/4 & 7/10 \\
4 & 6 & 12 & -4 & 1/12 & 7/12 & 1/4 & 2/3 \\
4 & 8 & 8 & -4 & 0 & 5/8 & 1/4 & 5/8 \\
5 & 4 & 20 & -10 & 1/5 & 9/20 & 3/10 & 13/20 \\
5 & 5 & 10 & -10 & 1/10 & 1/2 & 3/10 & 3/5 \\
6 & 4 & 12 & $\infty$ & 1/6 & 5/12 & 1/3 & 7/12 \\
6 & 6 & 6 & $\infty$ & 0 & 1/2 & 1/3 & 1/2 \\
\hline
7 & 2 & -7 & 14 & 9/14 & 1/7 & 5/14 & 11/14 \\
8 & 2 & -8 & 8 & 5/8 & 1/8 & 3/8 & 3/4 \\
9 & 2 & -9 & 6 & 11/18 & 1/9 & 7/18 & 13/18 \\
10 & 2 & -10 & 5 & 3/5 & 1/10 & 2/5 & 7/10 \\
12 & 2 & -12 & 4 & 7/12 & 1/12 & 5/12 & 2/3 \\
18 & 2 & -18 & 3 & 5/9 & 1/18 & 4/9 & 11/18 \\
\hline
7 & 3 & 42 & 14 & 13/42 & 13/42 & 5/14 & 13/21 \\
8 & 3 & 24 & 8 & 7/24 & 7/24 & 3/8 & 7/12 \\
9 & 3 & 18 & 6 & 5/18 & 5/18 & 7/18 & 5/9 \\
10 & 3 & 15 & 5 & 4/15 & 4/15 & 2/5 & 8/15 \\
12 & 3 & 12 & 4 & 1/4 & 1/4 & 5/12 & 1/2 \\
18 & 3 & 9 & 3 & 2/9 & 2/9 & 4/9 & 4/9 \\
\hline
7 & 7/2 & 14 & 14 & 3/14 & 5/14 & 5/14 & 4/7 \\
8 & 4  & 8 & 8 & 1/8 & 3/8 & 3/8 & 1/2 \\
9 & 9/2 & 6 & 6 & 1/18 & 7/18 & 7/18 & 4/9 \\
10 & 5 & 5 & 5 & 0 & 2/5 & 2/5 & 2/5 \\
12 & 4 & 6 & 4 & 1/12 & 1/3 & 5/12 & 5/12 \\
\hline
\caption{Deligne-Mostow lattices with three fold symmetry.} 
\label{tablelist}
\end{longtable}
\end{center}

The fifth parameter, $t$ is a real parameter used by Mostow to describe the lattices, together with $p=3,4,5$ in \cite{mostow}. 
It is called the phase shift, because Mostow's phase parameter is $\varphi$, defined by $\varphi ^3=e^{\pi i t}$. 
One particular critical value of this parameter is $\frac{1}{2}-\frac{1}{p}$. 
In Section \ref{mainthm} we will see why this is relevant for our analysis. 
We will say, following Mostow, that it is a lattice with \emph{large phase shift} if the condition $\lvert t \rvert > \frac{1}{2}-\frac{1}{p}$ holds.
The opposite condition is a \emph{small phase shift}.

\section{Cone structures}

\begin{figure}
\centering
\includegraphics[width=0.3\textwidth]{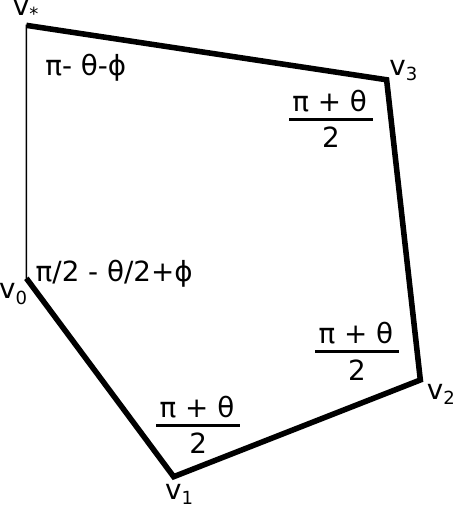}
\begin{quote}\caption{Double pentagon and cut through the five points. \label{doublepentagon}} \end{quote}
\end{figure}

Let us now consider a cone metric on the sphere with 5 cone singularities of angles 
\begin{equation} \label{singularities}
\left( \pi-\theta+2 \phi, \pi + \theta, \pi +\theta, \pi + \theta, 2\pi -2\theta -2\phi \right).
\end{equation}
We call the cone points $v_0, v_1, v_2, v_3, v_*$ respectively.
The angles $\theta$ and $\phi$ correspond respectively to $\frac{2 \pi}{p}$ and $\frac{\pi}{k}$, with $p$ and $k$ in Table \ref{tablelist}.

For simplicity, let us first assume that the position of the five cone singularities is, as in Figure \ref{doublepentagon}, such that the sphere is like a pentagonal pillowcase and let us consider a path in the sphere that starts from $v_0$ and passes in order through $v_1, v_2, v_3$, ending in $v_*$. 
Suppose we cut through this path and open up the surface, obtaining an octagon like the one in Figure \ref{octagonreal}, which we call $\Pi$. 
To be able to express the vertices of $\Pi$ with coordinates, we impose that the vertex $v_*$ coincides with the origin of the complex plane and we place $\Pi$ such that the coordinate of $v_0$ is a multiple of $i$ by a negative real number.
The vertices with positive real coordinates will be called $v_1, v_2, v_3$, while the corresponding vertices with negative real coordinates will be $v_{-1}, v_{-2}, v_{-3}$.

The sides of $\Pi$ are pairwise identified through a reflection with respect to the imaginary axis and this identification allows us to recover the cone metric on the sphere. 
More precisely, the vertices $v_i$ are identified to $v_{-i}$ and the edge between $v_i$ and $v_j$ is identified with the one between $v_{-i}$ and $v_{-j}$.
Since only the boundary points and not the interior are identified, this gives us back the shape of the cone metric as two pentagons glued through the boundary, forming the pentagonal pillowcase we started from. 

We can also describe $\Pi$ in terms of three real parameters, which we will call $x_1, x_2, x_3$.
Let us take three triangles $T_1, T_2$ and $T_3$ in the following way.
The triangle $T_1$ has the three angles $\phi$, $\frac{\pi - \theta}{2}$ and $\frac{\pi}{2}+\frac{\theta}{2}-\phi$ and side $x_1$ opposite to the angle $\frac{\pi}{2}+\frac{\theta}{2}-\phi$.
The triangle $T_2$ is isosceles. It has two angles equals to $\frac{\pi - \theta}{2}$ and one $\theta$. 
The two equal sides have length $x_2$. 
The triangle $T_3$ has the three angles $\phi$, $\pi - \theta -\phi$ and $\theta$ and side $x_3$ opposite to the angle $\pi - \theta -\phi$.

We now construct an octagon $\Pi$ by first taking a copy of the third triangle $T_3$, with the vertex with angle $\pi-\theta-\phi$ at 0 and the one with angle $\phi$ along the imaginary axis and below it.
Then remove from $T_3$ a copy of $T_2$ by making the two vertices of angle $\theta$ coincide and by making $x_2$ and $x_3$ be collinear, both vectors pointing towards the common corner of the two triangles $T_3$ and $T_2$. 
Similarly, remove from the figure obtained a copy of $T_1$ disposed such that the vertex of angle $\phi$ of $T_1$ coincides with the one of $T_3$ with the same angle and such that $x_2, x_3$ are collinear and pointing in the same direction.
At this point we reflect the whole construction along the imaginary axis, obtaining three more triangles $T_{-3}$, $T_{-2}$ and $T_{-1}$. 
We consider the quadrilateral made of the two triangles $T_3$ and $T_{-3}$, from which we delete triangles $T_i$, for $i \in \{\pm 1, \pm 2\}$. 
The figure obtained is an octagon $\Pi$ as in Figure \ref{octagonreal}.
This is clearly the same figure as we described previously when we label the vertices as explained before. 

It is easy, in the system previously described, to calculate the coordinates of the vertices of the octagon.
These are the same value that one can find in \cite{boadiparker}.

\begin{figure}[t]
\centering
\includegraphics[width=0.9\textwidth]{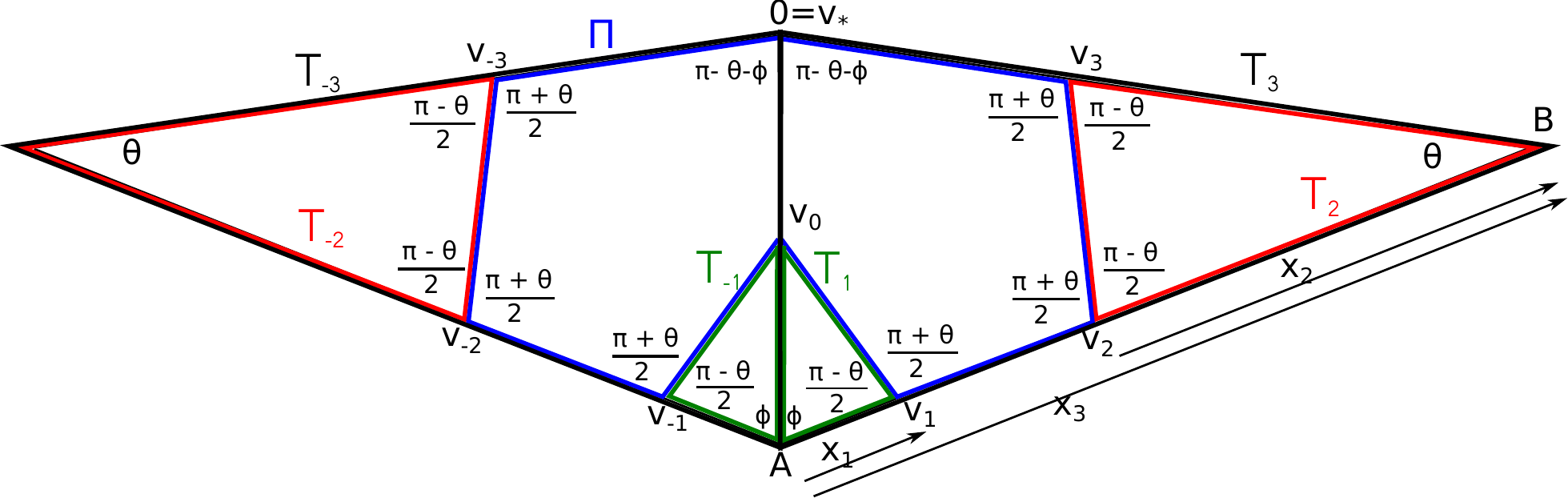}
\begin{quote}\caption{Octagon $\Pi$ when the parameters are real. \label{octagonreal}} \end{quote}
\end{figure}

We now consider a generic metric on the sphere and the same procedure applies, but we need now to allow the three variables to be complex, in order to describe all possible mutual positions of the singularities.
The variables describing the octagon will be called $z_1, z_2$ and $z_3$. 
We construct an octagon by taking the same three triangles and making the same vertices of the triangles coincide as before, but the three variables will be two dimensional vectors representing the sides of the triangles and they will no longer line up.
It will be as in Figure \ref{octagoncx}.

As before, we can recover the metric on the sphere identifying the side between $v_i$ and $v_j$ with the one between $v_{-i}$ and $v_{-j}$. 
We obviously obtain a cone manifold which is homeomorphic to the sphere and has five cone point of angles equal to those that we had in the beginning.

In the case of real variables $x_i$'s, the area of the right half of the octagon can be obtained taking the area of $T_3$ and subtracting the area of $T_1$ and the area of $T_2$.
We then need to double this quantity to have the total area of $\Pi$.
When allowing the variables to be complex, we can see, using a cut and paste map, that the area remains given by the same formula substituting each $x_i$ with $z_i$ complex.
A simple calculation then shows that 
\begin{align}
\area\Pi&=2 \left(\area T_3-\area T_1 - \area T_2 \right) \nonumber \\
&=\frac{\sin \theta \sin \phi}{\sin(\theta + \phi)} |z_3|^2 -\sin \theta |z_2|^2-\frac{\sin \theta \sin \phi}{(\sin \phi + \sin(\theta - \phi))}|z_1|^2. \label{area}
\end{align}
We remark that these are the same values obtained in \cite{boadiparker}.

\begin{figure}[t]
\centering
\includegraphics[width=0.9\textwidth]{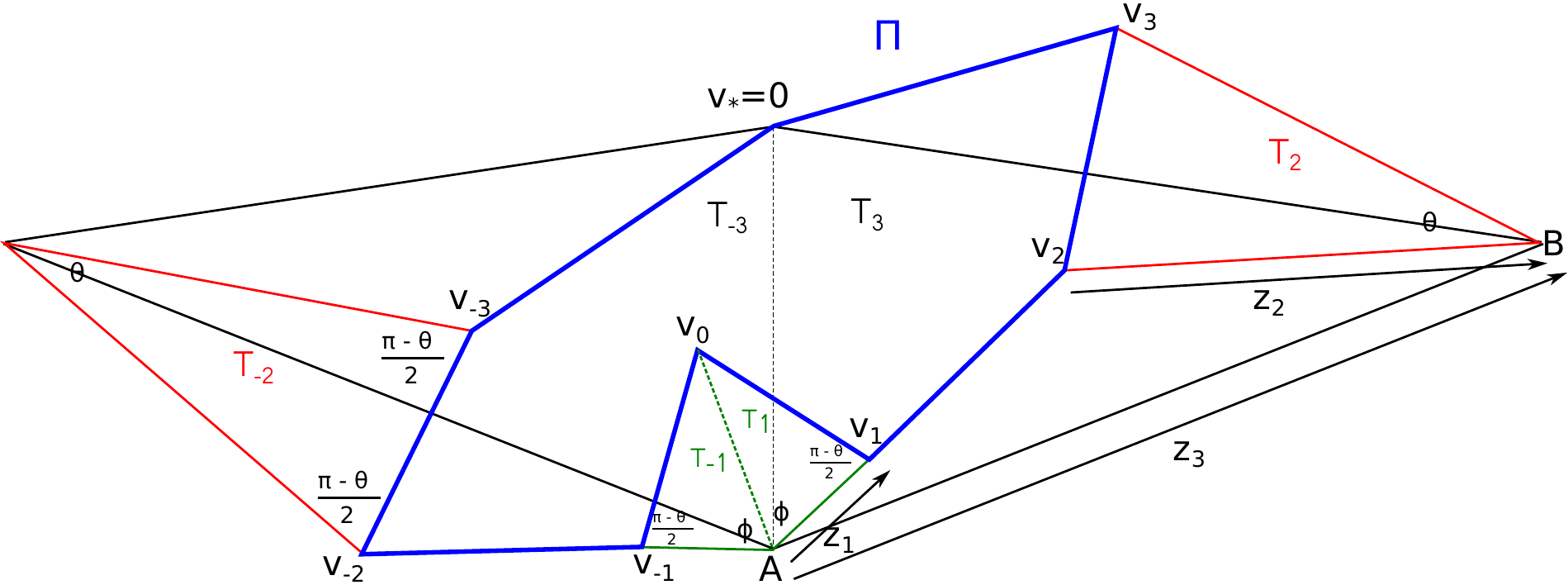}
\begin{quote}\caption{Octagon $\Pi$ when the parameters are complex. \label{octagoncx}} \end{quote}
\end{figure}

\section{Moves on the cone structures}

We will now define automorphisms of the polygons described above. 
This is the same procedure as in \cite{boadiparker}, which generalised \cite{livne}. 

We know that the second, third and fourth vertices have the same angle. 
This means that there is no canonical way of ordering them while chosing a path through the five points. 
Two of the three moves we will define are made by exchanging the order of the three cone point of same angle when making the cut. 
The third move will be in the spirit of Thurston's butterfly moves (see \cite{thurston}).

\begin{figure}[t]
\centering
\includegraphics[width=1\textwidth]{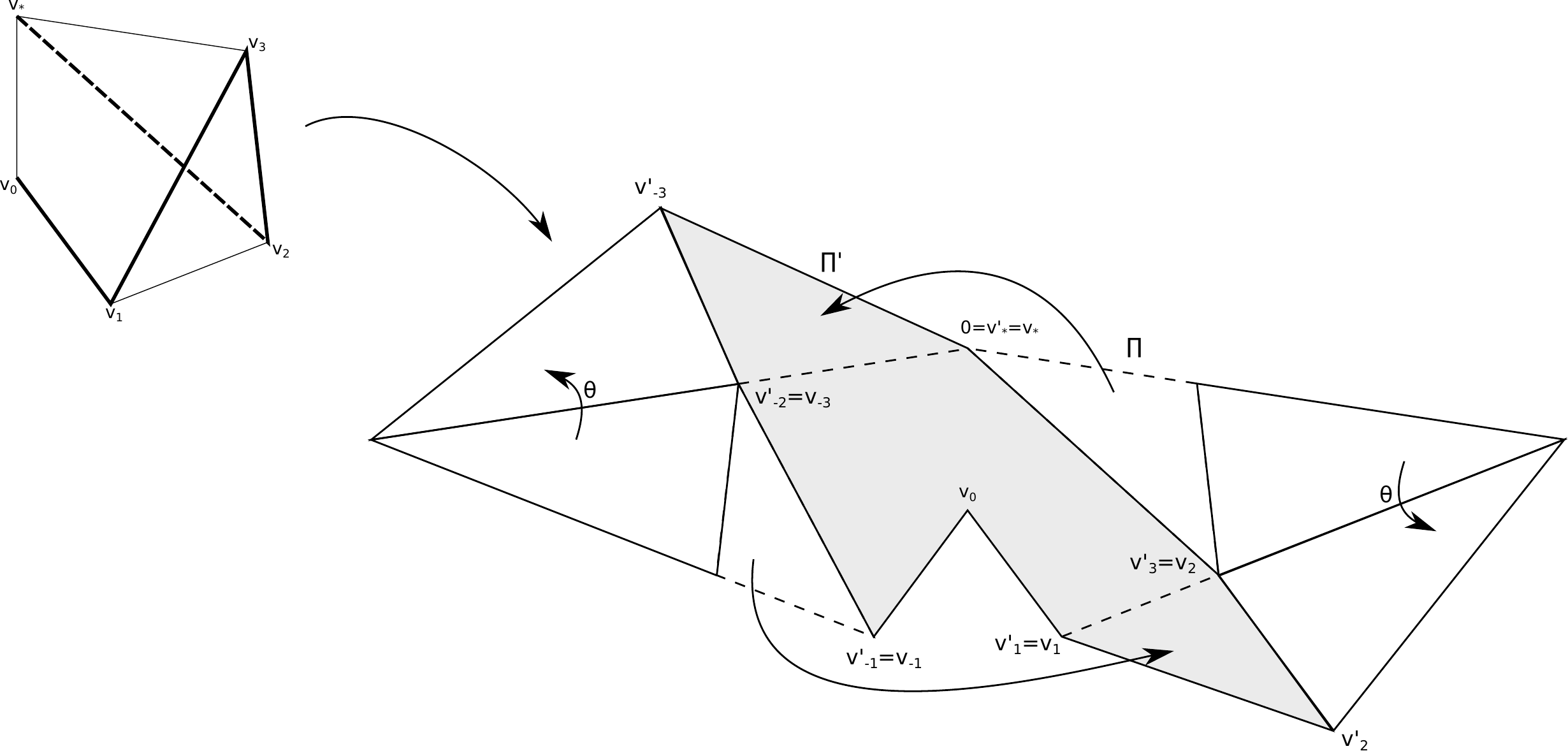}
\begin{quote}\caption{The cut for $R_1$ and the octagon we obtain. Vertices $v_i'$'s are the images under $R_1$ of $v_i$'s.\label{cutR1}} \end{quote}
\end{figure}

The first move $R_1$ fixes the vertices $v_*, v_0$ and $v_{1}$, and exchanges $v_{2}$ and $v_{3}$. 
This is equivalent to saying that the path on the sphere along which we will open up the surface to give the polygon $\Pi$ will be done starting in $v_0$, continuing in $v_{1}$ as before, but then passing, in order, through $v_{3}$ and $v_{2}$ and ending in $v_*$.
In Figure \ref{cutR1} we show the new cut in the glued pentagons case and the octagon that we obtain.

The new octagon can be obtained from the previous one by a cut and paste. 
In fact, the new cut from $v_*$ goes directly where $v_2$ was previously, as this is the image of $v_3$. 
So the triangle $v_*$, $v_3$, $v_2$ has to be glued on the segment between $v_*$ and $v_{-3}$ according to the identification of the sides.
Similarly, the triangle $v_{-1}$, $v_{-2}$, $v_{-3}$ has to be glued on the edge $v_1$, $v_2$, as in Figure \ref{cutR1}. 
This means that the move $R_1$ does not change the area of the octagon.

One way to find the matrix of $R_1$ is by describing geometrically the position of the new variables, image of the $z_i$'s.
In fact, if we leave $z_3$ and $z_1$ as before and we multiply $z_2$ by $e^{i \theta}$, it geometrically means that we are rotating $T_2$ and $T_{-2}$ by $\theta$, fixing the vertex corresponding to angle $\theta$, by definition of the variables. 
It is easy to see that this gives the configuration on the right hand side of Figure \ref{cutR1}.  

The matrix of $R_1$ will hence be:
\[R_1=
\begin{bmatrix}
1 & 0 & 0 \\
0 & e^{i \theta} & 0 \\
0 & 0 & 1 \\
\end{bmatrix}.
\]

There is yet another way of calculating the matrix.
As we can see in the figure, some of the images will be in the position where the vertices originally were. 
This means that, when considering their dependence on the new variables, it is enough to ask that the coordinates of these images (in term of the image of the variables $z_i$'s) coincide with the coordinates of the original vertices, which depended on the $z_i's$ themselves.
More specifically, to find the matrix of $R_1$, we need to solve equations  $v_0'=v_0$, $v_1'=v_1$, $v_3'=v_2$ and $v_{-2}'=v_{-3}$. 

\begin{figure}[t]
\centering
\includegraphics[width=1\textwidth]{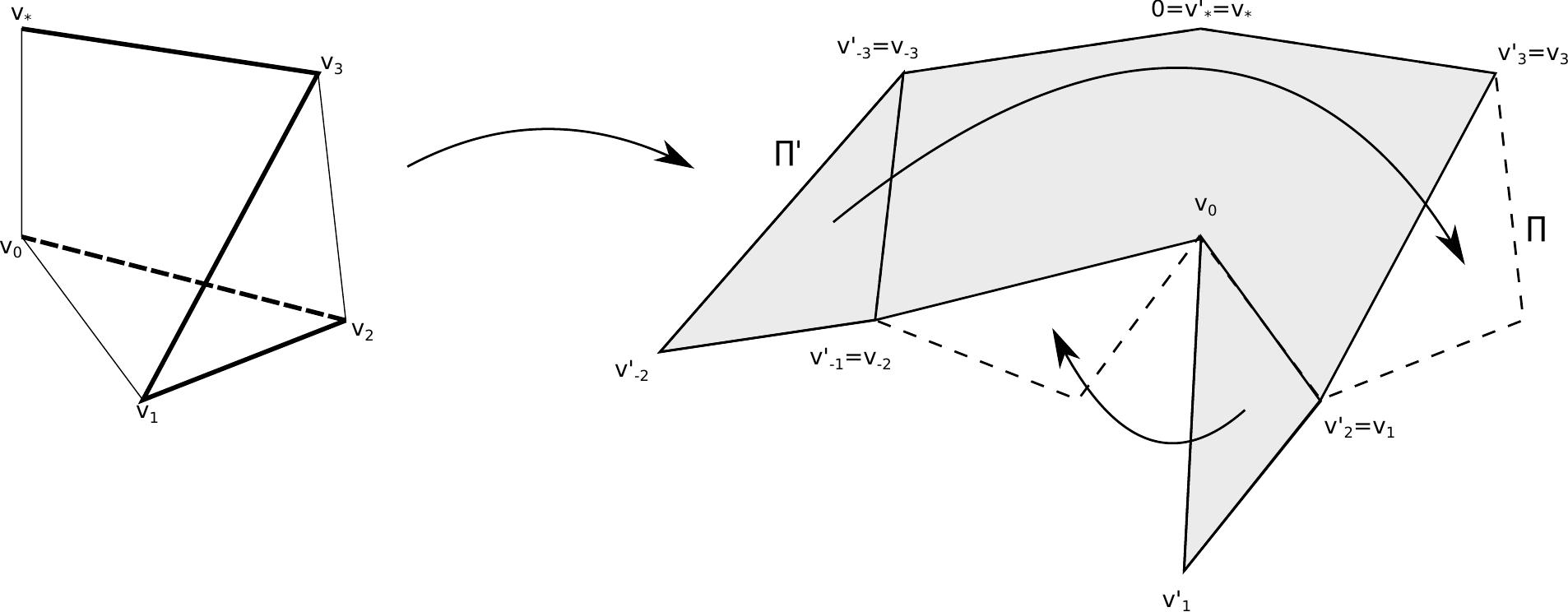}
\begin{quote}\caption{The cut for $R_2$. Again, $v_i'$ is the image under $R_2$ of $v_i$ \label{cutR2}} \end{quote}
\end{figure}

Let us now define the second move $R_2$.
This new move fixes $v_*, v_0$ and $v_{3}$, while it interchanges $v_{1}$ and $v_{2}$. 
As before, this means that the cut that we do goes first through $v_0$, then to $v_{2}$ and $v_{1}$ and finally it ends as before by cutting through $v_{3}$ and $v_*$. 
The cut and the octagon are shown is Figure \ref{cutR2}. 

As before, in the figure we also showed the cut and paste map that we need to recover the initial shape. 
In particular, the triangle between $v_3$, $v_2$ and $v_1$ has to be glued on the edge $v_{-2}$, $v_{-3}$, as this time the cut goes from $v_3$ directly to the image of $v_2$, that coincides now with the position of $v_1$.
Similarly, the triangle $v_0$, $v_{-1}$, $v_{-2}$ has to be glued on edge $v_0$, $v_1$. 
Both gluings are done according to the side identifications we described when recovering the come metric from the octagon.
We remark again that the existence of such a cut and paste implies that the area is preserved after applying the move $R_2$.

In this case the easiest method to find the matrix of the transformation is to see its action on the variables that determine the coordinates of the vertices.
According to Figure \ref{cutR2}, we therefore ask that $v_0'=v_0$, $v_2'=v_1$, $v_{-1}'=v_{-2}$ and $v_3'=v_3$. 

After some calculations that can be found in \cite{boadiparker}, we can get the matrix for $R_2$ as:
\[
R_2= \frac{1}{(1-e^{-i \theta}) \sin \phi}
\begin{bmatrix}
-\sin \theta e^{-i \phi} & -\sin \phi -\sin (\theta - \phi) & \sin \phi +\sin(\theta-\phi) \\
-\sin \phi & -\sin \phi e^{-i\theta} & \sin \phi \\
-\sin(\theta +\phi) &-\sin (\theta +\phi) & \sin \phi +\sin \theta e^{i \phi}
\end{bmatrix}.
\]

The two moves $R_1$ and $R_2$ correspond, as automorphisms of the sphere with 5 cone singularities, to a Dehn twist along a curve through the two points we are swapping, not separating the other singularities.

We will finally define the third move $A_1$. 
As we said, this is the generalisation of the "butterfly moves" used by Thurston in \cite{thurston}. 
In his case, he was moving one side across a region shaped like a butterfly such that in the end the signed area is the same.
Here, we make the triangle $T_1$ rotate so that vertices $v_*, v_{2}, v_{3}$ remain fixed, while $v_1'$ coincides this time with $v_{-1}$. 
We obtain an octagon with a point of self intersection and we need to consider the signed area to have it still preserved after applying the move. 

\begin{figure}[t]
\centering
\includegraphics[width=0.6\textwidth]{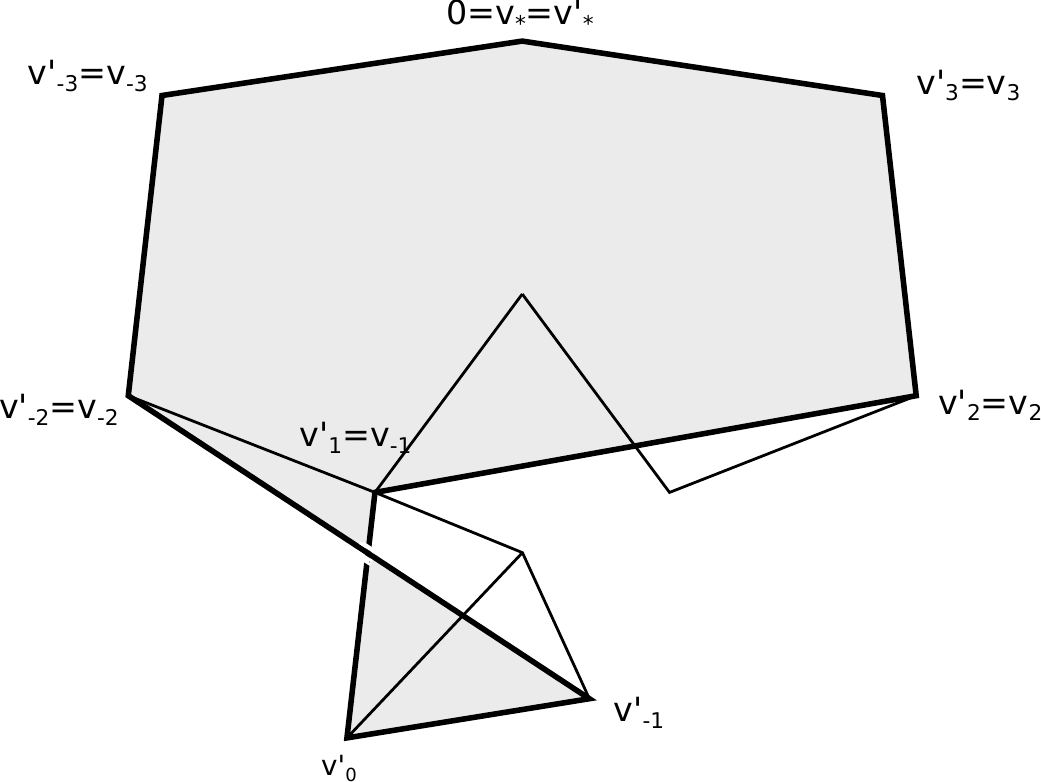}
\begin{quote}\caption{The octagon obtained after applying $A_1$. \label{octagonA1}} \end{quote}
\end{figure}

As we can see in Figure \ref{octagonA1}, the triangles $T_2$ and $T_3$ remain fixed and hence so are the variables $z_2$ and $z_3$.
The third triangle is rotated of an angle of $2\phi$. 
This gives us the matrix of the move, which will be
\[
A_1=\begin{bmatrix}
e^{2i \phi} & 0 & 0 \\
0 & 1 & 0 \\
0 & 0 & 1
\end{bmatrix}.
\]
As before, we can also see how it acts on the vertices and deduce from there the same matrix.

At this point, we want to consider the group $\Gamma= \langle R_1, R_2, A_1 \rangle$. 
For the values of $\phi$ and $\theta$ that we are considering, $\Gamma$ is discrete and is the list of Deligne-Mostow lattices described in Section \ref{list}.
In fact, here we are implementing Thurston's procedure described in \cite{thurston}, which, as he explains, is related with the groups previously constructed by Deligne and Mostow in \cite{delignemostow} and \cite{mostow}.
In the following sections we will construct a fundamental domain for the action of this group on the complex hyperbolic space. 

\section{Complex hyperbolic space as moduli space}

We will see now how the moduli space of cone metrics on the sphere of area 1, seen as the different shapes of polygons $\Pi$ that we can achieve, can be parametrised by a part of complex hyperbolic space. 
The moves we constructed will correspond to actions by isometries on the space. 

As we saw in Section \ref{H2C}, the 2-dimensional complex hyperbolic space is by definition the set of points for which a certain Hermitian form is positive, up to projectivisation.
First of all, up to now, all three parameters $z_1,z_2,z_3$ were freely chosen, but for our purpose two configurations such that the parameters are proportionals by the same constant are the same.
This is because we are considering the cone metrics to have fixed area, following Thurston (see \cite{thurston}, Theorem 0.2). 
From now on, we will hence fix $z_3=1$.
Recall that the area is given by \eqref{area} in terms of $z_1,z_2$ and $z_3$. 
The coordinates $z_1$ and $z_2$ will hence vary while keeping such quantity positive.
On the moduli space of cone metrics on the sphere this is equivalent to projectivising the coordinates. 

Let us now consider the area as given in equation \eqref{area}.
If we consider the Hermitian matrix 
\[
H=\sin \theta 
\begin{bmatrix}
-\frac{\sin \phi}{\sin \phi+\sin(\theta -\phi)}& 0 & 0 \\
0 & -1 & 0 \\
0 & 0 & \frac{\sin \phi}{\sin(\theta +\phi)}
\end{bmatrix},
\]
such formula is equivalent to saying 
\[
\area(\Pi)=\textbf z^* H \textbf z.
\]
In this sense, the area gives an Hermitian form of signature (1,2) on $\C^3$.

We define hence our model of complex hyperbolic space as 
\[
\hc=\{ \textbf z \colon \langle \textbf z, \textbf z \rangle = \textbf z^* H \textbf z > 0 \}.
\] 
Clearly, as we want our $\Pi$ to have positive area, this gives a complex hyperbolic structure on the moduli space of the polygon configurations. 
Equivalently, 
\begin{equation}\label{hypspace}
\hc = \left\{ 
\begin{bmatrix}
z_1 \\ z_2 \\z_3
\end{bmatrix} \colon 
\frac{-\lvert z_1 \rvert ^2 \sin \theta \sin \phi}{\sin \phi +\sin (\theta - \phi)} 
- \lvert z_2 \rvert ^2 \sin \theta 
+ \frac{\sin \theta \sin \phi}{\sin(\theta+\phi)}>0 \right\}.
\end{equation}

Since the moves preserve the area, they are unitary with respect to the Hermitian form, i.e. $R_1^* H R_1=H$ and same for $R_2$ and $A_1$.
This can also easily checked by calculation. 

\subsection{Some special maps}

In the group $\Gamma =\langle R_1, R_2, A_1 \rangle$, we will often use some special elements. 

The first one is $J$, defined as $J=R_1R_2A_1$.
Its matrix is 
\[
J=\frac{1}{\sin \phi (1-e^{-i \theta})}
\begin{bmatrix}
-\sin \theta e^{i \phi} & -\sin \phi - \sin(\theta-\phi) & \sin \phi+\sin(\theta-\phi) \\
-\sin \phi e^{i(2\phi+ \theta)} & -\sin \phi & \sin \phi e^{i \theta}\\
-\sin(\theta+\phi)e^{2i\phi} & -\sin(\theta+\phi) & \sin \phi +\sin \theta e^{i \phi}
\end{bmatrix}.
\]
We remark that $J$ has zero trace and hence it has order 3. 
Most of the time we will consider projective equalities and drop the initial factor $\frac{1}{\sin \phi (1-e^{-i \theta})}$.
Projective equivalence will be denoted by the symbol $\sim$.

The second one is $P$, defined by $P=R_1R_2$.
Its matrix is:
\[
P=\frac{1}{\sin \phi (1-e^{-i \theta})}
\begin{bmatrix}
-\sin \theta e^{-i \phi} & -\sin \phi - \sin(\theta-\phi) & \sin \phi+\sin(\theta-\phi) \\
-\sin \phi e^{i \theta} & -\sin \phi & \sin \phi e^{i \theta}\\
-\sin(\theta+\phi) & -\sin(\theta+\phi) & \sin \phi +\sin \theta e^{i \phi}
\end{bmatrix}.
\]
Note that $J$ previously defined can also be written as $J=PA_1$. 
The transformation $P$ will be used here to give a new set of coordinates different from the $\textbf z$-coordinates used until now.

The new coordinates are defined by 
\[
\textbf{w} = \left[ P^{-1}(\textbf{z})\right].
\]
This gives us the formulae
\begin{align}
w_1&= \frac{-\sin \theta e^{i \phi}z_1
-(\sin \phi +\sin (\theta-\phi))e^{-i \theta} z_2
+\sin \phi +\sin(\theta-\phi)}
{-\sin(\theta+\phi)z_1
-\sin (\theta+\phi)e^{-i \theta}z_2
+\sin \phi +\sin \theta e^{-i \phi}}, \label{w1}\\
w_2&= \frac{-\sin \phi z_1
-\sin \phi z_2
+\sin \phi}
{-\sin(\theta+\phi)z_1
-\sin (\theta+\phi)e^{-i \theta}z_2
+\sin \phi +\sin \theta e^{-i \phi}}, \label{w2}
\end{align}
with inverses
\begin{align}
z_1&= \frac{-\sin \theta e^{-i \phi}w_1
-(\sin \phi +\sin (\theta-\phi))w_2
+\sin \phi +\sin(\theta-\phi)}
{-\sin(\theta+\phi)w_1
-\sin (\theta+\phi)w_2
+\sin \phi +\sin \theta e^{i \phi}}, \label{z1} \\
z_2&= \frac{-\sin \phi e^{i \theta} w_1
-\sin \phi w_2
+\sin \phi e^{i \theta}}
{-\sin(\theta+\phi)w_1
-\sin (\theta+\phi)w_2
+\sin \phi +\sin \theta e^{i \phi}}. \label{z2}
\end{align}

The new set of coordinates makes it easier to describe the polyhedron, that will be defined by imposing that the arguments of the coordinates $z_1,z_2,w_1,w_2$ vary in a certain range.

We will often consider another transformation, which is the antiholomorphic isometry $\iota$ defined by $\iota(\textbf{z})=R_1R_2R_1 (\overline{\textbf{z}})$.
Equivalently, $\iota(\textbf{z})=PR_1(\overline{\textbf{z}})$. 
By definition, 
\begin{equation} \label{iota}
\iota \begin{bmatrix} z_1 \\ z_2 \\1 \end{bmatrix}
\sim \begin{bmatrix}
\overline w_1 \\ \overline w_2 e^{i \theta} \\1
\end{bmatrix}.
\end{equation}
This transformation will give us a symmetry of the polyhedron that we will construct (see Lemma \ref{iotaaction}). 

\begin{rk}\label{iotaconjug}
A simple computation shows that $\iota$ is consistent with the maps defined previously. 
In other words, we have
\begin{align*}
J \iota&=\iota J^{-1}, & P \iota &= \iota P^{-1} &
R_1 \iota &= \iota R_2^{-1} & R_2 \iota &= \iota R_1^{-1}.
\end{align*}
\end{rk}

\section{The polyhedron}\label{constr}

In this section we will construct the polyhedron that we will later prove to be a fundamental domain for the action of $\Gamma$. 
This is a general construction which contains all cases of lattices with three fold symmetry on Deligne and Mostow's list. 
The polyhedron as we will describe it here will be a fundamental domain in some of the cases described in Section \ref{list}.
In the other cases, the fundamental polyhedron will be obtained from this one by collapsing some triplets of vertices. 
Section \ref{degenerate} will be dedicated to the analysis of these cases. 

\subsection{Vertices}

We will now explain which points of $\hc$ are the special points which will represent the vertices of the polyhedron. 
For each of them we will give both $\textbf z$-coordinates and $\textbf w$-coordinates.
As before, $\textbf w= P^{-1}(\textbf z)$.

All these points will be obtained by making some cone points approach, until, in the limit, they coalesce. 
In this case, each vertex will be obtained by separately coalescing two distinct pairs of cone points. 
On the octagon $\Pi$, this corresponds to fixing the triangle $T_3$ and considering the cone metrics on the sphere corresponding to configurations when $T_1$ and $T_2$ are as small and as big as possible, in different directions, until pairs of vertices coincide. 
This is shown in Figure \ref{vertices}.
Every time that we make two points coalesce, we turn two cone points into a new one. 
Its curvature (complement of the cone angle), will be the sum of the curvatures of the two points that have coalesced.

\begin{figure}
\centering
\includegraphics[width=1\textwidth]{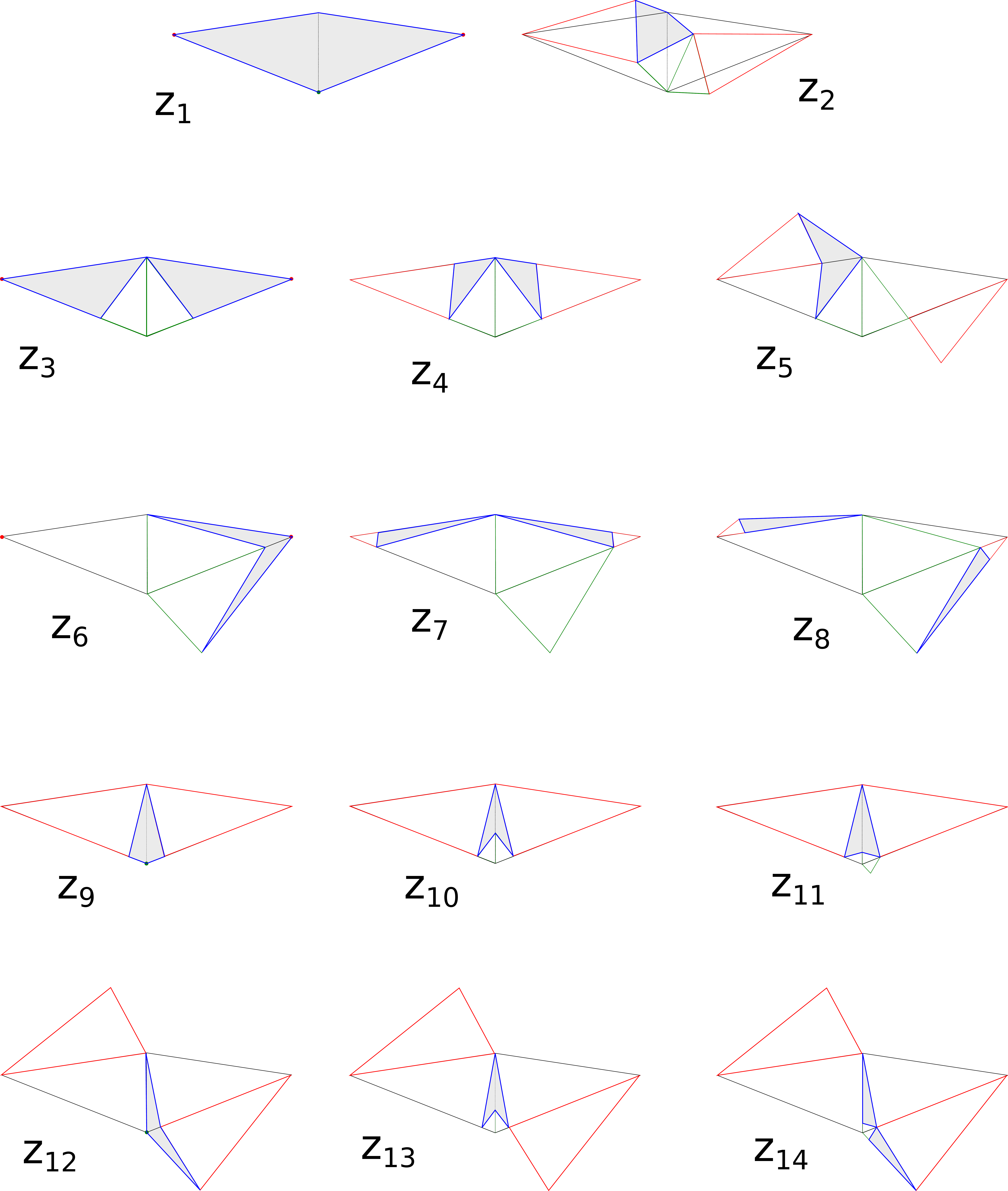}
\begin{quote}\caption{The degenerate configurations giving the vertices of the polyhedron. \label{vertices}} \end{quote}
\end{figure}

In the following tables we describe the vertices of the polyhedron.
The first one tells us, for each vertex, which cone points coalesced. 

\begin{center}
\begin{tabular}{|c|c|c||c|c|c||c|c|c|}
\hline
Vert. & \multicolumn{2}{|c||}{Cone points}  & Vert. & \multicolumn{2}{|c||}{Cone points} & Vert. & \multicolumn{2}{|c|}{Cone points}\\
\hline
$\textbf{z}_1$ & $v_0, v_{\pm 1}$ & $v_{\pm 2}, v_{\pm 3}$ & 
$\textbf{z}_6$ & $v_*,v_{\pm1} $ & $v_{\pm 2}, v_{\pm 3}$ &
$\textbf{z}_{11}$ & $v_*,v_{\pm3} $ & $v_0, v_{\pm 2}$ \\
$\textbf{z}_2$ & $v_0, v_{\pm 3}$ & $v_{\pm 1}, v_{\pm 2}$ &
$\textbf{z}_7$ & $v_*,v_{\pm1} $ & $v_0, v_{\pm 2}$ & 
$\textbf{z}_{12}$ & $v_*,v_{\pm2} $ & $v_0, v_{\pm 1}$ \\
$\textbf{z}_3$ & $v_*, v_0$ & $v_{\pm 2}, v_{\pm 3}$ &
$\textbf{z}_8$ & $v_*,v_{\pm1} $ & $v_0, v_{\pm 3}$ & 
$\textbf{z}_{13}$ & $v_*,v_{\pm2} $ & $v_{\pm1}, v_{\pm 3}$ \\
$\textbf{z}_4$ & $v_*,v_0 $ &  $v_{\pm 1}, v_{\pm 2}$ & 
$\textbf{z}_9$ & $v_*,v_{\pm3} $ & $v_0, v_{\pm 1}$ & 
$\textbf{z}_{14}$ & $v_*,v_{\pm2} $ & $v_0, v_{\pm 3}$ \\
$\textbf{z}_5$ & $v_*,v_0 $ & $v_{\pm 1}, v_{\pm 3}$ & 
$\textbf{z}_{10}$ & $v_*,v_{\pm3} $ & $v_{\pm1}, v_{\pm 2}$ &&&\\
\hline
\end{tabular}
\end{center}


When two cone points collapse, we get a complex line in $\hc$.
We will label these lines in the following way:
\[
L_{ij}= \text{ line obtained by making the cone points $v_i,v_j$ coalesce},
\]
for $i,j=0,1,2,3,*$.
We will also call $\textbf n_{ij}$ the polar vector to the line $L_{ij}$.
These complex lines are described by the following equations.

\begin{center}
\begin{tabular}{|c|c|c|c|}
\hline
$L_{ij}$ & Cone pts 
& $\textbf{z}$-coordinates equation
& $\textbf{w}$-coordinates equation \\
\hline

$L_{*0}$ & $v_*,v_0$
& $z_1= \frac{\sin \phi+ \sin(\theta - \phi)}{\sin (\theta +\phi)}$ 
& $w_1=\frac{\sin \phi + \sin (\theta - \phi)}{\sin(\theta+\phi)}$ \\

$L_{*1}$ & $v_*,v_{-1}$
& $z_1=e^{-i\phi}\frac{\sin \theta}{\sin(\theta +\phi)}$
& $w_2=e^{i \theta} \frac{\sin \phi}{\sin(\theta+\phi)}$\\

$L_{*2}$ & $v_*,v_{-2}$
& $z_2=e^{i \theta} \frac{\sin \phi}{\sin(\theta+\phi)}$
& $w_2=\frac{\sin \phi}{\sin(\theta+\phi)}$\\

$L_{*3}$ & $v_*,v_3$
& $z_2=\frac{\sin \phi}{\sin(\theta+\phi)}$
& $w_1=e^{i \phi}\frac{\sin \theta}{\sin(\theta +\phi)}$\\

$L_{01}$ & $v_0,v_1$
& $z_1=0$
& $\frac{\sin \theta}{\sin \phi+ \sin (\theta-\phi)} e^{-i \phi}w_1+w_2=1$\\

$L_{02}$ & $v_0,v_2$
& $\frac{\sin \theta}{\sin \phi+ \sin (\theta-\phi)}e^{i \phi}z_1+z_2=1$
& $\frac{\sin \theta}{\sin \phi+ \sin (\theta-\phi)} e^{-i \phi}w_1+e^{-i\theta}w_2=1$\\

$L_{03}$ & $v_0,v_3$
& $\frac{\sin \theta}{\sin \phi+ \sin (\theta-\phi)}e^{i \phi}z_1+e^{-i \theta}z_2=1$
& $w_1=0$\\

$L_{12}$ & $v_1,v_2$
& $z_1+z_2=1$
& $w_2=0$\\

$L_{23}$ & $v_2,v_3$
& $z_2=0$
& $w_1+e^{-i \theta}w_2=1$\\

$L_{13}$ & $v_1,v_3$
& $z_1+e^{-i \theta}z_2=1$
& $w_1+w_2=1$\\
\hline
\end{tabular}
\end{center}

With these equations, we can calculate the coordinates of the vertices by making the complex lines intersect or, equivalently, two pairs of points coalesce at the same time.
The first table will give us the $\textbf{z}$ coordinates of all the vertices, while the second one will give us their $\textbf{w}$ coordinates.

\begin{center}
\begin{tabular}{|c|c|c|c|}
\hline
Vertex 
& coordinate $z_1$ 
&  coordinate $z_2$ \\
\hline
$\textbf{z}_1$ 
& $0$ 
& $0$\\

$\textbf{z}_2$ 
& $ \frac{\sin \phi+ \sin(\theta - \phi)}{\sin \phi +e^{i \phi} \sin \theta}$
& $\frac{ e^{i \theta} \sin \phi}{\sin \phi + e^{i \phi} \sin \theta}$\\

$\textbf{z}_3$ 
& $ \frac{\sin \phi+ \sin(\theta - \phi)}{\sin (\theta +\phi)}$ 
& $0$ \\

$\textbf{z}_4$ 
& $ \frac{\sin \phi+ \sin(\theta - \phi)}{\sin (\theta +\phi)}$ 
& $\frac{\sin \phi (2 \cos \theta -1)}{\sin(\theta+\phi)}$ \\

$\textbf{z}_5$ 
& $ \frac{\sin \phi+ \sin(\theta - \phi)}{\sin (\theta +\phi)}$ 
& $ e^{i \theta}\frac{\sin \phi (2 \cos \theta -1)}{\sin(\theta+\phi)}$ \\

$\textbf{z}_6$ 
& $e^{-i\phi}\frac{\sin \theta}{\sin(\theta +\phi)}$
& $0$\\

$\textbf{z}_7$ 
&  $e^{-i\phi}\frac{\sin \theta}{\sin(\theta +\phi)}$ 
& $ 1-\frac{\sin^2 \theta}{\sin(\theta+\phi)(\sin \phi+ \sin(\theta-\phi))}$\\

$\textbf{z}_8$ 
& $e^{-i \phi} \frac{\sin \theta}{\sin(\theta+\phi)}$ 
& $e^{i \theta} \left( 1-\frac{\sin^2 \theta}{\sin(\theta+\phi)(\sin \phi+ \sin(\theta-\phi))}\right)$ \\

$\textbf{z}_9$ 
& $0$ 
& $\frac{\sin \phi}{\sin(\theta+\phi)}$ \\

$\textbf{z}_{10}$ 
& $\frac{\sin(\theta+\phi) -\sin \phi}{\sin(\theta+\phi)}$
& $\frac{\sin \phi}{\sin(\theta+\phi)}$ \\

$\textbf{z}_{11}$ 
& $e^{-i\phi}\frac{\sin \phi+\sin(\theta-\phi)}{\sin \theta} \left( 1- \frac{\sin \phi}{\sin (\theta+\phi)}\right)$
& $\frac{\sin \phi}{\sin(\theta+\phi)}$ \\

$\textbf{z}_{12}$ 
& $0$
& $e^{i \theta} \frac{\sin \phi}{\sin(\theta+\phi)}$ \\

$\textbf{z}_{13}$ 
& $\frac{\sin(\theta+\phi) -\sin \phi}{\sin(\theta+\phi)}$
& $e^{i \theta} \frac{\sin \phi}{\sin(\theta+\phi)}$\\

$\textbf{z}_{14}$ 
& $e^{-i\phi}\frac{\sin \phi+\sin(\theta-\phi)}{\sin \theta} \left( 1- \frac{\sin \phi}{\sin (\theta+\phi)}\right)$
& $e^{i \theta} \frac{\sin \phi}{\sin(\theta+\phi)}$\\
\hline
\end{tabular}
\end{center}

\begin{center}
\begin{tabular}{|c|c|c|c|}
\hline
Vertex 
& coordinate $w_1$ 
&  coordinate $w_2$ \\
\hline
$\textbf{z}_1$ 
& $\frac{\sin \phi + \sin (\theta - \phi)}{\sin \phi + e^{-i \phi} \sin \theta}$
& $\frac{ \sin \phi}{\sin \phi +e^{-i \phi} \sin \theta}$ \\

$\textbf{z}_2$ 
& $0$ 
& $0$ \\

$\textbf{z}_3$ 
& $\frac{\sin \phi + \sin (\theta - \phi)}{\sin(\theta+\phi)}$  
& $e^{i \theta} \frac{\sin \phi (2 \cos \theta -1)}{\sin(\theta+\phi)} $  \\

$\textbf{z}_4$ 
& $\frac{\sin \phi + \sin (\theta - \phi)}{\sin(\theta+\phi)}$  
& $0$\\

$\textbf{z}_5$ 
& $\frac{\sin \phi + \sin (\theta - \phi)}{\sin(\theta+\phi)}$  
& $\frac{\sin \phi (2 \cos \theta -1)}{\sin(\theta+\phi)}$\\

$\textbf{z}_6$ 
& $\frac{\sin(\theta+\phi) -\sin \phi}{\sin (\theta+\phi)}$ 
& $e^{i \theta} \frac{\sin \phi}{\sin(\theta+\phi)}$\\

$\textbf{z}_7$ 
& $e^{i \phi} \frac{\sin \phi+\sin(\theta-\phi)}{\sin \theta}\left( 1- \frac{\sin \phi}{\sin (\theta+\phi)}\right)$ 
&  $e^{i \theta} \frac{\sin \phi}{\sin(\theta+\phi)}$\\

$\textbf{z}_8$ 
& $0$
&  $e^{i \theta} \frac{\sin \phi}{\sin(\theta+\phi)}$\\

$\textbf{z}_9$ 
& $e^{i \phi} \frac{\sin \theta}{\sin(\theta+\phi)}$
& $ 1-\frac{\sin^2 \theta}{\sin(\theta+\phi)(\sin \phi+ \sin(\theta-\phi))}$ \\

$\textbf{z}_{10}$ 
& $e^{i \phi} \frac{\sin \theta}{\sin(\theta+\phi)}$
& $0$\\

$\textbf{z}_{11}$ 
& $e^{i \phi} \frac{\sin \theta}{\sin(\theta+\phi)}$
& $e^{i \theta} \left(1-\frac{\sin^2 \theta}{\sin(\theta+\phi)(\sin \phi+ \sin(\theta-\phi))}\right)$\\

$\textbf{z}_{12}$ 
& $e^{i \phi} \frac{\sin \phi+\sin(\theta-\phi)}{\sin \theta}\left( 1- \frac{\sin \phi}{\sin (\theta+\phi)}\right)$ 
& $\frac{\sin \phi}{\sin(\theta+\phi)}$\\

$\textbf{z}_{13}$ 
& $\frac{\sin(\theta+\phi)-\sin \phi}{\sin(\theta+\phi)}$
& $\frac{\sin \phi}{\sin(\theta+\phi)}$\\

$\textbf{z}_{14}$ 
& $0$
& $\frac{\sin \phi}{\sin(\theta+\phi)}$\\
\hline
\end{tabular}
\end{center}

These vertices present a symmetry given by the transformation $\iota$. 
In fact, as we can immediately verify on the coordinates in the table, the following lemma holds:

\begin{lemma}\label{iotaaction}
The isometry $\iota$ defined by \eqref{iota} has order 2 and acts on the vertices in the following way:
\emph{
\begin{align*}
\iota( \textbf{z}_1) &= \textbf{z}_2, 
& \iota( \textbf{z}_3)&= \textbf{z}_4,
& \iota( \textbf{z}_5)&= \textbf{z}_5, 
& \iota( \textbf{z}_6)&= \textbf{z}_{10}, \\
\iota( \textbf{z}_7)&= \textbf{z}_{11}, 
& \iota( \textbf{z}_8)&= \textbf{z}_9,
& \iota( \textbf{z}_{12})&= \textbf{z}_{14}, 
& \iota( \textbf{z}_{13})&= \textbf{z}_{13}.
\end{align*}}
\end{lemma} 

\subsection{The polyhedron and its sides}\label{sides}

In this section we will construct a polyhedron $D$ in complex hyperbolic space.
Later on, in Section \ref{poincare}, we will prove that this is a fundamental polyhedron for the group $\Gamma$ we are considering. 
The degenerate configurations of cone points on the sphere described in the previous section will indeed be the vertices of the polyhedron $D$. 

On the boundary of the polyhedron we have cells of different dimensions. 
The codimension 1 cells (3-dimensional cells) are called \emph{sides}. 
The 2-dimensional cells are called \emph{ridges} and the 1-dimensional are the \emph{edges}.
The \emph{vertices} are the 0-dimensional cells in the boundary of the polyhedron.
The sides of the polyhedron will be contained in bisectors, described in Section \ref{bisectors}. 

As we can easily see just by looking at the tables, if we consider one column of the first or second coordinates table (i.e. fixing one of $z_1,z_2,w_1,w_2$), most vertices have that particular coordinate either real or a real number multiplied by a unit complex number of the same argument along the column (respectively $e^{-i \phi}, e^{i \theta}, e^{i \phi},e^{i \theta}$).
More specifically, the only ones not following this rule are $\textbf{z}_1$ for the $\textbf{w}$-coordinates and $\textbf{z}_2$ for the $\textbf{z}$-coordinates.
This makes it natural to consider the portion of complex hyperbolic space consisting of all points with arguments of the coordinates included in the ranges bounded by these values. 
In fact, the two that, as we said, do not follow this rule, are still within the bounded ranges (even if strictly in the interior).
We hence define our polyhedron to be such region, in the following way:
\begin{equation} \label{polyhedron}
D= \left\{ \textbf{z}=P(\textbf{w}) \colon
\begin{array}{l l}
\arg(z_1) \in (-\phi,0), & \arg(z_2) \in (0,\theta),\\
\arg(w_1) \in (0,\phi), & \arg(w_2) \in (0,\theta)
\end{array} \right\}.
\end{equation}


The sides of the polyhedron will then be contained in bisectors, which are defined as in the following table.

\begin{center}
\begin{tabular}{|c|c|c|}
\hline
Bisector & Equation & Points in the bisector \\
\hline
$B(P)$ & $\im (z_1)=0$ 
& $\textbf{z}_1, \textbf{z}_3,\textbf{z}_4, \textbf{z}_5, \textbf{z}_{9}, \textbf{z}_{10}, \textbf{z}_{12}, \textbf{z}_{13}$ \\
$B(P^{-1})$ & $\im (w_1)=0 $
& $\textbf{z}_2, \textbf{z}_3,\textbf{z}_4, \textbf{z}_5, \textbf{z}_{6}, \textbf{z}_{8}, \textbf{z}_{13}, \textbf{z}_{14}$ \\
$B(J)$ & $\im (e^{i \phi} z_1)=0$ 
& $\textbf{z}_1, \textbf{z}_6,\textbf{z}_7, \textbf{z}_8, \textbf{z}_{9}, \textbf{z}_{11}, \textbf{z}_{12}, \textbf{z}_{14}$ \\
$B(J^{-1})$ & $\im (e^{-i \phi} w_1)=0 $
& $\textbf{z}_2, \textbf{z}_7,\textbf{z}_8, \textbf{z}_9, \textbf{z}_{10}, \textbf{z}_{11}, \textbf{z}_{12}, \textbf{z}_{14}$\\
$B(R_1)$ & $\im (z_2)=0$ 
& $\textbf{z}_1, \textbf{z}_3,\textbf{z}_4, \textbf{z}_6, \textbf{z}_{7}, \textbf{z}_{9}, \textbf{z}_{10}, \textbf{z}_{11} $\\
$B(R_1^{-1})$ & $\im (e^{-i \theta} z_2)=0$ 
& $\textbf{z}_1, \textbf{z}_3,\textbf{z}_5, \textbf{z}_6, \textbf{z}_{8}, \textbf{z}_{12}, \textbf{z}_{13}, \textbf{z}_{14}$\\
$B(R_2)$ & $\im (w_2)=0 $
& $\textbf{z}_2, \textbf{z}_4,\textbf{z}_5, \textbf{z}_9, \textbf{z}_{10}, \textbf{z}_{12}, \textbf{z}_{13}, \textbf{z}_{14}$\\
$B(R_2^{-1})$ & $\im (e^{-i \theta} w_2)=0 $
& $\textbf{z}_2, \textbf{z}_3,\textbf{z}_4, \textbf{z}_6, \textbf{z}_{7}, \textbf{z}_{8}, \textbf{z}_{10}, \textbf{z}_{11}$\\
\hline
\end{tabular}
\end{center}

The choice of the name of the bisectors has been made in such a way that the bisector $B(T)$ is sent by $T$ to the bisector $B(T^{-1})$, for $T \in \{ P,P^{-1},J, J^{-1}, R_1, R_1^{-1}, R_2, R_2^{-1}\}$. 

Finally, the following lemma proves that the subspaces defined are bisectors and that we named them following the convention just described. 

\begin{lemma}\label{lemmapoly}
In \emph{$\textbf{z}$} and \emph{$\textbf{w}$} coordinates, we have 
\begin{itemize}
\item $\im(z_1)<0$ if and only if \emph{
$\lvert \langle \textbf{z}, \textbf n_{*1} \rangle \rvert 
< \lvert \langle \textbf{z}, P^{-1}(\textbf n_{*3}) \rangle \rvert$},
\item $\im(w_1)>0$ if and only if \emph{
$\lvert \langle \textbf{w}, \textbf n_{*3} \rangle \rvert 
< \lvert \langle \textbf{w}, P(\textbf n_{*1}) \rangle \rvert$},
\item $\im(e^{i \phi} z_1)>0$ if and only if \emph{
$\lvert \langle \textbf{z}, \textbf n_{*0} \rangle \rvert 
< \lvert \langle \textbf{z}, J^{-1}(\textbf n_{*0}) \rangle \rvert$},
\item $\im(e^{-i \phi} w_1)<0$ if and only if \emph{
$\lvert \langle \textbf{w}, \textbf n_{*0} \rangle \rvert 
< \lvert \langle \textbf{w}, J(\textbf n_{*0}) \rangle \rvert$},
\item $\im(z_2)>0$ if and only if \emph{
$\lvert \langle \textbf{z}, \textbf n_{*2} \rangle \rvert 
< \lvert \langle \textbf{z}, R_1^{-1}(\textbf n_{*3}) \rangle \rvert$},
\item $\im(e^{-i \theta} z_2)<0$ if and only if \emph{
$\lvert \langle \textbf{z}, \textbf n_{*3} \rangle \rvert 
< \lvert \langle \textbf{z}, R_1(\textbf n_{*2}) \rangle \rvert$},
\item $\im(w_2)>0$ if and only if \emph{
$\lvert \langle \textbf{w}, \textbf n_{*1} \rangle \rvert 
< \lvert \langle \textbf{w}, R_2^{-1}(\textbf n_{*2}) \rangle \rvert$},
\item $\im(e^{-i \theta} w_2)<0$ if and only if \emph{
$\lvert \langle \textbf{w}, \textbf n_{*2} \rangle \rvert 
< \lvert \langle \textbf{w}, R_2(\textbf n_{*1}) \rangle \rvert$}.
\end{itemize}
\end{lemma}

The proof of the lemma goes simply by calculation. 
We will show just the first case and the others ones are done in similar ways. 
This is very similar as the proof of the equivalent statement for Livné lattices in \cite{livne}.

\begin{proof}
Let's take 
\[
\textbf n_{*1}=\begin{bmatrix}
e^{-i \phi} \frac{\sin \phi +\sin(\theta -\phi)}{\sin \theta} \\0 \\1
\end{bmatrix}
\]
and
\[
P^{-1}(\textbf n_{*3})=P^{-1}\left(\begin{bmatrix}
0 \\1 \\1
\end{bmatrix}\right)=
\begin{bmatrix}
e^{i \phi} \frac{\sin \phi +\sin(\theta -\phi)}{\sin \theta} \\0 \\1
\end{bmatrix}.
\]
It is immediate to verify that $\textbf n_{*1}$ is a vector normal to $L_{*1}$ and $\textbf n_{*3}$ is normal to $L_{*3}$.

Furthermore, we have 
\[
\lvert \langle \textbf{z}, \textbf n_{*1} \rangle \rvert=
\left\lvert -\sin \phi e^{i \phi}z_1+\frac{\sin \phi \sin \theta}{\sin(\theta+\phi)} \right\rvert
\]
and
\[
\lvert \langle \textbf{z}, P^{-1}(\textbf n_{*3}) \rangle \rvert=
\left\lvert -\sin \phi e^{-i \phi}z_1+\frac{\sin \phi \sin \theta}{\sin(\theta+\phi)} \right\rvert.
\]

So, to have 
$\lvert \langle \textbf{z}, \textbf n_{*1} \rangle \rvert < \lvert \langle \textbf{z}, P^{-1}(\textbf n_{*3}) \rangle \rvert$,
we need to have a point which verifies $-\re(e^{i \phi}z_1)<-\re(e^{-i \phi}z_1)$, which is equivalent to require that $\im(z_1)<0$.
\end{proof}

\begin{rk}\label{polydef}
By definition, a point is in the polyhedron $D$ if and only if it satisfies all the conditions on the left hand side in the lemma.
\end{rk}

As we mentioned, the lemma explains the name given to the bisectors. 
In fact, for example the bisector $B(P)$ is, by definition, given by $\im(z_1)=0$, which corresponds, by the lemma, to the points satisfying
\[
\lvert \langle \textbf{z}, \textbf n_{*1} \rangle \rvert 
= \lvert \langle \textbf{z}, P^{-1}(\textbf n_{*3}) \rangle \rvert.
\]
Applying $P$ to both sides of the equality, we get a point in the bisector defined by 
\[
\lvert \langle \textbf{w}, P(\textbf n_{*1}) \rangle \rvert 
=\lvert \langle \textbf{w}, \textbf n_{*3} \rangle \rvert,
\]
which is indeed $B(P^{-1})$.
The sides of the polyhedron are contained in the bisectors. 
We will define the side $S(T)$ to be the one contained in the bisector $B(T)$ and it will be obtained by intersecting it with $\overline D$. 

\subsection{Ridges and edges of the polyhedron}

\subsubsection{Useful inequalities}

In this section we will present some trigonometric inequalities that will be used all through the following sections. 
Some of them are equivalent to the inequalities found in \cite{livne} and \cite{boadiparker}.

\begin{lemma}\label{lemmasquare}
Let \emph{$\textbf{z} \in \hc$}. 
Then
\begin{align*}
\lvert z_1 \rvert^2, \lvert w_1 \rvert^2 &\leq \frac{\sin \phi+\sin(\theta-\phi)}{\sin(\theta+\phi)}, 
& \text{and} & &
\lvert z_2 \rvert^2, \lvert w_2 \rvert^2 &\leq \frac{\sin \phi}{\sin(\theta+\phi)}.
\end{align*}
\end{lemma}
The proof is straightforward considering the condition on the area for points of $\hc$, in a similar spirit as the inequalities in \cite{boadiparker}. 

The second useful lemma is the following, divided in two cases according to the values of $p$ and $l$, the latter as defined in Section \ref{ldef} in terms of $p$ and $k$.

\begin{lemma}\label{lemmaone}
Let \emph{$\textbf{z} \in \hc$}. 
Then we have 
\begin{enumerate}
\item If $p > 6$, then 
\[\lvert z_1 \rvert, \lvert w_1 \rvert < 1,\] 
\item If $l \geq 0$, then 
\[\lvert z_2 \rvert, \lvert w_2 \rvert \leq 1.\]
\end{enumerate}
\end{lemma}

\begin{proof}
Obviously if the square of the modulus of a coordinate is smaller than 1, so is the modulus of the coordinate itself.
We then just need to prove that the square of such moduli are smaller than 1.
By the previous Lemma \ref{lemmasquare}, we have 
\[
\lvert z_1 \rvert^2, \lvert w_1 \rvert^2 \leq \frac{\sin \phi+\sin(\theta-\phi)}{\sin(\theta+\phi)}.
\]

For the first part, we then just need to show that 
\[
\frac{\sin \phi+\sin(\theta-\phi)}{\sin(\theta+\phi)} < 1.
\]
But we have 
\[
\frac{\sin \phi+\sin(\theta-\phi)}{\sin(\theta+\phi)}=
\frac{\sin \phi-2\sin \phi \cos \theta}{\sin(\theta+\phi)} +1=
1- \frac{\sin \phi}{\sin(\theta+\phi)} (2\cos \theta -1) < 1,
\]
where the last inequality comes from the fact that $\frac{\sin \phi}{\sin(\theta+\phi)} (2\cos \theta -1)$ is positive when $0<\theta < \frac{\pi}{3}$.
Since $\theta=\frac{2\pi}{p}$, this is the case when $p > 6$, as required.

For the second inequality, by the same Lemma \ref{lemmasquare}, we just need to prove that 
\[
\frac{\sin \phi}{\sin(\theta+\phi)} \leq 1.
\]
But this is true as long as $\sin \phi \leq \sin(\theta+\phi)$. 
Moreover, this condition is equivalent to the statement 
\[
\theta+\phi \leq \pi- \phi \Longleftrightarrow 
0 \leq \pi-2\phi-\theta \Longleftrightarrow
0 \leq \frac{2 \pi}{2} -\frac{2 \pi}{k}-\frac{2\pi}{p} \Longleftrightarrow,
0 \leq l
\]
where the second equivalence comes from the fact that $\theta=\frac{2\pi}{p}$, $\phi=\frac{\pi}{k}$ and $\frac{1}{l}=\frac{1}{2}-\frac{1}{p}-\frac{1}{k}$.
This implies that the condition in the second inequality corresponds to $l \geq 0$ and hence we are done. 
\end{proof}

\subsubsection{Ridges}\label{ridges}

In this section we will present the dimension 2 facets of our polyhedron, i.e. the ridges. 
We will divide the ridges in two types. 
The first type of ridge is obtained by intersecting two bisectors containing either the vertex $\textbf{z}_1$ or $\textbf{z}_2$ in their intersection. 
We will get from these intersections some pentagonal ridges and some triangular ones. 
The former will be contained in Lagrangian planes, while the latter are contained in complex lines. 

The second type of ridge comes from the intersection of bisectors defined by one condition on the $\textbf{z}$-coordinates and one on the $\textbf{w}$-coordinates.
We will again get some triangular ridges, contained in complex lines, but this time we will also get hexagonal ridges, contained in Giraud discs.

We will name the ridges according to the following convention. 
The ridge named $F(T,S)$, for $T,S \in \{ P,P^{-1},J, J^{-1}, R_1, R_1^{-1}, R_2, R_2^{-1}\}$, will be the ridge contained in the intersection of the bisector $B(T)$ and $B(S)$. 

The following table summarizes the ridges of the first type.
In the first group there are ridges in the intersection of two bisectors, both containing the vertex $\textbf{z}_1$ (in other words, bisectors defined by conditions on the $\textbf{z}$-coordinates).
In the second group are ridges contained in two bisectors defined by conditions on the $\textbf{w}$-coordinates.
The last column says if the ridge is contained in a complex line, marked with S as it is a common slice of the two bisector, or in a Lagrangian plane, marked with M because it is a common meridian of the two bisectors.

\begin{center}
\begin{tabular}{|c|c|c|c|}
\hline 
Ridge  & Vertices in the ridge & Coordinates &\\
\hline
$F(P,J)$ & $\textbf{z}_{1}, \textbf{z}_{9}, \textbf{z}_{12}$ & $z_1=0$& S\\
$F(R_1,R_1^{-1})$ & $\textbf{z}_{1}, \textbf{z}_{3}, \textbf{z}_{6}$ & $z_2=0$& S \\
$F(P,R_1)$ & $\textbf{z}_{1}, \textbf{z}_{3}, \textbf{z}_{4}, \textbf{z}_{9}, \textbf{z}_{10}$ & $\im(z_1)=\im(z_2)=0$ & M\\
$F(P,R_1^{-1})$ & $\textbf{z}_{1}, \textbf{z}_{3}, \textbf{z}_{5}, \textbf{z}_{12}, \textbf{z}_{13}$ & $\im(z_1)=\im(e^{-i\theta}z_2)=0$ & M \\
$F(J,R_1)$ & $\textbf{z}_{1}, \textbf{z}_{6}, \textbf{z}_{7}, \textbf{z}_{9}, \textbf{z}_{11}$ & $\im(e^{i\phi}z_1)=\im(z_2)=0$ & M\\
$F(J,R_1^{-1})$ & $\textbf{z}_{1}, \textbf{z}_{6}, \textbf{z}_{8}, \textbf{z}_{12}, \textbf{z}_{14}$ & $\im(e^{i\phi}z_1)=\im(e^{-i\theta}z_2)=0$ & M \\
\hline
$F(P^{-1},J^{-1})$ & $\textbf{z}_{2}, \textbf{z}_8, \textbf{z}_{14}$ & $w_1=0$ & S\\
$F(R_2,R_2^{-1})$ & $\textbf{z}_{2}, \textbf{z}_4, \textbf{z}_{10}$ & $w_2=0$ & S \\
$F(P^{-1},R_2)$ & $\textbf{z}_{2}, \textbf{z}_{4}, \textbf{z}_{5}, \textbf{z}_{13}, \textbf{z}_{14}$ & $\im(w_1)=\im(w_2)=0$ & M\\
$F(P^{-1},R_2^{-1})$ & $\textbf{z}_{2}, \textbf{z}_{3}, \textbf{z}_{4}, \textbf{z}_{6}, \textbf{z}_{8}$ & $\im(w_1)=\im(e^{-i\theta}w_2)=0$ & M\\
$F(J^{-1},R_2)$ & $\textbf{z}_{2}, \textbf{z}_{9}, \textbf{z}_{10}, \textbf{z}_{12}, \textbf{z}_{14}$ & $\im(e^{-i\phi}w_1)=\im(w_2)=0$ & M\\
$F(J^{-1},R_2^{-1})$ & $\textbf{z}_{2}, \textbf{z}_{7}, \textbf{z}_{8}, \textbf{z}_{10}, \textbf{z}_{11}$ & $\im(e^{-i\phi}w_1)=\im(e^{-i\theta}z_2)=0$ & M\\
\hline
\end{tabular}
\end{center}

The second type of ridges are the ones not containing the vertices $\textbf{z}_1$ or $\textbf{z}_2$ and they are listed in the following table.
In this case the ridges are contained either in a Giraud disc or in a complex line.
The last column of the table will hence have a G in the first case and, as before, an S in the latter. 

\begin{center}
\begin{tabular}{|c|c|c|c|}
\hline
Ridge  & Vertices in the ridge & Coordinates &\\
\hline
$F(P,R_2)$ & $\textbf{z}_4, \textbf{z}_5, \textbf{z}_9, \textbf{z}_{10}, \textbf{z}_{12}, \textbf{z}_{13}$ & $\im (z_1)=\im(w_2)=0$ & G \\
$F(J,J^{-1})$ & $\textbf{z}_7, \textbf{z}_8, \textbf{z}_9, \textbf{z}_{11}, \textbf{z}_{12}, \textbf{z}_{14}$ & $\im(e^{i\phi} z_1)=\im(e^{-i \phi} w_1)=0$ & G \\
$F(R_1,R_2^{-1})$ & $\textbf{z}_3, \textbf{z}_4, \textbf{z}_6, \textbf{z}_7, \textbf{z}_{10}, \textbf{z}_{11}$ & $\im(z_2)=\im(e^{-i\theta}w_2)=0$ & G \\
$F(R_1^{-1},P^{-1})$ & $\textbf{z}_3,\textbf{z}_5,\textbf{z}_6, \textbf{z}_8, \textbf{z}_{13}, \textbf{z}_{14}$ & $\im(e^{-i\theta} z_2)=\im(w_1)=0$ & G \\
\hline
$F(P,P^{-1})$ & $\textbf{z}_3, \textbf{z}_4,\textbf{z}_5$ & $\im(z_1)=\im(w_1)=0$ & S \\
$F(J,R_2^{-1})$ & $\textbf{z}_6, \textbf{z}_7,\textbf{z}_8$ & $\im(e^{i\phi}z_1)=\im(e^{-i\theta}w_2)=0$ & S \\
$F(R_1,J^{-1})$ & $\textbf{z}_9, \textbf{z}_{10},\textbf{z}_{11}$ & $\im(z_2)=\im(e^{-i\phi}w_1)=0$ & S \\
$F(R_1^{-1},R_2)$ & $\textbf{z}_{12}, \textbf{z}_{13},\textbf{z}_{14}$ & $\im(e^{-i\theta}z_2)=\im(w_2)=0$ & S \\
\hline
\end{tabular}
\end{center}

From now on the ridges contained in a common slice will be called S-ridges, the ones contained in a meridian will be the M-ridges and the ones contained in a Giraud disk will be the G-ridges.

\subsubsection{Edges}
We so far discussed most facets of the polyhedron: the vertices, the ridges, the sides. 
In this section we will present the last missing ones, the 1-dimensional facets of $D$, called edges.  
The edge between two vertices $\textbf{z}_i$ and $\textbf{z}_j$ will be denoted by $\gamma_{i,j}=\gamma_{j,i}$.
The edges of the polyhedron $D$ arise as 1-dimensional intersection of three or more sides. 
In the following table we will list them, pointing out in which ridges they are contained. 

\begin{center}
\begin{longtable}{|c|c|c|c|c|c|}
\hline
Edge & S-ridge &M-ridge & M-ridge & G-ridge & G-ridge \\
\hline
$\gamma_{1,3}$ & $F(R_1,R_1^{-1})$ & $F(P,R_1)$ & $F(P,R_1^{-1})$ && \\
$\gamma_{1,6}$ & $F(R_1,R_1^{-1})$ & $F(J,R_1)$ & $F(J,R_1^{-1})$ && \\
$\gamma_{1,9}$ & $F(P,J)$ & $F(P,R_1)$ & $F(J,R_1)$ && \\
$\gamma_{1,12}$ & $F(P,J)$ & $F(P,R_1^{-1})$ & $F(J,R_1^{-1})$ && \\
$\gamma_{2,4}$ & $F(R_2,R_2^{-1})$ & $F(P^{-1},R_2^{-1})$ & $F(P^{-1},R_2)$ && \\
$\gamma_{2,8}$ & $F(P^{-1}, J^{-1})$ & $F(P^{-1},R_2^{-1})$ & $F(J^{-1},R_2^{-1})$ && \\
$\gamma_{2,10}$ & $F(R_2,R_2^{-1})$ & $F(J^{-1},R_2)$ & $F(J^{-1},R_2^{-1})$ && \\
$\gamma_{2,14}$ & $F(P^{-1},J^{-1})$ & $F(P^{-1},R_2)$ & $F(J^{-1},R_2)$ && \\
$\gamma_{5,13}$ && $F(P,R_1^{-1})$ & $F(P^{-1},R_2)$ & $F(P,R_2)$ & $F(R_1^{-1},P^{-1})$ \\
$\gamma_{7,11}$ && $F(J,R_1)$ & $F(J^{-1}, R_2^{-1})$ & $F(J,J^{-1})$ & $F(R_1,R_2^{-1})$ \\
$\gamma_{9,10}$ & $F(R_1,J^{-1})$ & $F(P,R_1)$ & $F(J^{-1},R_2)$ & $F(P,R_2)$ & \\
$\gamma_{3,4}$ & $F(P,P^{-1})$ & $F(P,R_1)$ & $F(P^{-1},R_2^{-1})$ & $F(R_1,R_2^{-1})$ & \\
$\gamma_{6,8}$ & $F(J,R_2^{-1})$ & $F(J,R_1^{-1}$ & $F(P^{-1},R_2)$ & $F(R_1^{-1},P^{-1})$ & \\
$\gamma_{12,14}$ & $F(R_1^{-1},R_2)$ & $F(J,R_1^{-1})$ & $F(J^{-1},R_2)$ & $F(J,J^{-1})$ & \\
$\gamma_{4,10}$ & $F(R_2,R_2^{-1})$ & $F(P,R_1)$ && $F(P,R_2)$ & $F(R_1,R_2^{-1})$ \\
$\gamma_{8,14}$ & $F(P^{-1},J^{-1})$ &  $F(J,R_1^{-1})$ && $F(J,J^{-1})$ & $F(R_1^{-1},P^{-1})$ \\
$\gamma_{9,12}$ & $F(P,J)$ & $F(J^{-1},R_2)$ && $F(P,R_2)$ & $F(J,J^{-1})$ \\
$\gamma_{3,6}$ & $F(R_1,R_1^{-1})$ & $F(P^{-1},R_2^{-1})$ && $F(R_1,R_2^{-1})$ & $F(R_1^{-1},P^{-1})$ \\
$\gamma_{13,14}$ & $F(R_1^{-1},R_2)$ & $F(P^{-1},R_2)$ && $F(R_1^{-1},P^{-1})$ & \\
$\gamma_{12,13}$ & $F(R_1^{-1},R_2)$ & $F(P,R_1^{-1})$ && $F(P,R_2)$ & \\
$\gamma_{10,11}$ & $F(R_1,J^{-1})$ & $F(J^{-1},R_2^{-1})$ && $F(R_1,R_2^{-1})$ & \\
$\gamma_{9,11}$ & $F(R_1,J^{-1})$ & $F(J,R_1)$ && $F(J,J^{-1})$ & \\
$\gamma_{7,8}$ & $F(J,R_2^{-1})$ & $F(J^{-1},R_2^{-1})$ && $F(J,J^{-1})$ & \\
$\gamma_{6,7}$ & $F(J,R_2^{-1})$ & $F(J,R_1)$ && $F(R_1,R_2^{-1})$ & \\
$\gamma_{4,5}$ & $F(P,P^{-1})$ & $F(P^{-1},R_2)$ && $F(P,R_2)$ & \\
$\gamma_{3,5}$ & $F(P,P^{-1})$ & $F(P,R_1^{-1})$ && $F(R_1^{-1},P^{-1})$ & \\
\hline
\end{longtable}
\end{center}

The edges verify the following proposition:
\begin{prop}
Each edge $\gamma_{i,j}$ of the polyhedron is a geodesic segment joining the two vertices \emph{$\textbf{z}_i$} and \emph{$\textbf{z}_j$}.
\end{prop}

\begin{proof}
We claim that each edge is contained in the common intersection of at least two totally geodesic subspaces of two bisectors. 
This implies that such edge is a geodesic arc.
Remember, from Section \ref{bisectors}, that slices and meridians are totally geodesic subspaces of bisectors.

To prove the claim, let us consider for each edge the ridges it is contained in, as in the previous table.
Just looking at the list we can easily remark the following information:
\begin{itemize}
\item Each edge containing either $\textbf{z}_1$ or $\textbf{z}_2$ is contained in two M-ridges and one S-ridge;
\item Two edges, namely $\gamma_{7,11}$ and $\gamma_{5,13}$, are contained in two M-ridges and two G-ridges;
\item All other edges are contained in an S-ridge, an M-ridge and a G-ridge; some of them are contained also in one more ridge, that is either an M-ridge or a G-ridge.
\end{itemize}
\end{proof}

\begin{rk}
For the edges containing either $\textbf{z}_1$ or $\textbf{z}_2$ we have additional information. 
Each of these edges is contained in two M-ridges of the same bisector.
This implies that such edges are in the spine of the bisectors. 
\end{rk}

\subsubsection{Other bisector intersections}

We will now analyse all the other intersections between pairs of bisectors, to show that the ones we listed in Section \ref{ridges} are the only possible ridges. 
We will first analyse certain bisector intersections which are made of the union of two edges of the polyhedron. 
In all the cases there will be three vertices inside the intersection and we will prove that the intersection actually consist in each case of the union of the only two edges connecting two of these points to a central one. 
We remark that we are always considering the parts of the intersection that are inside or on the boundary of our polyhedron.

The proofs for the following propositions go on the lines of the ones that can be found in the appendix of \cite{livne} and in \cite{boadiparker}.
For each case we will give one example and the others will be done in the exact same way.

\begin{prop}
The following bisector intersections consist of the union of two edges: 
\begin{align*}
B(P) \cap B(J^{-1})&= \gamma_{10,9} \cup \gamma_{9,12}, &
B(J^{-1}) \cap B(R_1^{-1}) &= \gamma_{8,14} \cup \gamma_{14,12}, \\
B(P) \cap B(R_2^{-1})&= \gamma_{3,4} \cup \gamma_{4,10}, &
B(J) \cap B(R_2)&= \gamma_{9,12} \cup \gamma_{12,14}, \\
B(R_1) \cap B(R_2) &= \gamma_{4,10} \cup \gamma_{10,9}, &
B(J) \cap B(P^{-1}) &= \gamma_{6,8} \cup \gamma_{8,14}, \\
B(R_1) \cap B(P^{-1}) &= \gamma_{4,3} \cup \gamma_{3,6}, &
B(R_1^{-1}) \cap B(R_2^{-1}) &= \gamma_{3,6} \cup \gamma_{6,8}.
\end{align*}
\end{prop}

\begin{proof}
Let us consider the intersection $B(P) \cap B(J^{-1})$ and a point $\textbf{z} \in D$ in it. 
It contains the three vertices $\textbf{z}_9, \textbf{z}_{10}, \textbf{z}_{12}$. 
From the table of the sides we can easily see that all three belong also to $B(R_2)$, which implies that also $\textbf{z}$ is in $B(R_2)$. 

Now, the coordinates of a point in $B(P) \cap B(J^{-1})$, satisfy $w_1=ue^{i \phi}$ and $z_1=x$. 
Using then the formulas for $w_1$ as given in (\ref{w1}) and the one for $z_1$ as given in (\ref{z1}), we have
\begin{align*}
ue^{i \phi} &=\frac{-xe^{i \phi} \sin \theta -z_2 e^{-i \theta} (\sin \phi+ \sin (\theta - \phi)) + \sin \phi+ \sin(\theta - \phi)}{-x \sin(\theta + \phi) -z_2e^{-i \theta} \sin(\theta+\phi)+\sin \phi+ e^{-i \phi} \sin \theta}, \\
x&= \frac{-u \sin \theta -w_2 (\sin \phi+\sin(\theta-\phi))+\sin \phi +\sin(\theta - \phi)}{-u \sin(\theta+\phi)e^{i \phi} - \sin(\theta +\phi)w_2+\sin \phi +e^{i \phi} \sin \theta}.
\end{align*}
We can solve the equations and find formulas for $z_2$ and $w_2$. 
They will be as follows. 
\begin{align*}
z_2 &= e^{i \theta} \frac{xu e^{i \phi} \sin(\theta + \phi) -u (e^{i \phi} \sin \phi+\sin \theta) - xe^{i \phi} \sin \theta + \sin \phi +\sin(\theta -\phi)}{\sin(\theta - \phi)+\sin \phi -ue^{i \phi}\sin(\theta+\phi)}, \\
w_2 &= \frac{xue^{i \phi} \sin(\theta+\phi) -x(e^{i \phi} \sin \theta +\sin \phi)-u \sin \theta +\sin \phi+\sin(\theta - \phi)}{\sin \phi+\sin(\theta - \phi) -x \sin(\theta + \phi)}.
\end{align*}

The condition for $\textbf{z}$ to be also in $B(R_2)$ gives us that $\im (w_2)=0$. 
We can hence apply this to the expression for $w_2$ that we just found and we get
\[
\im w_2 = \frac{x \sin \phi (u \sin(\theta+ \phi)-\sin \theta)}{\sin \phi+\sin(\theta - \phi) -x \sin(\theta + \phi)}=0.
\]

In order for this to be true we need the numerator to be 0 and since $\sin \phi \neq 0$ for our values of $\phi$, then we have either 
\begin{align*}
x&=0, & &\text{ or } & &u \sin(\theta+ \phi)-\sin \theta=0 \\
&&&& &\Longleftrightarrow u= \frac{\sin \theta}{\sin(\theta+\phi)}.
\end{align*}

This means that either $z_1=0$ or $w_1= e^{i \phi} \frac{\sin \theta}{\sin(	\theta+\phi)}$. 
In the first case, the condition implies that $\textbf{z} \in B(J)$ too, so it is on the edge $\gamma_{9,12}$.
In the second case the condition implies that we are on the line $L_{*3}$. 
But then, in the table defining the lines, we can read the $\textbf{z}$-coordinates of such lines and see that this implies that $z_2 \in \R$. 
Then $\textbf{z} \in B(R_1)$ and hence we are on the edge $\gamma_{10,9}$.
\end{proof}

In some of the ridges contained in a complex line the intersection actually consists of the union of a triangle, which is the ridge itself, and an extra edge connected to one of the vertices of the ridge and not belonging to it. 
We will now see this for the remaining intersections. The proposition will state that if we have a point in the bisector intersection, but not belonging to the complex line containing the ridge, then it is on an edge with one vertex on the ridge and one outside.

\begin{prop}
\begin{itemize}
The bisectors verify:
\item A point \emph{$\textbf{z}$} in the bisectors intersection $B(P) \cap B(P^{-1})$, with $z_1 \neq \frac{\sin \phi +\sin (\theta - \phi)}{\sin (\theta+\phi)}, w_1 \neq \frac{\sin \phi +\sin (\theta - \phi)}{\sin (\theta+\phi)}$, belongs to the edge $\gamma_{5,13}$.
\item A point \emph{$\textbf{z}$} in the bisectors intersection $B(J) \cap B(R_2^{-1})$, with $z_1 \neq e^{-i \phi} \frac{\sin \theta}{\sin (\theta+\phi)}$ and $ w_2 \neq e^{i \theta} \frac{\sin \phi}{\sin (\theta+\phi)}$, belongs to the edge $\gamma_{7,11}$.
\item Moreover, a point \emph{$\textbf{z}$} in the bisectors intersection $B(R_2) \cap B(R_1^{-1})$, with $z_2 \neq e^{i \theta}\frac{\sin \phi}{\sin (\theta+\phi)}$ and $w_2 \neq \frac{\sin \phi}{\sin (\theta+\phi)}$, belongs to the edge $\gamma_{5,13}$.
\item Finally, a point \emph{$\textbf{z}$} in the bisectors intersection $B(R_1) \cap B(J^{-1})$, with $z_2 \neq \frac{\sin \phi}{\sin (\theta+\phi)}$ and $w_1 \neq e^{i \phi} \frac{\sin \theta}{\sin (\theta+\phi)}$, belongs to the edge $\gamma_{7,11}$.
\end{itemize}
\end{prop}

\begin{proof}
Take a point $\textbf{z} \in B(P) \cap B(P^{-1})$.
The condition $z_1 \neq \frac{\sin \phi +\sin (\theta - \phi)}{\sin (\theta+\phi)}, w_1 \neq \frac{\sin \phi +\sin (\theta - \phi)}{\sin (\theta+\phi)}$ means that we are not on the line $L_{*0}$, so we are out of the triangular ridge of vertices $\textbf{z}_3, \textbf{z}_4, \textbf{z}_5$. 
Since we are on the intersection $B(P) \cap B(P^{-1})$, then both $z_1$ and $w_1$ have to be real. 
We will hence write $z_1=x$ and $w_1=u$. 
The conditions in the hypothesis implies that we have $x \neq \frac{\sin \phi +\sin (\theta - \phi)}{\sin (\theta+\phi)}, u \neq \frac{\sin \phi +\sin (\theta - \phi)}{\sin (\theta+\phi)}$. 

Using the formulas for $z_1$ in terms of $w_1$ and $w_2$ as given by (\ref{z1}) and the one for $w_1$ in terms of $z_1$ and $z_2$ as given in the formula (\ref{w1}), we can write:
\begin{align*}
x&= \frac{-ue^{-i\phi} \sin \theta -w_2 (\sin \phi +\sin (\theta-\phi))+ \sin \phi +\sin (\theta - \phi)}{-u \sin (\theta+\phi)-w_2 \sin (\theta+\phi) +\sin \phi +e^{i \phi} \sin \theta}, \\
u&= \frac{-xe^{i \phi} \sin \theta -z_2 e^{-i \theta}(\sin \phi +\sin (\theta - \phi))+\sin \phi +\sin (\theta - \phi)}{-x \sin (\theta +\phi)-z_2 e^{-i \theta}\sin(\theta + \phi)+ \sin \phi+ e^{-i \phi}\sin \theta}.
\end{align*}
By solving the equations we can find formulas for $z_2$ and $w_2$ and we get the following:
\begin{align*}
w_2&= \frac{xu \sin(\theta+\phi) -x(\sin \phi +e^{i \phi} \sin \theta)-ue^{-i \phi} \sin \theta +\sin \phi +\sin(\theta- \phi)}{\sin \phi+ \sin (\theta - \phi) -x \sin (\theta +\phi)}, \\
z_2&=e^{i \theta} \frac{xu \sin(\theta+\phi) -xe^{i \phi} \sin \theta -u(\sin \phi+e^{-i \phi}\sin \theta)+\sin \phi+\sin(\theta - \phi)}{\sin \phi+ \sin (\theta - \phi)-u \sin (\theta+ \phi)}.
\end{align*}
Taking only the imaginary part of both expressions, we have 
\begin{align*}
0 &\geq \im (e^{-i \theta} z_2)=\frac{\sin \theta \sin \phi (u-x)}{\sin \phi +\sin (\theta - \phi) -u \sin(\theta+ \phi)}, \\
0 &\leq \im (w_2) = \frac{\sin \theta \sin \phi (u-x)}{\sin \phi +\sin (\theta - \phi) -x \sin(\theta+ \phi)},
\end{align*}
where the first inequalities come from Lemma \ref{lemmapoly}, which holds because we are talking about points on the boundary of the polyhedron. 

We claim that the two quantities on the right side have same sign. 
In this case, since by the lemma they also need to have opposite sign, they must both be 0. 
But this means that the initial point must have also $\im (e^{-i \theta} z_2)=\im (w_2)=0$ and hence be in the bisectors $B(R_1^{-1})$ and $B(R_2)$, which implies that it is on the edge $\gamma_{5,13}$.

To show that the two quantities have the same sign, it is enough to show that the two denominators have same sign. 
The other informations we have about points in the polyhedron, is that $\im (z_2) \geq 0$ and $\im (e^{-i \theta} w_2) \leq 0$. 
From the second inequality, we can write
\[
0 \geq \im (e^{-i \theta} w_2)=
\frac{\sin \phi (u-1)(\sin \phi+\sin(\theta-\phi)-x \sin (\theta+\phi))}{\sin\phi+\sin(\theta-\phi) -u \sin(\theta+\phi)}.
\]
Since $\phi$ is positive and smaller than $\pi$, we have that $\sin \phi$ is positive. 
Also, by the first part of Lemma \ref{lemmaone}, $u-1$ is negative. 
For the whole expression to remain negative, the rest of it must then be positive and hence
\[
\frac{\sin \phi+ \sin(\theta-\phi) -x\sin (\theta+\phi)}{\sin \phi+\sin(\theta-\phi) -u \sin(\theta+\phi)} \geq 0.
\]
But this implies that the numerator and the denominator must have the same sign and this concludes our proof.
\end{proof}

In the proof of this proposition, we use Lemma \ref{lemmaone}. 
In Section \ref{degenerate} it will be clear why only for the values in the lemma that precise analysis of bisectors intersection makes sense, due to the collapsing of some ridges.

\section{Main theorem}\label{mainthm}

In this section, we will use the Poincaré polyhedron theorem to prove that $\Gamma=\langle R_1,R_2,A_1 \rangle$ is discrete, give a presentation for it and prove that $D$ constructed in the previous sections is its fundamental domain.
More precisely, we will prove the following:

\begin{theo} \label{main}
Let $\Gamma$ be the subgroup of $PU(H)$ characterised by $p$ and $k$ as explained in Section \ref{list} and such that the two parameters have any of the values in Table \ref{tablelist}. 
Then the polyhedron $D$ of the previous section is a fundamental domain for $\Gamma$, up to making some vertices collapse according to the following rule:
\begin{center}
\begin{tabular}{|m{3cm}|m{3.5cm}|m{6.8cm}|}
\hline
Value of $p$ & Value of $k$ & Fundamental polyhedron \\
\hline
$0<p\leq 6$ \newline $(d<0)$
& $k \leq \frac{2p}{p-2}$ \newline $(l<0)$ \newline (large phase shift)
& The polyhedron $D$ constructed in Section \ref{constr} with triples of vertices
$\textbf z_3, \textbf z_4, \textbf z_5$;  
$\textbf z_6, \textbf z_7, \textbf z_8$;  
$\textbf z_9, \textbf z_{10}, \textbf z_{11}$ and 
$\textbf z_{12}, \textbf z_{13}, \textbf z_{14}$ each collapsed to a single vertex is a fundamental domain.
This is the same as the polyhedron constructed in \cite{boadiparker}.  \\
\hline
$0<p\leq 6$ \newline $(d<0)$
& $k> \frac{2p}{p-2}$ \newline $(l>0)$ \newline (small phase shift)
& The polyhedron $D$ constructed in Section \ref{constr} with triples of vertices
$\textbf z_3, \textbf z_4, \textbf z_5$ each collapsed to a single vertex is a fundamental domain.
This is the same polyhedron obtained in \cite{type2}, as we will explain in Section \ref{type2} \\
\hline
$p>6$ \newline $(d>0)$
& $k \leq \frac{2p}{p-2}$ \newline $(l<0)$ \newline (large phase shift)
& The polyhedron $D$ constructed in Section \ref{constr} with triples of vertices 
$\textbf z_6, \textbf z_7, \textbf z_8$;  
$\textbf z_9, \textbf z_{10}, \textbf z_{11}$ and 
$\textbf z_{12}, \textbf z_{13}, \textbf z_{14}$ each collapsed to a single vertex is a fundamental domain.
This is the same as the polyhedron constructed in \cite{livne}.  \\
\hline
$p>6$ \newline $(d>0)$ 
& $k> \frac{2p}{p-2}$ \newline $(l>0)$ \newline (small phase shift)
& The polyhedron $D$ constructed in Section \ref{constr} is a fundamental domain. \\
\hline
\end{tabular}
\end{center}
\end{theo}

The table in the theorem is strictly related to Table \ref{tablelist}.
The first three groups, in fact, correspond exactly to the values of the Deligne-Mostow lattices of first, second and third (Livné lattices) type presented in the table. 
Lattices of the fourth and fifth type are in the fourth line of the table in the theorem.

\begin{rk} 
The condition $k \lesseqgtr \frac{2p}{p-2}$ is equivalent to saying that the phase shift parameter, as described in Section \ref{list}, is smaller or greater than $\frac{1}{2}- \frac{1}{p}$. 
\end{rk}

We also remark that the equality cases have to be treated a bit more carefully. 
For $p=6$ the vertex obtained collapsing $\textbf z_3, \textbf z_4, \textbf z_5$ is on the boundary of the complex hyperbolic space. 
These values are discussed in \cite{boadiparker} and can be included in the case of the lower values. 
The same discussion is true for the critical value of $k$ and the first group is the only case where such an equality actually holds. 

\subsection{Group presentations and Euler characteristic}

To prove Theorem \ref{main} we will use the Poincaré polyhedron theorem.
Its power lies not only in the fact that it allows to prove that $D$ is a fundamental domain for $\Gamma$, but because it also gives a presentation for the group. 

\begin{theo}\label{presentation}
Suppose $(p,k)$ is one of the pairs in Table \ref{tablelist}. 
Then the group $\Gamma$ generated by the side pairing maps of $D$, i.e. $P,J,R_1,R_2$ as described has presentation 
\[
\Gamma=\left\langle J,P,R_1,R_2 \colon
\begin{array}{l l}
J^3=P^{3d}=R_1^p=R_2^p=(P^{-1}J)^k=(R_2R_1J)^l=I, \\
R_2=P R_1P^{-1}=JR_1J^{-1}, \ P=R_1R_2
\end{array} \right\rangle,
\]
with each relation in the first line holding only when the order of the map is positive and finite.
\end{theo}

A proof of this theorem comes out automatically while using the Poincaré polyhedron theorem  to prove Theorem \ref{main} and is given in Section \ref{cyclerel}.

To conclude this section, we calculate the orbifold Euler characteristic $\chi( \hc / \Gamma)$.
The standard Euler characteristic is calculated taking the alternating sum of the number of cells of each dimension.
As explained in \cite{survey}, the orbifold Euler characteristic is calculated similarly, with the difference that now each orbit of cells is counted with a weight, which is the reciprocal of the order of its stabiliser. 
To do that we consider the following table, in which we consider the polyhedron $D$ constructed in Section \ref{constr} and we list the orbits of facets by dimension, calculate the stabiliser of the first element in the orbit and give its order. 

\begin{longtable}{|c|c|c|}
\hline
Orbit of the facet & Stabiliser & Order \\
\hline 
$\textbf{z}_1, \textbf z_2$ & $\langle A_1,R_1 \rangle$ & $kp$ \\
$\textbf{z}_3, \textbf z_4, \textbf z_5$ & $\langle P^3,R_1 \rangle$ & $pd$ \\
$\textbf{z}_6, \textbf z_{10}, \textbf z_{13}$ & $\langle A_1',R_1 \rangle$ & $pl$ \\
$\textbf{z}_{8}, \textbf z_{7}, \textbf z_{9}, \textbf{z}_{11}, \textbf z_{12}, \textbf z_{14}$ & $\langle A_1, A_1' \rangle$ & $kl$ \\
\hline
$\gamma_{1,3}, \gamma_{2,4}$ & $\langle R_1 \rangle$ & $p$\\
$\gamma_{1,6}, \gamma_{2,10}$ & $\langle R_1 \rangle$ & $p$\\
$\gamma_{3,6}, \gamma_{5,13}, \gamma_{4,10}$ & $\langle R_1 \rangle$ & $p$\\
$\gamma_{2,8}, \gamma_{1,9}, \gamma_{1,12}, \gamma_{2,14}$ & $\langle A_1 \rangle$ & $k$ \\
$\gamma_{7,11}, \gamma_{9,12}, \gamma_{8,14}$ & $\langle JR_1 \rangle$ & $2k$\\
$\gamma_{9,10}, \gamma_{12,13}, \gamma_{6,7}, \gamma_{13,14}, \gamma_{6,8}, \gamma_{10,11}$ & $\langle A_2' \rangle$ & $l$\\
$\gamma_{7,8}, \gamma_{12,14}, \gamma_{9,11}$ & $\langle JR_1^{-1} \rangle$ & $2l$ \\
$\gamma_{4,5}, \gamma_{3,5}, \gamma_{3,4}$ &  $\langle R_2P \rangle$& $2d$ \\
\hline
$\begin{array}{l l}
F(P,J), & F(P^{-1}, J^{-1}) 
\end{array}$ 
& $A_1$ & $k$ \\
$F(R_1, R_1^{-1}) $
& $R_1$ &  $p$ \\
$F(R_2, R_2^{-1}) $
& $R_2$ & $p$ \\
$\begin{array}{l l l l}
F(P,R_1), & F(P,R_1^{-1}), & F(P^{-1}, R_2), & F(P^{-1}, R_2^{-1})
\end{array}$
& 1 & 1 \\
$\begin{array}{l l l l}
F(J,R_1), & F(J,R_1^{-1}), & F(J^{-1}, R_2), & F(J^{-1}, R_2^{-1})
\end{array}$
& 1 & 1 \\
$\begin{array}{l l l}
F(P,R_2), & F(R_1,R_2^{-1}), & F(R_1^{-1}, P^{-1})
\end{array}$
& 1 & 1 \\
$\begin{array}{l l l}
F(J,R_2^{-1}), & F(R_1,J^{-1}), & F(R_1^{-1}, R_2)
\end{array}$
& $A_1'$ & $l$ \\
$F(J,J^{-1})$ & $J$ & 3 \\
$F(P,P^{-1})$ & $P$ & $3d$ \\
\hline 
$S(J),S(J^{-1})$ & 1 & 1 \\
$S(R_1),S(R_1^{-1})$ & 1 & 1 \\
$S(R_2),S(R_2^{-1})$ & 1 & 1 \\
$S(P),S(P^{-1})$ & 1 & 1 \\
\hline 
$D$ & 1 & 1 \\
\hline
\end{longtable}

The vertices are all contained in two orthogonal complex lines, which implies that the stabiliser is a direct product of two cyclic groups generated each by the complex reflections in these lines. 
The ridges are stabilised by the cycle relations, while the sides are fixed only by the identity, as the side pairing maps send the sides one in the other. 

To find the stabiliser of the edges requires slightly more work. 
If the map $T$ stabilises an edge, then either it will fix the endpoints or it will swap them. 
If we can find a map that swaps them, then it will generate the maps that fix them. 
If the vertices are not in the same orbit, then there is no map that swaps them and analysing the action of the side pairing maps (i.e. the generators of the group) of the vertices, we can verify that the stabilisers are as in the table. 
If they are, the same analysis will tell us if there are maps swapping the endpoints or just fixing them.
In this way it is easy to check that the stabilisers are the above. 

From the table it follows that the Euler orbifold characteristic is
\begin{align} \label{vol}
\chi (\hc / \Gamma)&= && \frac{1}{kp}+\frac{1}{pd}+ \frac{1}{pl} +\frac{1}{kl} 
-\frac{1}{p}-\frac{1}{p}-\frac{1}{p} -\frac{1}{k} -\frac{1}{2k} -\frac{1}{l} -\frac{1}{2l} -\frac{1}{2d} \notag\\ \notag
&&&+\frac{1}{k} +\frac{1}{p} +\frac{1}{p} +1 +1 +1 +\frac{1}{l} +\frac{1}{3} +\frac{1}{3d} \notag\\
&=&& \frac{1}{kp}+\frac{1}{2p} -\frac{3}{p^2}+\frac{1}{2p} -\frac{1}{p^2} -\frac{1}{pk} +\frac{1}{2k}- \frac{1}{pk} -\frac{1}{k^2} \notag\\
&&& -\frac{1}{p} -\frac{1}{2k} -\frac{1}{4} +\frac{1}{2p} +\frac{1}{2k} -\frac{1}{4} +\frac{3}{2p} +\frac{1}{3} +\frac{1}{6} -\frac{1}{p} \notag\\
&=&& -\frac{4}{p^2} -\frac{1}{pk} -\frac{1}{k^2}+ \frac{1}{2k}+\frac{1}{p} \notag\\
&=&& \frac{p^2 +12p -60}{16p^2}-\frac{t^2}{4},
\end{align}
where for the second equality we used $\frac{1}{l}=\frac{1}{2}-\frac{1}{p}-\frac{1}{k}$ and 
$\frac{1}{d}=\frac{1}{2} -\frac{3}{p}$, while in the last one we used $t=-\frac{1}{2}+\frac{1}{p}+\frac{2}{k}$.

This value is coherent with the one found by Sauter in Theorem 5.3 of \cite{sauter}, up to a scalar multiplicative factor.
This factor is related to the fact that he considers different groups related to the ones we have. 
In Section 7 of \cite{sauter} he explains the exact relation between the different groups he considers and shows how the multiplicative factor appears by calculating the volume for the groups we are considering, too.
Let us remark that this value is also consistent with those found in \cite{survey} for the polyhedra obtained by collapsing vertices as in Theorem \ref{main}.

\subsection{Poincaré's polyhedron theorem} \label{poincare}

We will now present the version of the  Poincaré polyhedron theorem that we will use, following the one in \cite{livne}.

\begin{defin}
A \emph{combinatorial polyhedron} is a cellular space homeomorphic to a compact polytope, with ridges contained in exactly two sides.
A \emph{polyhedron} $D$ is the realisation of a combinatorial polyhedron as a cell complex in a manifold $X$.
A polyhedron is \emph{smooth} if its cells are smooth.
By convention, we will take the polyhedron to be open. 
\end{defin}

For the Poincaré polyhedron theorem we will need some conditions on the sides and on the ridges of the polyhedron. 
We will now present such conditions. 
A smooth polyhedron satisfying all of them is called a \emph{Poincaré polyhedron}.

Let $D$ be a smooth polyhedron in $X$ with sides $S_j$, side pairing maps $T_j \in \is (X)$ such that:

(S.1) For each side $S_i$ of $D$, there is another side $S_j$ of $D$ and a side-pairing map $T_i$ such that $T_i(S_i)=S_j$.

(S.2)[\emph{reflection relation}] If $T_i(S_i)=S_j$, then $T_i=T_j^{-1}$. This implies that if $i=j$, then $T_i^2=\id$, so if a map sends a bisector to itself, then it fixes it pointwise. The relations $T_i=T_j^{-1}$ are called reflection relations.

(S.3) $T_i^{-1}(D) \cap D= \emptyset$.

(S.4) $T_i^{-1}(\overline D) \cap \overline D= S_i$.

(S.5) There are only finitely many sides in $D$ and each side contains only finitely many ridges.

(S.6) There exists $\delta >0$ such that for each pair of disjoint sides, they are at distance at least $\delta$.

To list the conditions on the ridges we first need to explain what the cycle transformations are. 
Let $S_1$ be a side of $D$ and $F$ be a ridge in the boundary of $S_1$. 
Let also $T_1$ be the side pairing map associated to $S_1$ and consider the image under $T_1$ of the ridge $F$. 
As we remarked in the definition, each ridge is contained in the boundary of exactly two sides. 
$T_1(F)$ will hence be in the boundary of $T_1(S_1)$, but also in the one of some other side $S_2$. 
We call $T_2$ the side-pairing map associated to $S_2$ and we apply it to the ridge $T_1(F)$. 
Iterating this procedure, we get a sequence of ridges, a sequence of sides $S_i$ and a sequence of maps $T_i$. 
Since we know that the amounts of sides and of ridges are finite, these sequences must be periodic. 
Let $k$ be the smallest integer such that all three sequences are periodic of period $k$. 
Then $T_k \circ \dots \circ T_2 \circ T_1(F)=F$ and we call $T_k \circ \dots \circ T_2 \circ T_1$ the \emph{cycle transformation} at the ridge $F$.
Now, for $T=T_k \circ \dots \circ T_2 \circ T_1$ and $m$ an integer, we define:
\begin{align*}
U_0&=1, & U_1&=T_1, 
& \dots & & U_{k-1}&=T_{k-1} \circ \dots \circ T_2 \circ T_1,\\
U_k&=T, & U_{k+1}&=T_1 \circ T 
& \dots & & U_{2k-1}&=T_{k-1} \circ \dots \circ T_2 \circ T_1 \circ T,\\
& \vdots & & \vdots & & & & \vdots \\
U_{(m-1)k}&=T^{m-1}, &  U_{(m-1)k+1}&=T_1 \circ T^{m-1},
& \dots & & U_{mk-1}&=T_{k-1} \circ \dots \circ T_1 \circ T^{m-1}.
\end{align*}
The ridge conditions are then the following.

(F.1) Every ridge is a submanifold of $X$, homeomorphic to a ball of codimension 2.

(F.2) For each ridge $F$ with cycle transformation $T$, there exists an integer $\ell$ such that $T^\ell$ restricted to $F$ is the identity. This means that a power of $T$ fixes $F$ pointwise.

(F.3)[\emph{cycle relations}] For each ridge $F$ with cycle transformation $T$, it exists an integer $m$ so that $(T^\ell)^m$ is the identity on the whole space $X$.  Moreover, for the $U_i$ defined previously, the preimages $U_i^{-1}(D)$, for $i=0, \dots, m \ell k-1$ are disjoint and the closures of such polyhedra $U_i^{-1}(\overline D)$ cover a neighbourhood of the interior of $F$. In this case we say that $D$ and its images tessellate a neighbourhood of $F$. The relations $T^{\ell m}=\id$ are called cycle relations. 

The Poincaré polyhedron theorem now states

\begin{theo}
Let $D$ be a Poincaré polyhedron with side-pairing transformations $T_j\in \Sigma$, satisfying side conditions \emph{(S.1)--(S.6)} and ridge conditions \emph{(F.1)--(F.3)}. 
Then the group $\Gamma$ generated by the side-pairing transformations is a discrete subgroup of $\is(X)$ and $D$ is a fundamental domain for its action.
A presentation for such group is given by
\[
\Gamma=\left\langle \Sigma \colon
\begin{array}{l l}
\text{reflection relations} \\
\text{cycle relations}
\end{array} \right\rangle.
\]
\end{theo}

\subsection{Proof of the main Theorem \ref{main}}

In this section we will prove that all the hypothesis of the Poincaré polyhedron theorem hold and explain how to use it to prove Theorem \ref{main}.

\subsubsection{Side pairing maps}

Let us now consider the maps $J,P,R_1$ and $R_2$. 
These maps pair the eight sides of the polyhedron, as shown in Figure \ref{sidepairing}. 
In this section we want to show that these side pairing maps verify the conditions (S.1)--(S.6). 

Conditions (S.1), (S.2), (S.5) follow clearly from our construction of the sides. 
Also, (S.6) is an empty condition, because each pair of sides of our polyhedron intersects. 
The following proposition shows that conditions (S.3) and (S.4) are verified by the sides of $D$. 

\begin{prop}
Let $T$ be one of $J^{\pm1},P^{\pm1},R_1^{\pm1}$ and $R_2^{\pm1}$. 
Then $T^{-1}(D) \cap D = \emptyset$.
Moreover, $T^{-1}(\overline D) \cap \overline D =S(T)$. 
\end{prop}

\begin{proof}
Let us take a side $S(T)$. 
By definition it is contained in a bisector $B(T)$. 
By Lemma \ref{lemmapoly}, there exist two vertices $\textbf{z}_i$ and $\textbf{z}_j$ such that $B(T)$ is the set of points equidistant from $\textbf{z}_i$ and $T^{-1}(\textbf{z}_j)$. 
By applying $T$ we get that $T(B(T))$ is $B(T^{-1})$, which is the bisector equidistant from $T(\textbf{z}_i)$ and $\textbf{z}_j$.
By Remark \ref{polydef}, the points of the polyhedron are closer to $\textbf{z}_i$ than to $T^{-1}(\textbf{z}_j)$, while the ones of $T(D)$ are closer to $T(\textbf{z}_i)$ than to $\textbf{z}_j$. 
This implies that $T^{-1}(D) \cap D = \emptyset$.

If we now also consider the boundary of the polyhedron and we pass to $T^{-1}(\overline D) \cap \overline D =S(T)$, we are considering the equality cases in Lemma \ref{lemmapoly}. 
But the lemma itself guarantees that the intersections, which corresponds to the equality cases of the lemma, are always contained in $B(T)$.
Since by definition $S(T)= \overline D \cap B(T)$, we are done.
\end{proof}

\begin{figure}[!ht]
\centering
\includegraphics[width=0.9\textwidth]{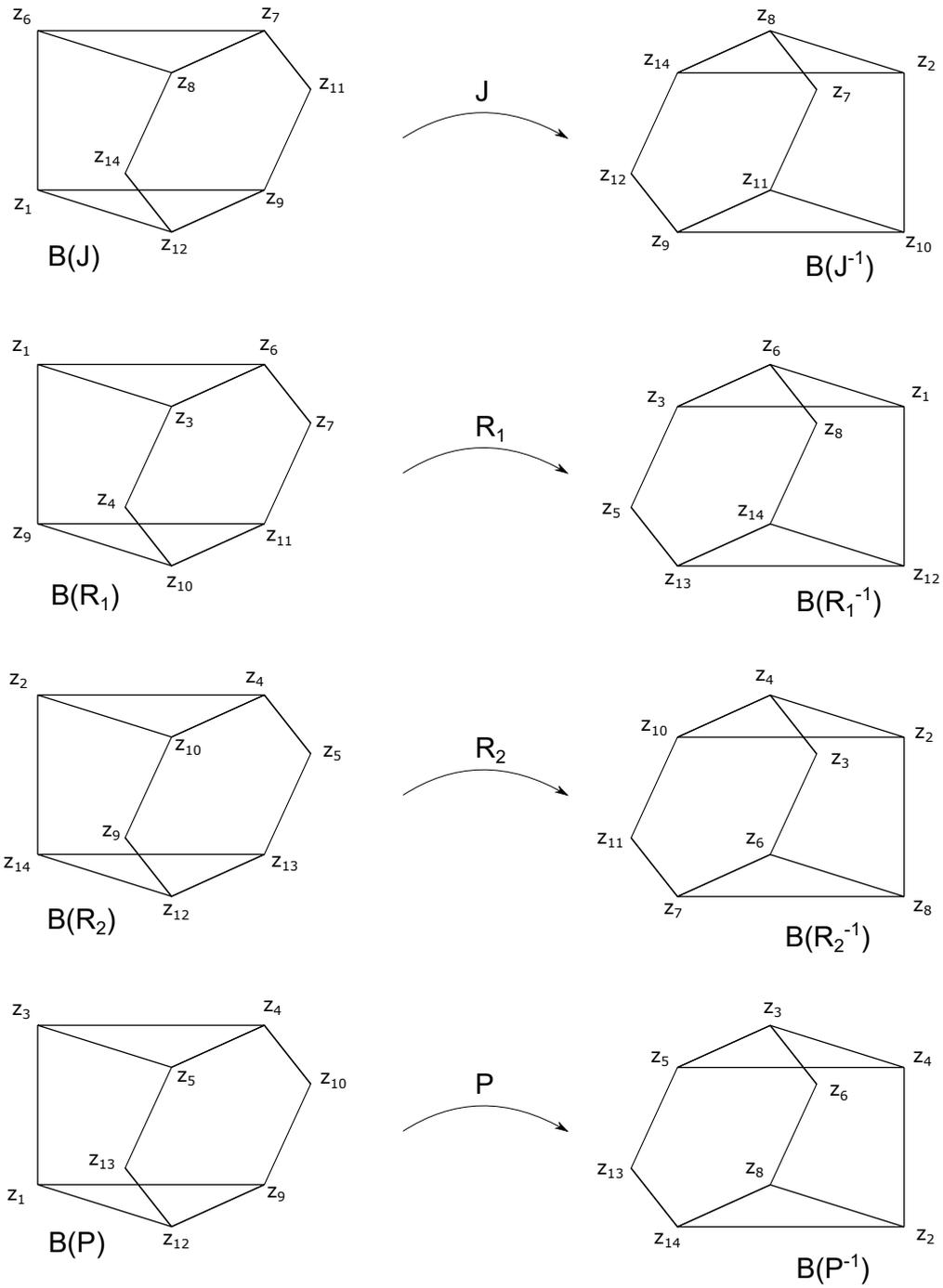}
\begin{quote}\caption{The sides of the polyhedron with the corresponding side pairing maps. \label{sidepairing}} \end{quote}
\end{figure}

\subsubsection{Cycle relations}\label{cyclerel}

It remains now to show that the ridges of the polyhedron $D$ satisfy conditions (F.1)--(F.3). 
This will be done in this and next section. 
The first condition is straightforward in this case. 
In fact it is easy to see that the edges in a ridge intersect so that they bound a polygon, giving hence a ridge homeomorphic to a ball. 
In the following table we summarise the cycle relations coming from Properties (F.2) and (F.3). Proving them is a simple calculation of the action of the transformations on the bisectors. 

\clearpage
\begin{center}
\begin{tabular}{|c|c|c|c|}
\hline
Ridges in the cycle & Transformation & $\ell$ & $m$ \\
\hline 
$\begin{array}{l l}
F(P,J), & F(P^{-1}, J^{-1}) 
\end{array}$ 
& $P^{-1}J$ & 1 & $k$ \\
$F(R_1, R_1^{-1}) $
& $R_1$ & 1 & $p$ \\
$F(R_2, R_2^{-1}) $
& $R_2$ & 1 & $p$ \\
$\begin{array}{l l l l}
F(P,R_1), & F(P,R_1^{-1}), & F(P^{-1}, R_2), & F(P^{-1}, R_2^{-1})
\end{array}$
& $R_1^{-1}P^{-1}R_2 P$ & 1 & 1 \\
$\begin{array}{l l l l}
F(J,R_1), & F(J,R_1^{-1}), & F(J^{-1}, R_2), & F(J^{-1}, R_2^{-1})
\end{array}$
& $R_1^{-1}J^{-1}R_2 J$ & 1 & 1 \\
$\begin{array}{l l l}
F(P,R_2), & F(R_1,R_2^{-1}), & F(R_1^{-1}, P^{-1})
\end{array}$
& $R_2 P^{-1} R_1 $ & 1 & 1 \\
$\begin{array}{l l l}
F(J,R_2^{-1}), & F(R_1,J^{-1}), & F(R_1^{-1}, R_2)
\end{array}$
& $R_2 R_1 J $ & 1 & $l$ \\
$F(J,J^{-1})$ & $J$ & 3 & 1 \\
$F(P,P^{-1})$ & $P$ & 3 & $d$ \\
\hline
\end{tabular}
\end{center}
This table gives immediately a proof the presentation as given in Theorem \ref{presentation}, as they correspond to the cycle relations in the Poincaré polyhedron theorem and the reflection relations are empty.
The second part of property (F.3) will be proved in the next section.

\subsubsection{Tessellation around the ridges}

We now want to prove that the images of the polyhedron under the side paring maps tessellate around neighbourhoods of the interior of the ridges. 
This is proved in different ways, depending on whether the ridges described in Section \ref{ridges} are contained in a Giraud disc, in a Lagrangian plane or in a complex line.

\paragraph*{Tessellation around ridges contained in a Giraud disc.}

The easiest case to treat is the tessellation around the ridges $F(J,J^{-1}), F(R_1,R_2^{-1}),F(P,R_2)$ and $ F(P^{-1},R_1^{-1})$, contained in Giraud discs. 
The main tool for this is Lemma \ref{lemmapoly}. The proof goes along the lines of the one in \cite{livne}.

\begin{prop}\label{JJ-}
We have the following:
\begin{itemize}
\item The polyhedron $D$ and its images under $J$ and $J^{-1}$ tessellate around the ridge $F(J,J^{-1})$.
\item The polyhedron $D$ and its images under $R_1^{-1}$ and $R_2$ tessellate around the ridge $F(R_1,R_2^{-1})$.
\item Moreover, the polyhedron $D$ and its images under $R_2^{-1}$ and $P^{-1}$ tessellate around the ridge $F(P,R_2)$.
\item Finally, the polyhedron $D$ and its images under $R_1$ and $P$ tessellate around the ridge $F(P^{-1},R_1^{-1})$.
\end{itemize}
\end{prop}

\begin{proof}
The proof consists in dividing the space into points that are closer to one of $L_{*0}$, $J(L_{*0})$ or $J^{-1}(L_{*0})$ and showing that $D$ and its images under $J$ are contained each in a different one of these domains and coincide with them around the ridge $F(J,J^{-1})$.

More formally, by Lemma \ref{lemmapoly} we know that $D$ is contained in the part of space closer to $L_{*0}$ than to its images under $J$ and $J^{-1}$. 
We can hence write 
\begin{equation}\label{DJ}
D \subset \{\textbf z \in \hc \colon \lvert \langle \textbf{z}, \textbf n_{*0} \rangle \rvert < \lvert \langle \textbf{z}, J(\textbf n_{*0}) \rangle \rvert, \quad \lvert \langle \textbf{z}, \textbf n_{*0} \rangle \rvert < \lvert \langle \textbf{z}, J^{-1}(\textbf n_{*0}) \rangle \rvert \}.
\end{equation}
For a point $\textbf z \in J ^{\pm 1} (D)$, we also have $J^{\mp 1}(\textbf z) \in D$. 
Applying the conditions in (\ref{DJ}) to $J^\mp(\textbf z)$, we get 
\[
\lvert \langle J^{\mp 1}(\textbf{z}), \textbf n_{*0} \rangle \rvert < \lvert \langle J^{\mp 1}(\textbf{z}), J(\textbf n_{*0}) \rangle \rvert, \quad \lvert \langle J^{\mp 1}(\textbf{z}), \textbf n_{*0} \rangle \rvert < \lvert \langle J^{\mp 1}(\textbf{z}), J^{-1}(\textbf n_{*0}) \rangle \rvert.
\]
By applying $J^{\pm 1}$ to all terms of (\ref{DJ}), we obtain 
\[
J^ {\pm 1} (D) \subset \{\textbf z \in \hc : \lvert \langle \textbf{z}, J^{\pm 1}( \textbf n_{*0}) \rangle \rvert < \lvert \langle \textbf{z}, \textbf n_{*0} \rangle \rvert, \quad \lvert \langle \textbf{z}, J^{\pm 1}(\textbf n_{*0}) \rangle \rvert < \lvert \langle \textbf{z}, J^{\mp 1}\textbf n_{*0} \rangle \rvert \}.
\]
Clearly, we used the fact that $J$ has order 3, so $J^2=J^{-1}$. 
It is obvious that $D,J(D)$ and $J^{-1}(D)$ are disjoint. 

The ridge we are considering is characterized by $\im (e^{i \phi}z_1)=\im (e^{-i \phi}w_1)=0$. 
We take a neighbourhood of the interior small enough, so that it does not meet the other sides of $D$. 
Then a point of $U$ is in $\overline D$ if and only if it is closer to $L_{*0}$ than to its images.
This is because if we consider the $z_1$ and $w_1$ coordinates small enough, $D$ actually coincides with the set described in (\ref{DJ}) and same for the images.
From this, it's easy to see that $D$, $J(D)$ and $J^{-1}(D)$ tessellate around $U$. 

The cycle transformation is 
\[
F(J,J^{-1}) \xrightarrow{J} F(J,J^{-1}).
\]

The other points of the proof are done in the same way, by taking the different images mentioned in the statement and using the same proof strategy.
\end{proof}

\paragraph*{Tessellation around ridges contained in Lagrangian planes.}

The second type are the ridges $F(P,R_1), F(P,R_1^{-1}), F(J,R_1)$ and $F(J,R_1^{-1})$, contained in Lagrangian planes. 
Again, the proofs are similar to the ones in \cite{livne}.

They contain either vertex $\textbf z_1$ or $\textbf z_2$ and they are defined by conditions only on the $\textbf z$-coordinates or on the $\textbf w$-coordinates.
It is enough to show that the polyhedron and its images under the side pairing maps tessellate around the ridges containing the vertex $\textbf z_1$. 
By applying $\iota$ we will have the same for ridges containing $\textbf z_2$. 

\begin{prop} \label{P,R1}
The polyhedron $D$ and its images under $R_1^{-1}, P^{-1}$ and $R_1^{-1}P^{-1}$ tessellate around the ridge $F(P,R_1)$.
\end{prop}

\begin{proof}
Considering that $\textbf{w}= P^{-1}(\textbf{z})$ and that applying $R_1$ means to add $\theta$ to the argument of $z_2$, we can prove the signs in the following table. 

\begin{center}
\begin{tabular}{|c|c|c|c|c|}
\hline
Image of $D$ & $\im (z_1)$ & $\im (e^{i \phi} z_1)$ & $\im (z_2)$ &  $\im (e^{-i \theta} z_2)$ \\
\hline 
$D$ & - & + & + & - \\
$R_1^{-1}(D)$ & - & + & - & - \\
$P^{-1}(D)$ & + & + & + & - \\
$R_1^{-1}P^{-1}(D)$ & + & + & - & - \\
\hline
\end{tabular}
\end{center}

We can see from the table that each pair of images have some coordinates whose imaginary part has different sign. 
This clearly implies that they are disjoint.

Now, the ridge $F(P,R_1)$ is characterised by $\im (z_1)=\im (z_2)= 0$. 
Let us now consider a neighbourhood $U$ of the ridge and a point $\textbf z \in U$. 
If $\textbf z$ has argument of $z_1$ smaller than 0, then $D$ and $R_1^{-1}$ cover $U$, in the respective cases when the argument of $z_2$ and positive or negative. 
Similarly, when $\textbf z$ has first coordinate of argument bigger than 0, then $P^{-1}(D)$ and $R_1^{-1}P^{-1}(D)$ cover $U$, when $\arg (z_2)$ is positive or negative respectively. 

The corresponding cycle transformation is 
\[
F(P,R_1) \xrightarrow{P} F(P^{-1},R_2) \xrightarrow{R_2} F(P^{-1},R_2^{-1}) \xrightarrow{P^{-1}} F(P,R_1^{-1}) \xrightarrow{R_1^{-1}} F(P,R_1).
\]
\end{proof}

By applying $R_1, PR_1$ and $P=R_2^{-1}PR_1$ we get similar results for the other ridges in the cycle, namely $F(P,R_1^{-1})$, $F(P^{-1},R_2^{-1})$ and $F(P^{-1},R_2)$ respectively.

In a similar way, we can also prove
\begin{prop}
The polyhedron $D$ and its images under $R_1^{-1}, J^{-1}= A_1^{-1}P^{-1}$ and $R_1^{-1}A_1^{-1}P^{-1}$ tessellate around the ridge $F(J,R_1)$.
\end{prop}

Again, by applying the maps in the cycle transformation, which is 
\[
F(J,R_1) \xrightarrow{J} F(J^{-1},R_2) \xrightarrow{R_2} F(J^{-1},R_2^{-1}) \xrightarrow{J^{-1}} F(J,R_1^{-1}) \xrightarrow{R_1^{-1}} F(J,R_1),
\]
we can get that the tessellation property (F.3) holds also for $F(J^{-1},R_2),F(J^{-1},R_2^{-1})$ and $F(J,R_1^{-1})$.

\paragraph*{Tessellation around ridges contained in complex lines.}

In this section we will show that the images of $D$ tessellate around the ridges contained in complex lines. 
We will divide them in two parts for which we will use slightly different methods. 

We will start with the ridges contained in complex lines and defined by conditions either on the $\textbf{z}$-coordinates or on the $\textbf{w}$-coordinates.
These are ridges $F(P,J)$, $F(R_1,R_1^{-1})$, $F(P^{-1},J^{-1})$ and $F(R_2,R_2^{-1})$. 
From the first two, the others follow by applying $\iota$.
We will again omit the proofs, as they are equivalent to the ones in \cite{livne}.
These proofs strongly rely on the fact that $p$ and $k$ are integers. 
In some of the cases that we are considering, though, $k$ is of the form $p/2$, with $p$ odd. 
The proof can be adapted, as we will explain in Section \ref{rationalk}.

\begin{prop}\label{P,J}
The polyhedron $D$ and its images under $P^{-1},A_1$ and $A_1P^{-1}$ tessellate around the ridge $F(P,J)$.
Moreover, the polyhedron $D$ and its images under $R_1$ tessellate around the ridge $F(R_1,R_1^{-1})$.
\end{prop}

By applying $\iota$ we have equivalent results around $F(P^{-1},J^{-1})$ and $F(R_2,R_2^{-1})$. 

Moreover, in exactly the same way as in \cite{livne} we can prove that $D$ and appropriate images tessellate around $F(P,P^{-1})$. 
%
The proof is done by showing that in some coordinates $P^3$ rotates $\textbf n_{*0}$ by $e^{i\psi}$, with $\psi=\frac{2\pi}{d}$ and $d=\frac{2p}{p-6}$, as in Table \ref{tablelist}.
At the same time, $P^3$ fixes the ridge itself. 
Then the polyhedron and its images under $P$ and $P^{-1}$ will be contained in different sectors for the arguments of at least one of the new coordinates and they will cover a sector of length $\psi$.
Applying $P^3$ it will cover a whole neighbourhood of the ridge by rationality of $\psi$, since $d$ is always an integer.

The corresponding cycle transformation is 
\[
F(P,P^{-1}) \xrightarrow{P} F(P,P^{-1}).
\]

Finally, we have the last set of ridges.

\begin{prop}
The polyhedron $D$ and its images under $J$, $JR_2$, $R_1R_2J$ and their compositions tessellate around the ridge $F(R_1,J^{-1})$.
\end{prop}

\begin{proof}
The proof works similarly to those for ridges $F(P,J)$ and $F(P,P^{-1})$.
We can in fact change coordinates as in the latter case, so to have an analogous situation to the one in the former. 
In this case though, we will define $\psi=\frac{2\pi}{l}$, for $l$ defined in \eqref{ldef}.

First of all, we recall that $F(J^{-1},R_1)$ is contained in $L_{*3}$. 
Furthermore, the map $JR_2R_1$ rotates the normal vector $\textbf n_{*3}$ by $-\psi$ and it fixes pointwise the ridge.
We then change basis to new coordinates, so that the first coordinate is along the normal vector to the complex line (up to a minus sign, which will be useful in the calculations) and the other two are along two vectors spanning the complex line once we pass to projective coordinates. 

The vector in the new basis will hence be 
\[
\begin{pmatrix}
z_1 \\ z_2 \\ 1
\end{pmatrix}
=\frac{\sin \phi-\sin (\theta+\phi)z_2}{\sin(\theta+\phi)-\sin \phi}
\begin{pmatrix}
0 \\ -1 \\ -1
\end{pmatrix}
+z_1 \begin{pmatrix}
1 \\ 0 \\ 0
\end{pmatrix}
+ \frac{1-z_2}{\sin(\theta+\phi)-\sin \phi}
\begin{pmatrix}
0 \\ \sin \phi \\ \sin (\theta+\phi)
\end{pmatrix}.
\]

We define then the $\xi$-coordinates to be
\begin{align}\label{xiD}
\xi_1&=\frac{\sin \phi-\sin (\theta+\phi)z_2}{1-z_2},  \nonumber\\
\xi_2&= \frac{z_1(\sin(\theta+\phi)-\sin \phi)}{1-z_2}.
\end{align}

\begin{figure}
\centering
\includegraphics[width=1\textwidth]{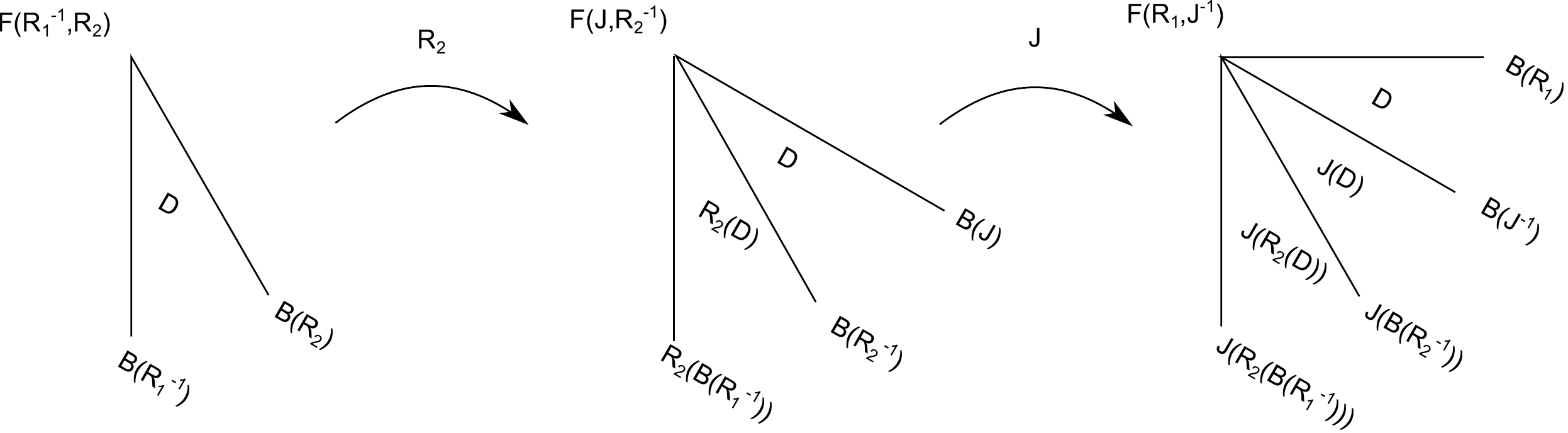}
\begin{quote}\caption{The tessellation around $F(R_1,J^{-1})$. \label{tessR1J-}} \end{quote}
\end{figure}

Let us now look at Figure \ref{tessR1J-}. 
By definition the ridge $F(R_1,J^{-1})$ is contained in the intersection of $B(R_1)$ and $B(J^{-1})$. 
It is clear that on $B(R_1)$, since $z_2$ is real, also $\xi_1$ will be real. 

If we take the ridge $F(R_1^{-1},R_2)$, we know that the polyhedron $D$ is as in the first image of Figure \ref{tessR1J-}.
By definition of the bisectors, $R_2(B(R_2))=B(R_2^{-1})$.
Also, $R_2$ sends $F(R_1^{-1},R_2)$ to $F(J,R_2^{-1})$ (see cycle relation below). 
Then we can apply the map to the first image and get the second configuration, since $F(J,R_2^{-1})$ is in $B(J)$ and $B(R_2^{-1})$ by definition but also in $R_2(B(R_1^{-1}))$ by construction. 
We can do the same thing applying $J$ and we get the third configuration in the figure. 

We now want to prove that in the argument of the coordinate $\xi_1$, $D$, $J(D)$ and $JR_2(D)$ make a sector of length $\psi$. 
Once we prove this, we can apply an argument as in \ref{P,J} and apply $R_1$.
But this gives us the map $R_1R_2J$ which acts on the $\xi$ coordinates $(\xi_1,\xi_2)$ by sending to $(e^{-i \psi} \xi_1,\xi_2)$, and hence it carries the configuration all around the ridge and tessellates the space because of rationality of $\psi$, which comes from the fact that $l$ is always an integer.

To prove that the length of the sector is $\psi$, we will prove that the argument of the $\xi_1$ coordinate of a point on $JR_2(B(R_1^{-1}))$ is $-\psi$. 
This is just a calculation, as it turns out that 
\begin{align*}
JR_2 \textbf z 
&= JR_2 \begin{pmatrix}
z_1 \\ z_2 \\1
\end{pmatrix}
=J \begin{pmatrix}
-\sin \theta e^{-i \phi}z_1+(\sin \phi+\sin (\theta-\phi))(1-z_2) \\
\sin \phi (1-z_1-e^{-i \theta}z_2) \\
-\sin(\theta+\phi)(z_1+z_2)+\sin \phi+\sin \theta e^{i \phi}
\end{pmatrix} \\
&= \begin{pmatrix}
2z_1 \sin^2 \phi (1-\cos \theta) \\
2z_2 \sin^2 \phi e^{i \phi} (\cos (\theta+\phi)-\cos \phi)
+\sin^2 \phi (1-e^{i \theta})(e^{2i \phi}-1) \\
z_2 (1-e^{-i \theta}) \sin \phi \sin(\theta+\phi) (e^{2i \phi}-1)
+ \sin^2 \phi (1-e^{-i \theta})(1-e^{i(2\phi+\theta)})
\end{pmatrix}.
\end{align*}

Then we can calculate its $\xi_1$ coordinate and we have 
\begin{align}\label{xiJR2D}
\xi_1 &=\frac{-e^{i(\theta+2\phi)}\sin^2 \phi(2(1-\cos \theta)(\sin \phi-e^{-i \theta} z_2 \sin(\theta+\phi)))}
{-2 \sin^2 \phi(1-\cos\theta)e^{-i \theta}z_2+2 \sin^2 \phi(1-\cos \theta)}= \nonumber \\
&=e^{-i\psi} \frac{\sin \phi-\sin(\theta+\phi) e^{-i \theta} z_2}{1-e^{-i \theta}z_2}
\end{align}

If a point $\textbf{z}$ is in $B(R_1^{-1})$, then its $z_2$ coordinate is $z_2=e^{i \theta} u$ and hence the previous expression is 
\[
\xi_1=e^{-i \psi} \frac{\sin \phi -\sin(\theta+\phi)u}{1-u}.
\]
Clearly, the argument of the new coordinate is $-\psi$. 

The last thing we need to show is that the three images are disjoint. 
We already saw that $D$ is disjoint from $J(D)$ and $R_2(D)$ in \ref{JJ-} and in the equivalent statement of \ref{P,J} for $R_2$, respectively. 
But then also $J(D)$ and $JR_2(D)$ are disjoint because $J$ is an isometry.
To prove the disjointness of $D$ and $JR_2(D)$, we look at the expression for the $\xi_1$ coordinate of a point in $D$, as in \eqref{xiD}, and of a point in $JR_2(D)$, as in \eqref{xiJR2D}. 

To show disjointness, we will show that $D$ and $JR_2(D)$ are contained in the sector where the argument of $\xi_1$ is respectively bigger and smaller than $-\frac{\psi}{2}$.
To do that we just need to show that $B(J^{-1})$ and $J(B(R_2^{-1}))=JR_2(B(R_2))$ are as said. 

Since both these bisectors are defined by equations on the $\textbf{w}$-coordinates, it is useful to rewrite the two equations in terms of these, using Formulae \eqref{z1} and \eqref{z2}.
They will be as following. 
If $\textbf z \in D$, then 
\[
\xi_1=2\sin \frac{\theta}{2}\sin \phi e^{-i \frac{\psi}{2}}
\frac{\sin \theta -\sin(\theta+\phi)e^{-i \phi}w_1}
{-\sin \theta e^{-i \phi}w_1+(\sin \phi -\sin(\theta+\phi))w_2+\sin \phi+\sin(\theta-\phi)},
\]
with $w_1$ and $w_2$ coordinates of $\textbf z$.
We will consider points in $B(J^{-1})$, so $w_1=e^{i \phi}u$, with $u$ real and we want to show that $\im (e^{-i \frac{\psi}{2}}\xi_1)>0$.

Taking the imaginary part of the expression above, this means requiring that 
\[
(\sin \theta -\sin(\theta+\phi)u)(\sin(\theta+\phi)-\sin \phi)\im(w_2)>0.
\]
The third term is positive for points in $D$, while the second one is positive as long as $l$ is positive, which is the case where the ridge we are tessellating around does not collapse. 
The last thing we need is then to prove that in $B(J^{-1})$ the modulus of $w_2$ remains smaller than $\frac{\sin \theta}{\sin(\theta+\phi)}$. 
But looking at the structure of the side, as in Figure \ref{sidepairing}, we can see that the side is bounded by the complex lines $L_{03}$ and $L_{*3}$, so the modulus of $w_2$ is between 0 and $\frac{\sin \theta}{\sin(\theta+\phi)}$.

On the other hand, if $\textbf{z}$ is in $JR_2(B(R_2))$, its coordinate will be 
\[
\xi_1=2\sin \frac{\theta}{2}\sin \phi e^{-i \frac{\psi}{2}}
\frac{\sin \phi- \sin(\theta+\phi)w_2}
{(\sin \phi-\sin(\theta+\phi))e^{-i \phi}w_1 -\sin \theta w_2+\sin \theta},
\]
with $w_1$ and $w_2$ coordinates of a point in $D$. 
As they vary through the possible values, $\textbf{z}$ varies in $JR_2(B(D))$. 
Here we consider points in $JR_2(B(R_2))$, so where $w_2=x$, with $x$ real and we want to show this time that $\im (e^{-i \frac{\psi}{2}}\xi_1)<0$. 

We now take the imaginary part of the expression for $\xi_1$ and we obtain that such a condition is equivalent to requiring that 
\[
(\sin \phi -\sin(\theta+\phi)x)(\sin(\theta+\phi)-\sin \phi)\im(e^{-i \phi}w_1)<0.
\]
As before, this reduces to show that the first term is positive and this is true because of the structure of $B(R_2)$, which is contained between $L_{12}$ and $L_{*2}$.
This concludes the proof.

The corresponding cycle transformation is
\[
F(R_1,J^{-1}) \xrightarrow{R_1} F(R_1^{-1},R_2) \xrightarrow{R_2} F(J, R_2^{-1}) \xrightarrow{J} F(R_1,J^{-1}).
\]
\end{proof}

By applying the isometries that compose the cycle transformation, we obtain the tessellation around the last ridges, $F(R_1^{-1},R_2)$, $F(J, R_2^{-1})$ and $F(R_1,J^{-1})$.

\subsection{Polyhedra with extra symmetry}\label{rationalk}

In this section we will describe the particular case when $l$ or $k$ are equal $\frac{p}{2}$.
Considering $k$ or $l$ is equivalent, since swapping them corresponds to swapping $\mu_1$ and $\mu_5$ in the ball quintuple, which geometrically corresponds to choosing whether to have $v_*$ or $v_0$ in the origin of the coordinates and will hence give us the same construction.
In this case the polyhedron has an extra symmetry, because by definition the condition implies that $\phi=\theta$. 
The pairs $(p,k)$ in our list and satisfying this condition, are $(5,5/2)$, $(6,3)$, $(7,7/2)$, $(8,4)$, $(9,9/2)$, $(10,5)$, $(12,4)$ and $(18,3)$.
By Theorem 6.2 in \cite{sauter}, the lattice $(p,\frac{p}{2})$ is isomorphic to the one of the form $(p,2)$.

This includes the cases when $k$ is not an integer, which have not been treated previously because previous proofs for tessellation rely on the fact that $k$ was always an integer. 
When tessellating a neighbourhood of $F(P,J)$, in fact, $D$ and $P^{-1}(D)$ are contained in sectors where the argument of $z_1$ is between 0 and $\phi$ and between $\phi$ and $2\phi$ respectively. 
Then, one can apply $A_1$ to the polyhedra and translate of $2\phi$ the sector.
In order to cover exactly all the possible values of the argument of $z_1$ one then needs $k$ to be an integer. 

To avoid this problem, one can use a slightly different version of the same theorem, namely Poincaré polyhedron theorem for coset decompositions.
The statement is very similar to the one we gave and can be found in \cite{mostow} and in \cite{nonarithm}. 
The basic difference is the presence of a finite group $\Upsilon < \is \hc$ preserving the polyhedron and compatible with the side pairing maps. 

Then one just needs tessellation around one facet in each orbit of the action of $\Upsilon$ and the cosets of the polyhedron will tessellate the space. 
This also gives a different presentation for the group generated by $\Upsilon$ and the side pairings, with the additional relations given by a presentation of $\Upsilon$ and by the compatibility relations. 
Here the group $\Upsilon$ will be a finite cyclic group. 

The reason why this approach is reasonable lies in the fact that when $k=\frac{p}{2}$, by definition, $\phi=\theta$ and hence the configuration space has an extra symmetry. 
The main difference is that we do not need then a butterfly move $A_1$, because we can introduce a move that swaps points $v_0$ and $v_1$ (which now have same cone angle).
The new move, squared, is the same as $A_1$ we used so far.
This solves the problem because the new move acts on the $z_1$ coordinate by rotating by $\phi$ instead of $2\phi$ as before, so we just need $2k$ to be an integer. 

From now on, we will assume we are in the case where $k=\frac{p}{2}$, hence $\phi=\theta=\frac{2\pi}{p}$.
Clearly, the calculations to find vertices, area and moves could be simplified by adding the relation $\phi=\theta$ in the equations, but for simplicity we will leave them as they are. 
We will have the moves $R_1$ and $R_2$ defined as before, but we will also have an extra move corresponding to swapping the vertices $v_0$ and $v_1$, as we already mentioned. 
This new move, that we will call $S_1$, can be found by requiring that the images under $S_1$ of the $v_i$'s, which we denote by $v_i'$, satisfy the equations $v_3'=v_3, v_2'=v_2, v_1'=v_0$ and $v_0'=v_{-1}$. 
The move is illustrated in Figure \ref{octagonS1}. 

\begin{figure}[t]
\centering
\includegraphics[width=1\textwidth]{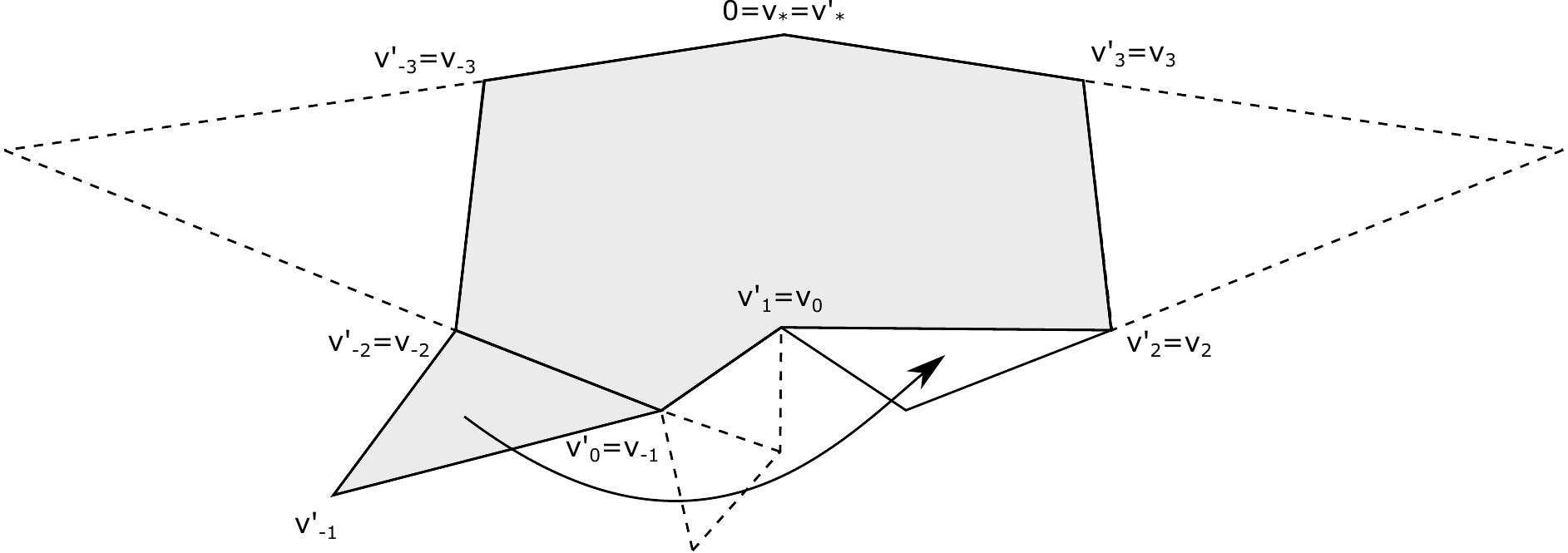}
\begin{quote}\caption{The move $S_1$. \label{octagonS1}} \end{quote}
\end{figure}

Solving the equations or looking at the geometric meaning of the move, one can deduce the matrix of $S_1$. 
The three moves will hence be
\begin{align*}
R_1&=
\begin{bmatrix}
1 & 0 & 0 \\
0 & e^{i \theta} & 0 \\
0 & 0 & 1 \\
\end{bmatrix},
& R_2&=\frac{1}{1-e^{-i\theta}}
\begin{bmatrix}
-e^{-i \theta} & -1 & 1 \\
-1 & -e^{-i \theta} & 1 \\
-2\cos \theta & -2 \cos \theta & 1+e^{i \theta} \\
\end{bmatrix},
&S_1&=
\begin{bmatrix}
e^{i \theta} & 0 & 0 \\
0 & 1 & 0 \\
0 & 0 & 1 \\
\end{bmatrix}.
\end{align*}
\begin{rk}\label{S1prop}
We remark that $S_1$ commutes with $R_1$ and satisfies the braid relation with $R_2$.
\end{rk}

By looking at the coordinates of the vertices of the polyhedron and keeping in mind that $\phi=\theta$, it is easy to see that the action of $S_1$ on the vertices is the following:
\begin{align*}
S_1& \colon \textbf{z}_1 \to \textbf{z}_1, & 
S_1& \colon \textbf{z}_6 \to \textbf{z}_3, &
S_1& \colon \textbf{z}_7 \to \textbf{z}_4, & 
S_1& \colon \textbf{z}_8 \to \textbf{z}_5, \\
S_1& \colon \textbf{z}_9 \to \textbf{z}_9, & 
S_1& \colon \textbf{z}_{11} \to \textbf{z}_{10}, &
S_1& \colon \textbf{z}_{12} \to \textbf{z}_{12}, & 
S_1& \colon \textbf{z}_{14} \to \textbf{z}_{13}.
\end{align*}
In other words, this means that $S_1 \colon B(J) \to B(P)$.

It is then natural to use $S_1$ as a side pairing map and to find another map which will map $B(J^{-1})$ and $B(P^{-1})$ to each other. 
With $P=R_1R_2$ as before, we can define $S_2=PS_1P^{-1}$, which will act on the $\textbf{w}$-coordinates in the same way as $S_1$ does on the $\textbf{z}$-coordinates. 
In this sense they have an analogous relation to the one between $R_1$ and $R_2$. 
By inspection on the table of coordinates of the vertices, one can see that the action of $S_2$ is
\begin{align*}
S_2& \colon \textbf{z}_2 \to \textbf{z}_2, & 
S_2& \colon \textbf{z}_3 \to \textbf{z}_{11}, &
S_2& \colon \textbf{z}_4 \to \textbf{z}_{10}, & 
S_2& \colon \textbf{z}_5 \to \textbf{z}_9, \\
S_2& \colon \textbf{z}_6 \to \textbf{z}_7, & 
S_2& \colon \textbf{z}_{8} \to \textbf{z}_{8}, &
S_2& \colon \textbf{z}_{13} \to \textbf{z}_{12}, & 
S_2& \colon \textbf{z}_{14} \to \textbf{z}_{14}.
\end{align*}
This means that $S_2$ sends $B(P^{-1})$ to $B(J^{-1})$ as required. 

The new side pairing maps will then be
\begin{align*}
R_1 &\colon B(R_1) \to B(R_1^{-1}), & R_2 &\colon B(R_2) \to B(R_2^{-1}), \\
S_1 &\colon B(J) \to B(P), & S_2 &\colon B(P^{-1}) \to B(J^{-1}).
\end{align*}

In order to apply the Poincaré polyhedron theorem for cosets, we now need a group $\Upsilon$ that leaves the polyhedron invariant and is compatible with the action of the side pairing maps. 
Let us then define $K=R_1R_2S_1$. 
This is similar to the definition of $J$, but using $S_1$ instead of $A_1$. 
Multiplying the matrices gives 
\[
K=\frac{1}{1-e^{-i\theta}}
\begin{bmatrix}
-1 & -1 & 1 \\
-e^{2i\theta} & -1 & e^{i\theta} \\
-2\cos \theta e^{i \theta} & -2 \cos \theta & 1+e^{i \theta} \\
\end{bmatrix}.
\]

\begin{rk}
By looking at the eigenvalues of $K$, one can see that projectively it has order 4, since $e^{i \theta}K$ has both determinant and trace equal 1.
\end{rk}

One can apply $K$ to the vertices of the polyhedron and verify that its action is the following:
\begin{align*}
K& \colon \textbf{z}_1 \to \textbf{z}_2, & 
K& \colon \textbf{z}_2 \to \textbf{z}_1, &
K& \colon \textbf{z}_3 \to \textbf{z}_{10}, & 
K& \colon \textbf{z}_4 \to \textbf{z}_9, \\
K& \colon \textbf{z}_5 \to \textbf{z}_{11}, & 
K& \colon \textbf{z}_6 \to \textbf{z}_4, &
K& \colon \textbf{z}_7 \to \textbf{z}_5, & 
K& \colon \textbf{z}_8 \to \textbf{z}_3 \\
K& \colon \textbf{z}_9 \to \textbf{z}_{14}, & 
K& \colon \textbf{z}_{10} \to \textbf{z}_{12}, &
K& \colon \textbf{z}_{11} \to \textbf{z}_{13}, & 
K& \colon \textbf{z}_{12} \to \textbf{z}_8, \\
K& \colon \textbf{z}_{13} \to \textbf{z}_7, & 
K& \colon \textbf{z}_{14} \to \textbf{z}_6. &&&&
\end{align*}
This means that $K$ preserves the polyhedron and acts on the sides as
\begin{align*}
B(R_1) &\xrightarrow{K} B(R_2) \xrightarrow{K} B(J) \xrightarrow{K} B(P^{-1}) \xrightarrow{K} B(R_1), \\
B(R_1^{-1}) &\xrightarrow{K} B(R_2^{-1}) \xrightarrow{K} B(P) \xrightarrow{K} B(J^{-1}) \xrightarrow{K} B(R_1^{-1}),
\end{align*}
namely it cyclically permutes them, preserving the two columns in Figure \ref{sidepairing}.
Using Remark \ref{S1prop}, and the braid relation between $R_1$ and $R_2$, it is easy to see that 
\begin{align*}
R_2&=KR_1K^{-1}, & S_1&=K^2 R_1K^{-2}, & S_2&=K^3 R_1K^{-3}, & R_1&=K^4 R_1K^{-4}
\end{align*}
which proves that the action of $K$ is compatible with the side pairing maps. 

We now define $\Upsilon= \langle K \rangle$ and we are in the framework of the Poincaré polyhedron theorem for coset decompositions.
The theorem ensures that we need to check the tessellation only for one ridge per cycle (which we already knew) and for one ridge per orbit under the action of $K$. 
This means that we need to analyse only the ridges contained in $B(R_1)$, which are $F(R_1,R_1^{-1}), F(R_1,P), F(R_1,J), F(R_1,R_2^{-1})$ and $F(R_1, J^{-1})$, for which we already proved the tessellation property. 

We just need to check how the ridge cycles change with the new side pairing maps, so to give a presentation for these groups according to the theorem. 
The cycles for the ridges we mentioned are the following:
\begin{align*}
F&(R_1,R_1^{-1}) \xrightarrow{R_1} F(R_1,R_1^{-1}), \\
F&(R_1,P) \xrightarrow{R_1} F(P,R_1^{-1}) \xrightarrow{S_1^{-1}} F(R_1^{-1},J) 
\xrightarrow{R_1^{-1}} F(J, R_1) \xrightarrow{S_1} F(R_1,P), \\
F&(R_1,R_2^{-1}) \xrightarrow{R_1} F(P^{-1},R_1^{-1}) \xrightarrow{K} F(R_1,R_2^{-1}) \\
F&(R_1, J^{-1}) \xrightarrow{R_1} F(R_2,R_1^{-1}) \xrightarrow{K^{-1}} F(R_1, J^{-1}).
\end{align*}
Remark that we stop when we come back in the same cycle or when we arrive in the same ridge orbit under the action of $K$. 

The presentation obtained from Poincaré polyhedron theorem for coset decompositions is then
\[
\Gamma=\left\langle K,R_1 \colon
\begin{array}{l l}
R_1^p=K^4=(K^{-1}R_1)^{3d}=(KR_1)^3=K^2S_1^{-1}R_1=I,\\
 (K^2R_1)^2=(R_1K^2)^2
\end{array} \right\rangle.
\]

We want to remark that since $k=\frac{p}{2}$, by rewriting \eqref{ldef} or simply by inspection in Table \ref{tablelist}, we have that $l=d$. 
It is then not surprising that the relation in the presentation where $l$ appeared, here it becomes $(K^{-1})^{3d}=I$.

Finally, when $k=\frac{p}{2}$, we have $t=\frac{5}{p}-\frac{1}{2}$ and by applying this to the formula found in \eqref{vol}, we obtain 
\[
\chi (\hc / \Gamma)= \frac{2(p-5)}{p^2},
\]
which is 4 times the formula in the Livné case. 
This is what we expected since, as we already mentioned, these lattices are isomorphic to the corresponding ones of the form $(p, 2)$, which are the Livné ones and $D$ containes four copies of a fundamental domain for them.

\section{Previously known cases}

In this case we will show how to change the polyhedron according to the values of $p$ and $k$ so as to include all lattices with three fold symmetry listed in Section \ref{list}, including the cases previously treated.

\subsection{Degenerate cases}\label{degenerate}

The first thing to remark is that the parametrisation we chose in \eqref{singularities} is completely general and can be used to parametrise all possible lattices in our list when we impose $\theta= \frac{2\pi}{p}$ and $\phi= \frac{\pi}{k}$ as before.

In \cite{livne}, the same angle parametrisation holds after imposing $\phi=\frac{\pi}{2}$, since for all lattices of that group $k=2$.
In \cite{boadiparker}, this parametrisation has explicitly been used. 
Other cases on the list could be treated with an extra condition. 
The lattices of fourth type, for example, always have $\phi=\frac{\pi}{3}$.
All of the ones of type 5, instead, satisfy $\theta=\phi$ since $k=\frac{p}{2}$, as mentioned after interchanging $k$ and $l$ if necessary. 
This construction though includes all the other cases up to imposing the values of $p$ and $k$ that we want to consider. 

\begin{figure}
\centering
\includegraphics[width=1\textwidth]{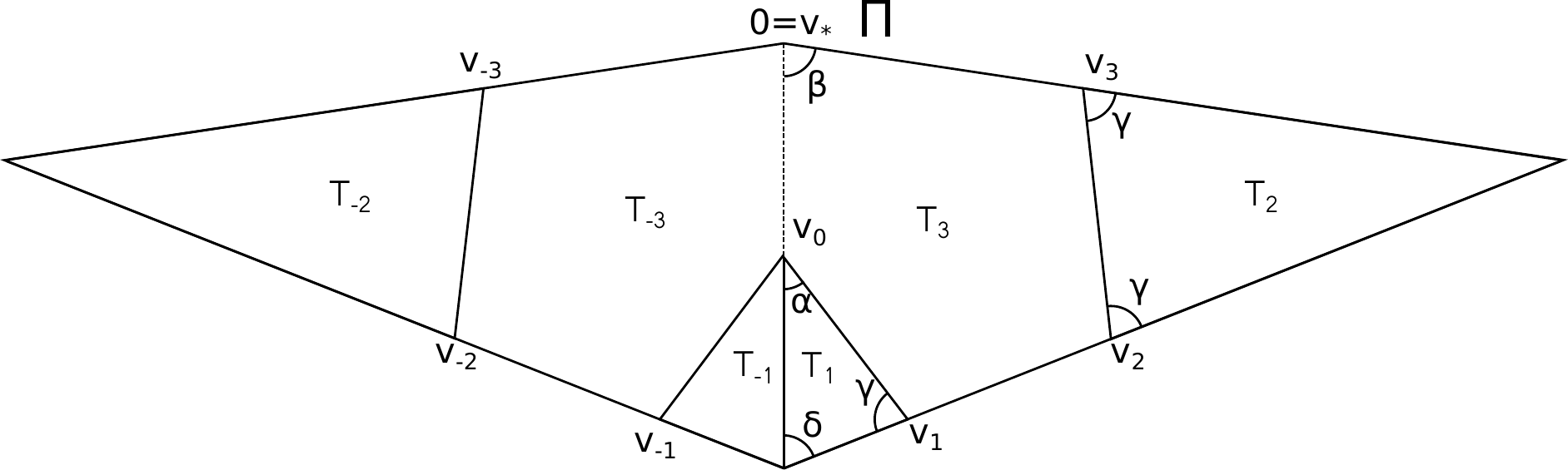}
\begin{quote}\caption{The angles whose values determines which polyhedron we shall consider. \label{angleclass}} \end{quote}
\end{figure}

The difference comes out when we start making the singularities collapse in order to find the vertices of the polyhedron. 
This is because when we make $T_1$ and $T_2$ shrink or enlarge, the vertices of $D$ change according to the size of the angles.
Let us consider a generic configuration as in Figure \ref{octagonreal}. 

The angles that we will have to consider are marked in Figure \ref{angleclass}. 
In particular, the vertices of the polyhedron will depend on the values of 
\begin{itemize}
\item The angle in $T_1$ at the vertex $v_0$, which we will call $\alpha$;
\item The angle in $T_3$ at $v_*$, which we will call $\beta$;
\item The two equal angles in $T_2$, which we will call $\gamma$;
\item The angle in $T_1$ at $v_1$, which by construction is equal to the angle $\gamma$ defined previously;
\item The third angle in $T_1$, which we will call $\delta$.
\end{itemize}

In this section we will explain the conditions on this angles to determine which are the vertices of our polyhedron.
Then we will substitute their values, that can be easily calculated in terms of $p$ and $k$. 

What we need to show is that, for the particular values we are considering, the vertices that we can obtain by making cone points collapse are the ones described in the theorem. 
Let us first consider the cases where $p>0$.

We have the following situation:
\begin{enumerate}
\item Vertices $\textbf z_1$ and $\textbf z_2$ are always possible and they do not depend on the angles at all. They will hence always be in the polyhedron.

\begin{figure}[!h]
\centering
\includegraphics[width=1\textwidth]{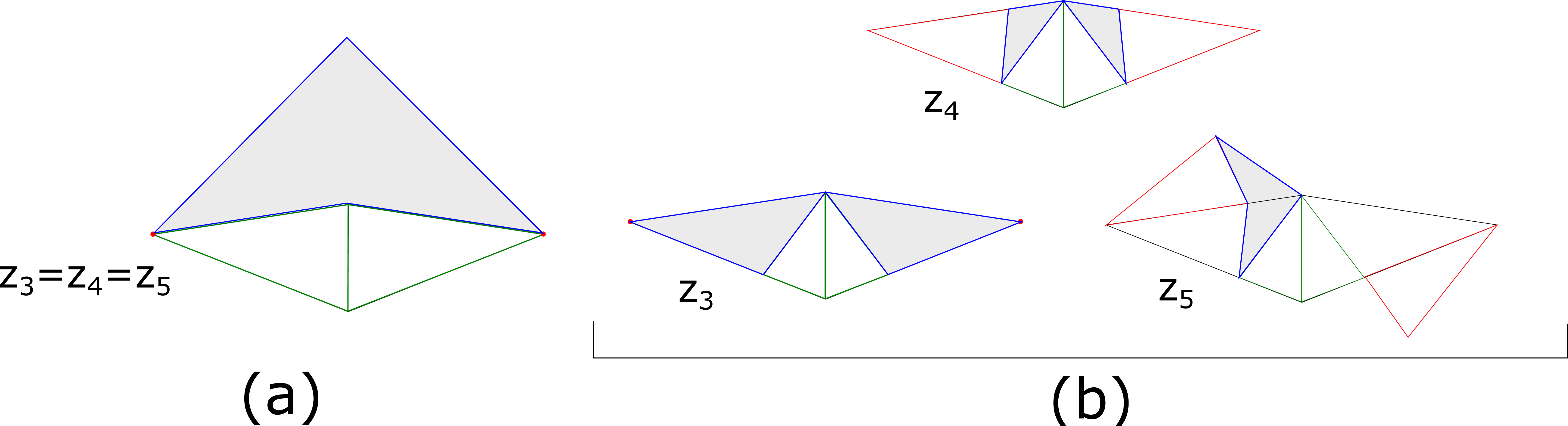}
\begin{quote}\caption{The two possibilities for the vertices in case 2. \label{345}} \end{quote}
\end{figure}

\item If we let $z_1$ be as big as possible, keeping it real and such that $T_1$ is in the interior of $T_3$, there are two possibilities, illustrated in Figure \ref{345}.
As the coordinate grows, either $v_1$ will coincide with the apex vertex of $T_2$, or $v_0$ will coalesce with $v_*$.

In the first case (a) there is no other possibility for $T_2$ but to collapse to a point, giving a single vertex defined by $v_1 \equiv v_2 \equiv v_3$. 
This is the case when $\beta \leq \alpha$.

In the second case (b) we have instead that $v_0 \equiv v_*$.
Also, $T_2$ has still some degrees of freedom, so we can make $z_2$ either to be 0, either to be as large as possible but still real, or to be as large as possible but after rotating it as in Figure \ref{345}. 
The three options give respectively that also $v_2 \equiv v_3$, $v_1 \equiv v_2$ or $v_1 \equiv v_3$.
This is the case when $\beta \geq \alpha$.

\begin{figure}[ht]
\centering
\includegraphics[width=1\textwidth]{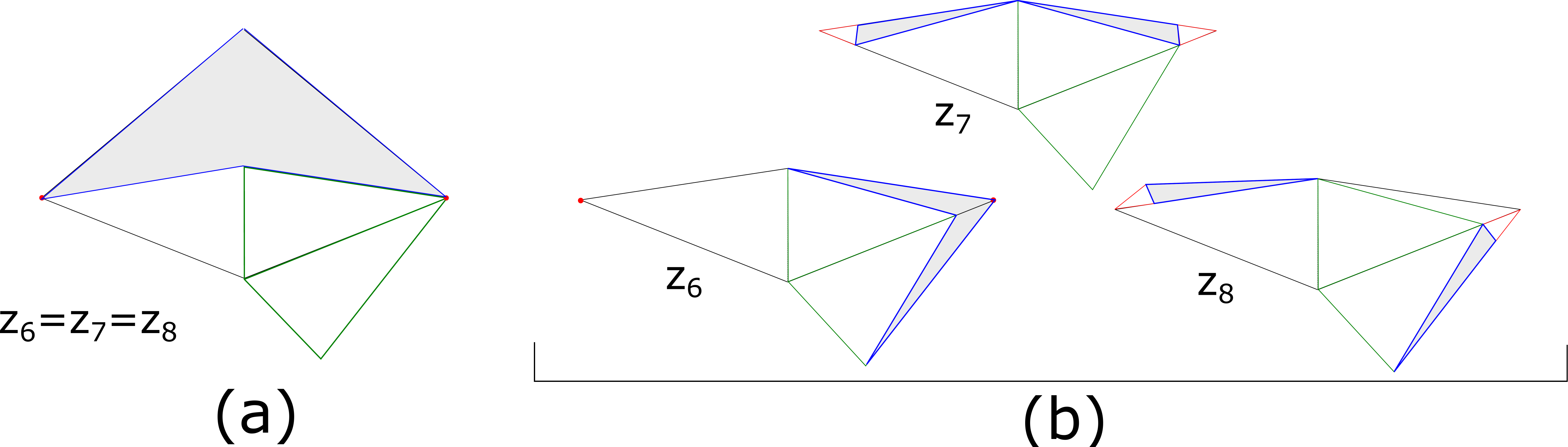}
\begin{quote}\caption{The two possibilities for the vertices in case 3. \label{678}} \end{quote}
\end{figure}

\item With a similar argument, by imposing $z_1= re^{-i \phi}$ with $r$ as big as possible, but such that $T_{-1}$ is inside $T_3$, we can get the two possibilities in Figure \ref{678}. 

Case (a) will correspond to when the cone points collapsing are  $v_0 \equiv v_2 \equiv v_3$ and it corresponds to the case when $\gamma \geq \beta$.

Case (b) is when we have $v_* \equiv v_{-1}$.
The three choices will be when also $v_2 \equiv v_3$, $v_0 \equiv v_2$ or $v_0 \equiv v_3$ and it occurs when $\gamma \leq \beta$.

\begin{figure}[!h]
\centering
\includegraphics[width=1\textwidth]{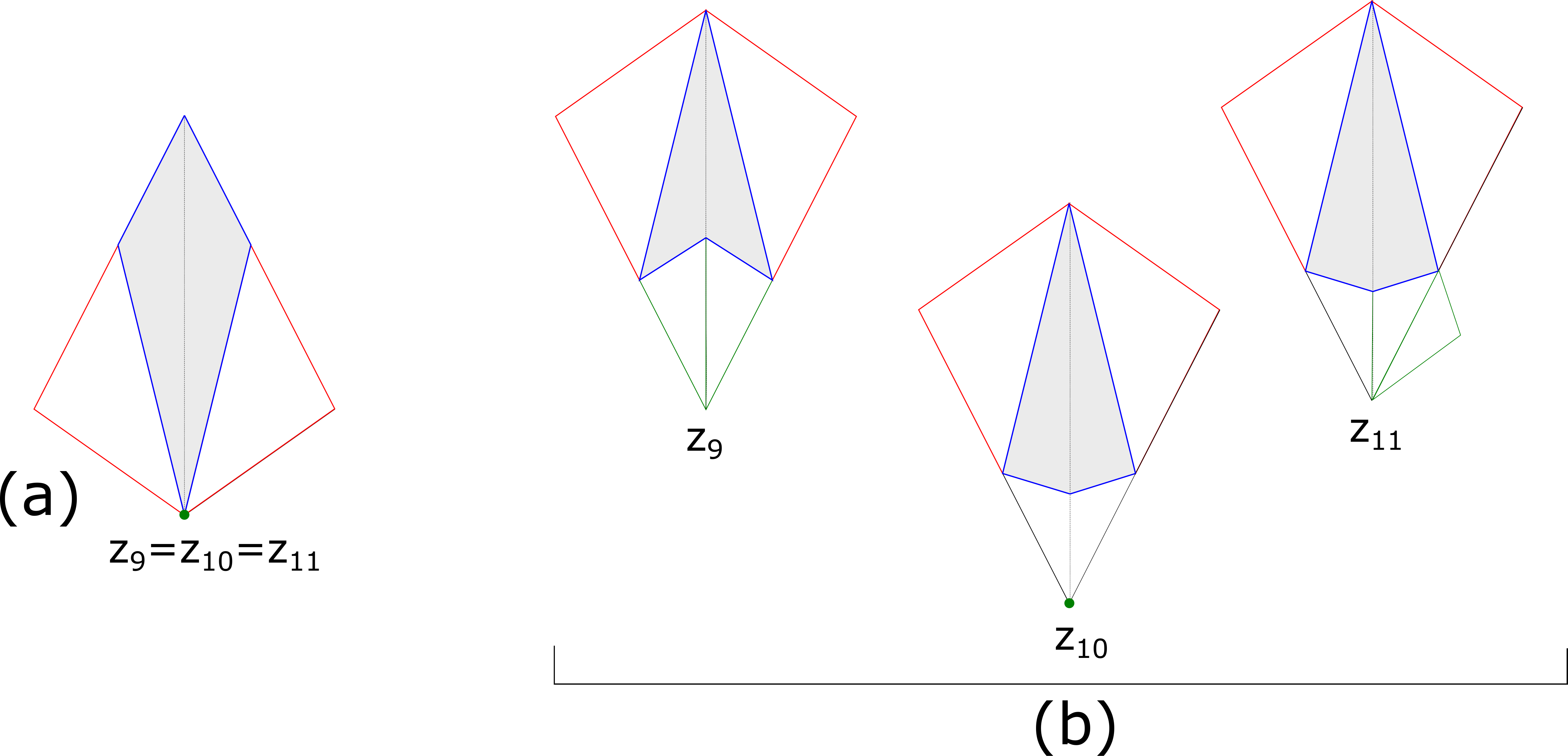}
\begin{quote}\caption{The two possibilities for the vertices in case 4. \label{91011}} \end{quote}
\end{figure}

\item Similarly, when $z_2$ is real, as big as possible and such that $T_2$ is inside $T_3$, we can get the configurations in Figure \ref{91011}. 

Case (a) occurs when $\gamma \leq \delta$ and the points will be $v_2 \equiv v_1 \equiv v_0$.

In Case (b) we always have the condition $v_* \equiv v_3$, with the three possibilities as $v_0 \equiv v_1$, $v_1 \equiv v_2$ or $v_0 \equiv v_2$.
This happens when $\gamma \geq \delta$.

\item Once more, when $z_2=r e^{i \theta}$, for $r$ as big as possible but still maintaining a positive area, we can have the configurations as in Figure \ref{121314}.

We will hence have Case (a), when $\delta \geq \beta$ and where $v_0 \equiv v_1 \equiv v_3$. 

When $\delta \leq \beta$ we will have Case (b) instead, with $v_* \equiv v_{-2}$ for all the three vertices and $v_0 \equiv v_1$, $v_1 \equiv v_3$ or $v_0 \equiv v_3$ in the each of them.
\end{enumerate}

\begin{figure}[!h]
\centering
\includegraphics[width=1\textwidth]{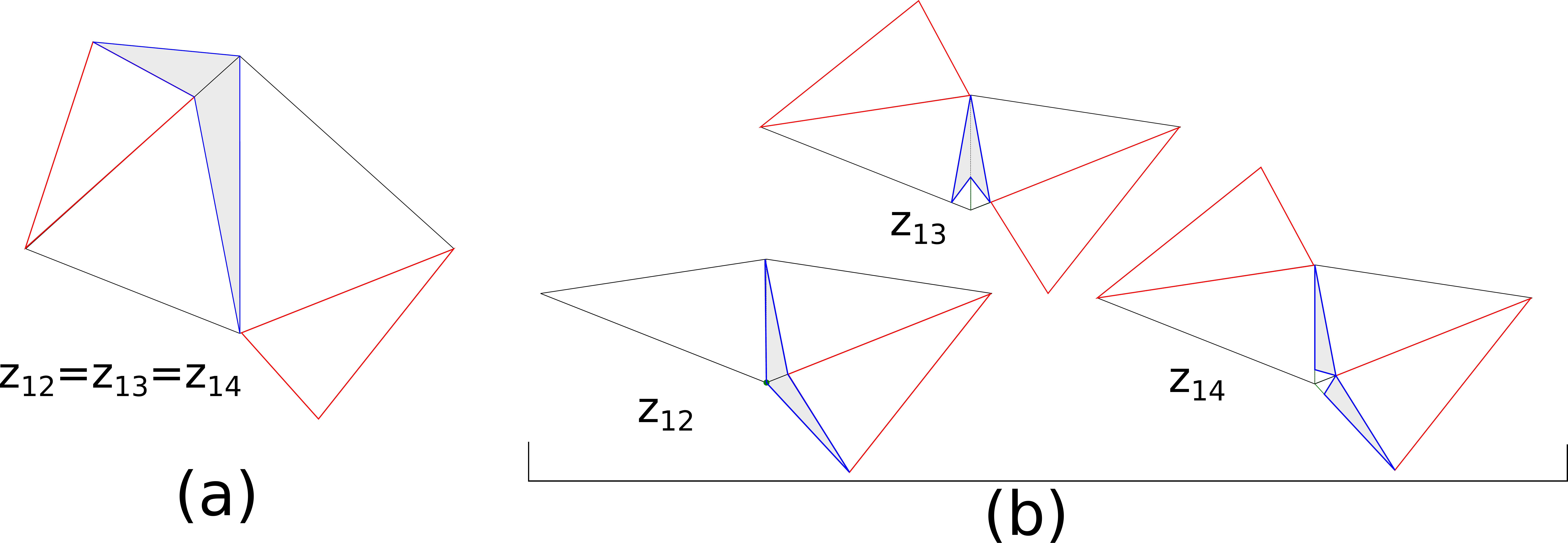}
\begin{quote}\caption{The two possibilities for the vertices in case 5. \label{121314}} \end{quote}
\end{figure}

It is clear that since in each case we have either one or three vertices, the cases with fewer vertices will be obtained by the case with more vertices by making triplets of vertices collapse to just one. 
On the other hand, the case with many vertices can be obtained from the other by cutting through a corner so to make one vertex become three. 
We will see in Section \ref{type2} that this is exactly the case, for the values of $p$ and $k$ that have already been treated. 

In Figure \ref{octagonreal} it is easy to see that 
\begin{align*}
\alpha &=\frac{\pi}{2}+\frac{\theta}{2}-\phi, & \beta &=\pi - \theta - \phi, &
\gamma &=\frac{\pi}{2}-\frac{\theta}{2}, & \delta &= \phi.
\end{align*}
Substituting the values of the angles in terms of $p$ and $k$, we can summarise the cases with the following table.

\begin{center}
\begin{tabular}{|cc|c|c|}
\hline
Case && Relation on the angles & Relation on $p$ and $k$ \\
\hline 
2 & (a) & $\beta \leq \alpha$ & $p \leq 6$ \\
 & (b) & $\beta \geq \alpha$ & $p \geq 6$ \\
\hline
3 & (a) & $\beta \leq \gamma$ & $k \leq \frac{2p}{p-2}$ \\
 & (b) & $\beta \geq \gamma$ & $k \geq \frac{2p}{p-2}$ \\
\hline
4 & (a) & $\gamma \geq \delta$ & $k \geq \frac{2p}{p-2}$ \\
 & (b) & $\gamma \leq \delta$ & $k \leq \frac{2p}{p-2}$ \\
\hline
5 & (a) & $\beta \leq \delta$ & $k \leq \frac{2p}{p-2}$ \\
 & (b) & $\beta \geq \delta$ & $k \geq \frac{2p}{p-2}$ \\
\hline
\end{tabular}
\end{center}
As we can see, three of these conditions correspond to the same values for $p$ and $k$, so we will either have all cases of the three vertices or all cases of a single vertex. 
Consequently, there are four possible cases and they are the four values of $p$ and $k$ given in the Theorem \ref{main}.

It is clear that the case of $D$ described in the previous section is the one where all 14 vertices remain distinct.
The other cases of the theorem follow immediately by our analysis.
In fact, we will have one case where only one triplet collapses, one case where three triplets collapse and one case where all four do. 
By considering the theorem and the figures to see which vertices are collapsing, we just need to consider that the name of the configurations given in Figures \ref{345}--\ref{121314} are the same as the ones given for $D$ in the previous sections.

We remark that when the angles we are considering are equal, while making the points collapse to get a vertex, we obtain some configurations with zero area, so on the boundary of the complex hyperbolic space. 
A more precise discussion of what happens in these cases can be found in \cite{livne} and \cite{boadiparker}.
Moreover, it is clear that we do not have the choice of the three configurations, so it is more natural to include them in the case of the lower values of the parameters as we did in Theorem \ref{main}.

Another way to see this is to notice that the cases where three vertices collapse correspond to when the values of $l$ and $d$ are negative.
We saw that these two values are the order of the cycle maps $R_2R_1J$ and $P^3$ respectively.
As explained in \cite{survey}, when $l$ or $d$ is negative, the corresponding map becomes a complex reflection in a point instead of a complex reflection in a line. 
The ridge on the mirror indeed becomes a single point.
When they are not finite, the corresponding map becomes a parabolic element.

\subsection{Relation with the previous construction for type 2}\label{type2}

In this section we will analyse the relation between this method and the previous fundamental polyhedra found for Deligne-Mostow lattices with three fold symmetry lattices.

\begin{figure}[!ht]
\centering
\includegraphics[width=0.7\textwidth]{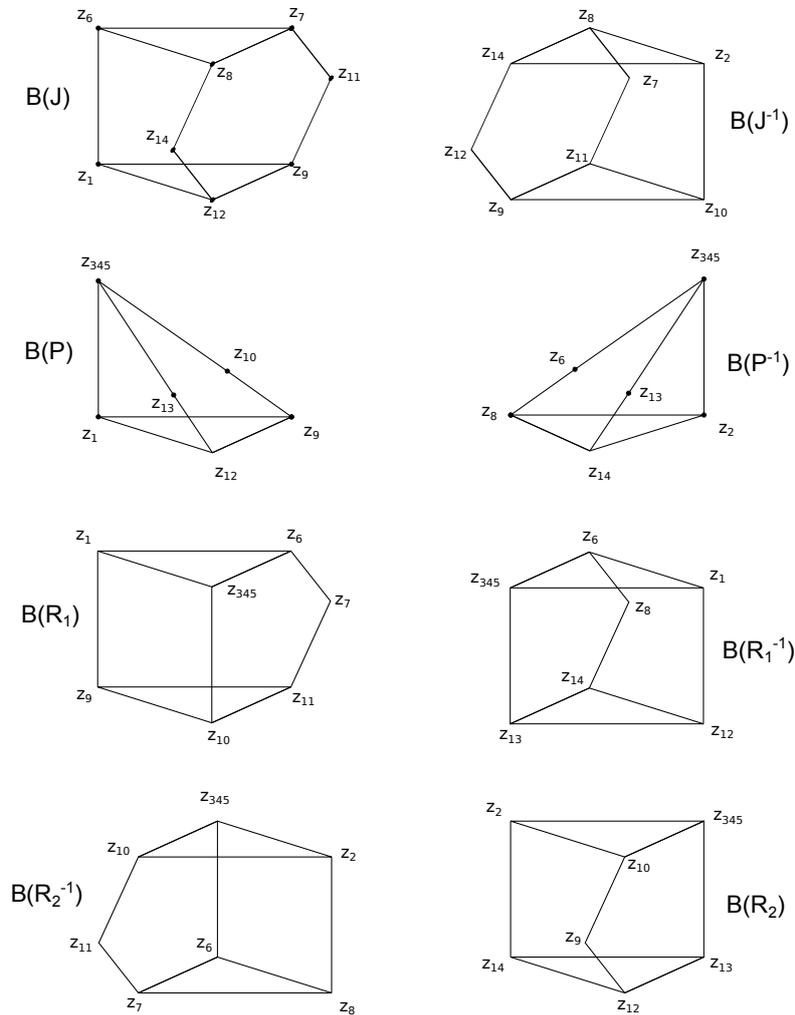}
\begin{quote}\caption{The sides and the side pairing maps compared for our polyhedron and the previous one for type 2 lattices. \label{sidest2}} \end{quote}
\end{figure}

For the cases analysed in \cite{boadiparker} and \cite{livne} our construction follows step by step the one used there. 
Already in \cite{survey} it has been explained that the fundamental polyhedron for type 1 can be obtained from the one of type 3 by truncating a vertex with a triangle contained in a complex line.
In that case, one vertex becomes three and we will see that it corresponds to the case (a) and (b) in point 2 of our analysis of the vertices. 
Comparing the sides for these cases and the ones for ours it is easy to see that the same thing can be done from our polyhedron.

For type 2, a construction was already found in \cite{type2}.
Since the approach there is a bit different from ours, Parker in \cite{survey} already showed how to see in their procedure an approach similar to ours. 
What we do here though, gives a different presentation for the group and an easier construction of the polyhedron, more coherent with the known construction for the other cases. 

The main difference comes from the fact that the sides and the side-pairing maps considered there are slightly different from ours. 
We now want to explain how to reconcile the two presentations. 
First of all, for the case we are talking about we need to make the vertices $\textbf{z}_3, \textbf{z}_4$ and $\textbf{z}_5$ collapse to a single vertex as we saw in Theorem \ref{main} and we will call this new vertex $\textbf z_{345}$.
The sides of the polyhedron $D$ after collapsing sides as described for the second case of our main theorem, will be as in Figure \ref{sidest2}.

We want now to compare our construction with the sides of the polyhedron considered in \cite{type2} as shown in Figure 11 of \cite{survey}. 
To refer to sides in our construction, we will use $B(T)$, for $T \in \{J^{\pm 1},P^{\pm 1},R_1^{\pm 1},R_2^{\pm 1}\}$, while for the sides used before we will be coherent with their notation and call them $S(T)$, for $T \in \{J^{\pm 1},P_1^{\pm 1},P_2^{\pm 1},R_1^{\pm 1},R_2^{\pm 1}\}$.

The map $J$ considered in each case coincides and so do the sides $B(J)=S(J)$ and the sides $B(J^{-1})=S(J^{-1})$.
The same thing is true for $P=P_1=R_1R_2$ and the corresponding sides. 
On the other hand, the four sides $B(R_1^{\pm 1})$ and $B(R_2^{\pm 1})$ and the side pairing $R_1$ and $R_2$ include in their action the six remaining sides $S(R_1^{\pm 1})$, $S(R_2^{\pm 1})$ and $S(P_2^{\pm 1})$. 
In fact, the previous procedure splits the sides $B(R_1)$ and $B(R_1^{-1})$ in two blocks each, by cutting along a line through vertices $\textbf{z}_9,\textbf{z}_{11},\textbf{z}_{345}$ and a line through $\textbf{z}_{12},\textbf{z}_{14},\textbf{z}_{345}$ respectively.
Then, for each of $B(R_1)$ and $B(R_1^{-1})$, of the two pieces of side obtained we consider the one not containing vertex $\textbf z_{10}$ and vertex $\textbf z_{13}$ respectively.
These are exactly the sides $S(R_1)$ and $S(R_1^{-1})$, and $R_1$ sends the first to the latter. 
Similarly, for $B(R_2)$ and $B(R_2^{-2})$, we divide the sides in two blocks by cutting with a line through $\textbf{z}_{12},\textbf{z}_{14},\textbf{z}_{345}$ and a line through $\textbf{z}_{7},\textbf{z}_{8},\textbf{z}_{345}$ respectively. 
We then consider the block not containing vertex $\textbf z_{13}$ and $\textbf z_6$ respectively and these are sides $S(R_2)$ and $S(R_2^{-1})$ respectively, the first sent to the second by $R_2$. 

We have then four more block to consider. 
The first remark is that there are, in fact, only three blocks, because the parts of $B(R_1^{-1})$ and of $B(R_2)$ containing vertex $\textbf z_{13}$ are the same block.
For simplicity, we will call it $S(T)$. 
The other two blocks are exactly sides $S(P_2)$ and $S(P_2^{-1})$.  
We also know by our construction that $R_1$ sends $S(P_2)$ to $S(T)$, while $R_2$ sends $S(T)$ to $S(P_2^{-1})$. 
Since $P_2=R_2R_1$ by definition, that is the side pairing map that sends the two new blocks $S(P_2)$ to $S(P_2^{-1})$, as described in \cite{survey}. 
This is illustrated in Figure \ref{splitsides}.

\begin{figure}[t]
\centering
\includegraphics[width=1\textwidth]{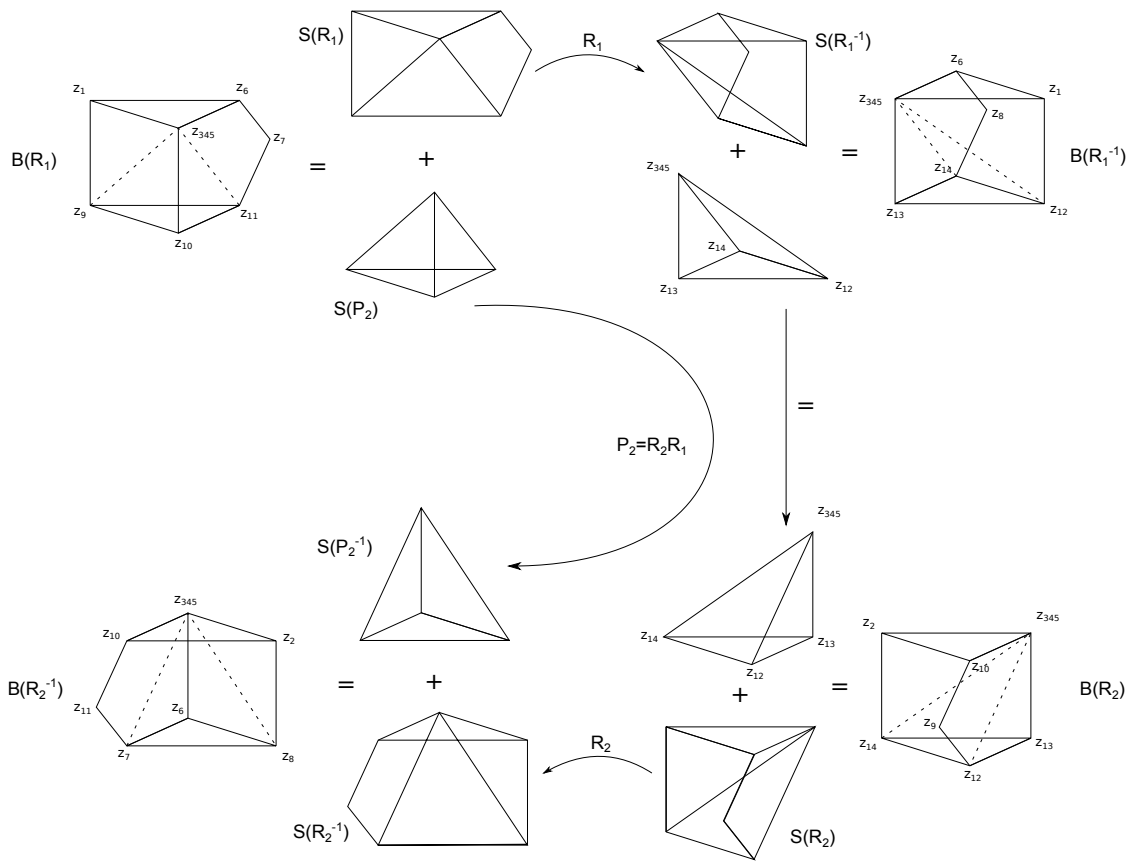}
\begin{quote}\caption{The sides and the side pairing maps compared for our polyhedron and the previous one. \label{splitsides}} \end{quote}
\end{figure}

\clearpage
\addcontentsline{toc}{section}{\refname}
\bibliographystyle{alpha}
\bibliography{biblio}

\begin{thebibliography}{DPP16}

\bibitem[BP15]{boadiparker}
Richard~K. Boadi and John~R. Parker.
\newblock Mostow's lattices and cone metrics on the sphere.
\newblock {\em Adv. Geom.}, 15(1):27--53, 2015.

\bibitem[DFP05]{type2}
Martin Deraux, Elisha Falbel, and Julien Paupert.
\newblock New constructions of fundamental polyhedra in complex hyperbolic
  space.
\newblock {\em Acta Math.}, 194(2):155--201, 2005.

\bibitem[DM86]{delignemostow}
P.~Deligne and G.~D. Mostow.
\newblock Monodromy of hypergeometric functions and nonlattice integral
  monodromy.
\newblock {\em Inst. Hautes \'Etudes Sci. Publ. Math.}, (63):5--89, 1986.

\bibitem[DPP16]{nonarithm}
Martin Deraux, John~R. Parker, and Julien Paupert.
\newblock New non-arithmetic complex hyperbolic lattices.
\newblock {\em Invent. Math.}, 203(3):681--771, 2016.

\bibitem[FP06]{eisensteinpicard}
Elisha Falbel and John~R. Parker.
\newblock The geometry of the {E}isenstein-{P}icard modular group.
\newblock {\em Duke Math. J.}, 131(2):249--289, 2006.

\bibitem[Gol99]{goldman}
William~M. Goldman.
\newblock {\em Complex hyperbolic geometry}.
\newblock Oxford Mathematical Monographs. The Clarendon Press, Oxford
  University Press, New York, 1999.
\newblock Oxford Science Publications.

\bibitem[Mos80]{mostow}
G.~D. Mostow.
\newblock On a remarkable class of polyhedra in complex hyperbolic space.
\newblock {\em Pacific J. Math.}, 86(1):171--276, 1980.

\bibitem[Mos86]{mostow2}
G.~D. Mostow.
\newblock Generalized {P}icard lattices arising from half-integral conditions.
\newblock {\em Inst. Hautes \'Etudes Sci. Publ. Math.}, (63):91--106, 1986.

\bibitem[Mos88]{mostow3}
G.~D. Mostow.
\newblock On discontinuous action of monodromy groups on the complex
  {$n$}-ball.
\newblock {\em J. Amer. Math. Soc.}, 1(3):555--586, 1988.

\bibitem[Par06]{livne}
John~R. Parker.
\newblock Cone metrics on the sphere and {L}ivn\'e's lattices.
\newblock {\em Acta Math.}, 196(1):1--64, 2006.

\bibitem[Par09]{survey}
John~R. Parker.
\newblock Complex hyperbolic lattices.
\newblock In {\em Discrete groups and geometric structures}, volume 501 of {\em
  Contemp. Math.}, pages 1--42. Amer. Math. Soc., Providence, RI, 2009.

\bibitem[Sau90]{sauter}
John~Kurt Sauter, Jr.
\newblock Isomorphisms among monodromy groups and applications to lattices in
  {${\rm PU}(1,2)$}.
\newblock {\em Pacific J. Math.}, 146(2):331--384, 1990.

\bibitem[Thu98]{thurston}
William~P. Thurston.
\newblock Shapes of polyhedra and triangulations of the sphere.
\newblock In {\em The {E}pstein birthday schrift}, volume~1 of {\em Geom.
  Topol. Monogr.}, pages 511--549. Geom. Topol. Publ., Coventry, 1998.

\end{thebibliography}
\nocite{*}

\end{document}